\documentclass[a4paper,11pt]{amsart}
\usepackage[latin1]{inputenc}
\usepackage[T1]{fontenc}
\usepackage{bbm,geometry}
\usepackage{tabularx}
\usepackage{amsmath}
\usepackage{amsthm}
\usepackage{latexsym}
\usepackage{amssymb}
\usepackage{xcolor}
\usepackage[all]{xy}
\usepackage{calc}
\usepackage{multicol}
\usepackage{graphicx,subfigure,color}
\usepackage{leqno}
\usepackage{hyperref}
\usepackage{bigints}
\usepackage{dsfont}

\numberwithin{equation}{section}
\newtheorem{thm}{Theorem}[section]
\newtheorem{prop}[thm]{Proposition}

\newtheorem{lem}[thm]{Lemma}

\newtheorem{conj}[thm]{Conjecture}

\newtheorem{rmq}[thm]{Remark}

\newcommand{\R}{\mathbb{R}}
\numberwithin{equation}{section}

\newcommand{\be}{\begin{equation}}
\newcommand{\ee}{\end{equation}}

\newcommand{\e}{\varepsilon}

\newcounter{exercice}

\def\Bcal{\mathcal{B}}

\title[Radial blow-up for the Lin-Ni-Takagi problem]{Classification of radial blow-up at the first critical exponent for the Lin-Ni-Takagi problem in the ball}

\begin{document}

\thanks{B. P. was supported by the FNRS CdR grant J.0135.19 and the Fonds Th\'elam. D.B. and B.P. are supported by an ARC Avanc\'e 2020 at ULB. J.-B.C was supported by FCT - Funda\c c\~ao para a Ci\^encia e a Tecnologia, under the project: UIDB/04561/2020. D.B. is Francqui Research Professor 2021-2024. He thanks the Fondation Francqui for the support.}

\author[Denis Bonheure]{Denis Bonheure}
\address{D\'epartement de Math\'ematique, Universit\'e libre de Bruxelles, Campus de la Plaine CP
213, Bd. du Triomphe, 1050 Bruxelles, Belgium}
\email{denis.bonheure@ulb.be}
\email{Bruno.Premoselli@ulb.be}
\author[Jean-Baptiste Casteras]{Jean-Baptiste Casteras}
\address{CMAFcIO, Faculdade de Ci\^encias da Universidade de Lisboa, Edificio C6, Piso 1, Campo Grande 1749-016 Lisboa, Portugal}
\email{jeanbaptiste.casteras@gmail.com}
\author[Bruno Premoselli]{Bruno Premoselli}

\begin{abstract}
We investigate the behaviour of radial solutions to the Lin-Ni-Takagi problem in the ball $B_R \subset \R^N$ for $N \ge 3$:
\begin{equation*}
\left \{ \begin{aligned}
- \triangle u_p + u_p & = |u_p|^{p-2}u_p & \textrm{ in } B_R, \\
\partial_\nu u_p & = 0 & \textrm{ on } \partial B_R,
\end{aligned} \right.
\end{equation*}
 when $p $ is close to the first critical Sobolev exponent $2^* = \frac{2N}{N-2}$. We obtain a complete classification of finite energy radial smooth blowing up solutions to this problem. We describe the conditions preventing blow-up as $p \to 2^*$, we give the necessary conditions in order for blow-up to occur and we establish their sharpness by constructing examples of blowing up sequences. Our approach allows for asymptotically supercritical values of $p$. We show in particular that, if $p \geq 2^\ast$, finite-energy radial solutions are precompact in $C^2(\overline{B_R})$ provided that $N\geq 7$. 
Sufficient conditions are also given in smaller dimensions if $p=2^\ast$. 
Finally we compare and interpret our results to the bifurcation analysis of Bonheure, Grumiau and Troestler in Nonlinear Anal. 147 (2016). 
\end{abstract}

\maketitle

\section{Introduction}

\subsection{Statement of the results}

Let $R>0$ and $B_R=B(0,R)$ be the ball in $\R^N$, $N \ge 3$. We investigate the behaviour of families $(u_p)_{p>2}$ of radial solutions to the so-called Lin-Ni-Takagi problem in $B_R$:
\begin{equation} \label{LNT}
\left \{ \begin{aligned}
- \triangle u_p + u_p & = |u_p|^{p-2}u_p & \textrm{ in } B_R \\
\partial_\nu u_p & = 0 & \textrm{ in } \partial B_R.
\end{aligned} \right.
\end{equation} 
This equation can be derived from the Patlak-Keller-Segel system
\begin{equation}\label{kssys}
\begin{cases}
\dfrac{\partial u}{\partial t}=\Delta u - D_1\nabla \cdot (u\nabla \phi(v)) \quad &\text{in } B_R ,\\[4mm]  
\dfrac{\partial v}{\partial t}=D_2 \Delta v-D_3 v+D_4 u \quad &\text{in } B_R ,\\[2mm]
\partial_\nu u=\partial_\nu v=0 \quad &\text{on } \partial B_R, \\
u,v>0 &\text{in } B_R ,
\end{cases}
\end{equation}
where $D_i$, $i=1,\ldots ,4$ are positive constants and $\phi$ is a smooth strictly increasing function. This system has been introduced in \cite{KellerSegel} to model chemotaxis, a situation where organisms move towards areas of high concentration of the chemicals they secrete. In this model $u(x,t)$ represents the concentration of the considered organisms and $v(x,t)$ the one of the chemical released. A huge literature exists on the dynamics of this system, concerning several fondamental questions such as the Cauchy theory, blow-up in finite or infinite time and stability to cite a few of them. Since it is not our purpose to analyse the time-depending solutions, we just refer to some recent papers \cite{MR3936129,arxiv.1911.12417,arxiv.1912.00721,MR4179771} and the many citations therein to trace the most important results in this direction. A large variety of highly nontrivial equilibria have been detected since the seminal work \cite{MR929196}, see for instance  
\cite{MR1923818,MR2124160,MR2296306,DruetRobertWei,MR3073256,MR3262455,MR3304582,MR3564729,MR3527634,MR3487266,MR3625083,MR3641921,MR4299895}. This list of references is far from being exhaustive. 
One can prove that stationary positive solutions satisfy
$$\nabla\cdot (u \nabla (\log u - D_1 \phi (v))=0, $$
which, together with the boundary conditions, implies that $u=C e^{D_1 \phi (v)}$ for some positive constant $C$. Taking $\phi (v)=\ln v$ one recovers \eqref{LNT} for a suitable choice of parameters $D_i$ and $C$. The Lin-Ni-Takagi problem has been intensively studied these last decades. We refer to \cite{MR929196,MR974610,MR1106295,MR1103293,MR1480241,ReyWei,MR2114411,MR2645043,MR2806101,DruetRobertWei} for a non-exhaustive list of existence and non-existence results for \eqref{LNT}, related to the so-called Lin-Ni conjecture, and mostly concerning the close-to-critical case. 
As mentioned by several experts of the community, the role of the many highlighted equilibria, as actually the role of any stationary solution, in the dynamics of the parabolic problem \eqref{kssys}, seems unclear up to now. However, it is to expect that loss of compactness by bubbling in the stationary frame plays a role in the dynamics of \eqref{kssys}. This issue has been investigated in other time-depending equations, as for instance in \cite{MR2431658,MR2494455,MR4054257,MR4046015,MR4108414,MR4307216,MR4157676} and the included citations.

\medskip

We will be interested in this work in the blow-up behaviour of radial solutions of \eqref{LNT} at the first critical exponent $p = 2^*$, where we have let $2^* = \frac{2N}{N-2}$. Our main goal in this paper is to provide a complete description of the radial blow-up behaviour of \eqref{LNT} when $p \to 2^*$. This means giving sufficient conditions preventing the blow-up; describing \emph{a priori} necessary condition in order for blow-up to occur; and constructing, when the necessary conditions are satisfied, radial blowing up sequences of solutions of \eqref{LNT}. Our investigations are motivated, among other things, by the bifurcation analysis in \cite{MR3564729}, for which we refer to Section \ref{motivation} below. 

\medskip

For simplicity we will work, in the following, up to a subsequence: we let $(p_k)_{k\ge0}$ be a sequence of positive real numbers converging towards $2^*$ as $k \to + \infty$ and we let $(u_k)_{k \ge0}$ be a sequence of \emph{possibly sign-changing} radial solutions of the following equation:
\begin{equation} \label{LNTk}
\left \{ \begin{aligned}
- \triangle u_k + u_k & = |u_k|^{p_k-2}u_k & \textrm{ in } B_R, \\
\partial_\nu u_k & = 0 & \textrm{ in } \partial B_R,
\end{aligned} \right.
\end{equation} 
which is just \eqref{LNT} with $p = p_k$. 
We will always work with solutions that are \emph{energy-bounded}  in the following sense: we assume that there exists $C >0$ such that 
\begin{equation} \label{energy}
\Vert u_k \Vert_{L^{\alpha_k}(B_R)} \le C, 
\end{equation} 
for all $k \ge 0$, where 
\[ \alpha_k = \max \big( 2^*, \frac{N}{2}p_k - N \big). \]
If $p_k \le 2^*$, then $\alpha_k = 2^*$ and \eqref{energy} implies that each $u_k$ is smooth by the classical argument of \cite{Trudinger} (see also \cite{BrezisKato}). 
Observe that in this case \eqref{energy} is equivalent to a $H^1(B_R)$ bound on the sequence $(u_k)_{k \ge0}$. When $p_k>2^*$ we have $\alpha_k > p_k$ and thus assumption \eqref{energy} is stronger than a uniform bound on the $H^1(B_R)$ norm. The exponent $\alpha_k$ appears however naturally since it defines the precise norm which is preserved by the natural scaling that is formally associated to \eqref{LNT} : given $p>2$, the $L^\alpha(\R^N)$ norm of $\lambda^{\frac{2}{p-2}}u(\lambda x)$ is preserved if and only if $\alpha =\frac{N}{2}(p-2)$. Such an $\alpha$ is also the exact exponent for which a natural bootstrap procedure in \eqref{LNT} does not improve the integrability of a solution. Observe also that when $p_k>2^*$ in \eqref{LNTk}, the smoothness of the solutions is not guaranteed solely by \eqref{energy} anymore. We thus assume in addition, if $p_k > 2^*$, that $u_k$ is smooth, by which we mean $C^2(\overline{B_R})$,  for any $k \ge 0$. Smooth radial solutions of \eqref{LNT} for all $p$ have for instance been constructed in \cite{MR2124160,MR2765510,MR2948289,MR3056708,MR3564729}. Singular solutions have been constructed in \cite{MR3408072,MR4149344} in the supercritical case.

\medskip

We say that a sequence of smooth, possibly sign-changing, radial solutions $(u_k)_{\ge0}$ of \eqref{LNTk} blows up if
$$\lim_{k \to + \infty} \Vert u_k \Vert_{L^\infty(B_R)} = + \infty. $$
Since we deal with energy bounded sequences, Strauss lemma \cite{Strauss} equivalently implies 
$$\lim_{k \to + \infty} |u_k(0)| = + \infty. $$
Standard elliptic theory and \eqref{LNTk} show that if $(u_k)_{k \ge0}$ does not blow up, then it is precompact in the $C^2(\overline{B_R})$ topology. Our first result is a precompactness theorem that states that if the $p_k$'s approach the critical exponent $2^*$ from above, then an energy-bounded sequence of solutions $(u_k)_k$ of \eqref{LNTk} never blows up in dimensions $N \ge 7$. 

\begin{thm}[Critical or supercritical case in dimensions $N\ge 7$] \label{thstabi-Nge7}\mbox{ } \\
Assume $N \ge 7$ and let $(p_k)_{k \ge0}$ be a sequence of positive real numbers such that $p_k \ge 2^*$ for all $k \ge0$ and $\lim_{k \to + \infty} p_k = 2^*$.
Any sequence $(u_k)_{k \ge0}$ of $C^2(\overline{B_R})$ radial solutions of \eqref{LNTk} satisfying \eqref{energy} is precompact in $C^2(\overline{B_R})$.
\end{thm}

By standard elliptic theory, Theorem \ref{thstabi-Nge7} implies that $(u_k)_{k\ge0}$ strongly converges, up to a subsequence, in $C^2(\overline{B_R})$ to a solution of \eqref{LNT} with $p=2^*$. 
By the result of \cite{Ni83}, radial solutions of \eqref{LNTkr} when $p_k \ge 2^*$ have fixed sign (otherwise there would be a Dirichlet solution on some smaller ball and this is impossible in a star-shaped domain). In lower dimensions, we can cover the exactly critical case $p_k = 2^*$ for all $k\ge 0$. 

\begin{thm}[Exactly critical case]
 \label{thstabi2etoile}\mbox{ } \\
Let
$p_k=2^*$ for all $k\ge 0$.
Assume that $(u_k)_{k \ge0}$ is a sequence of radial solutions of \eqref{LNTk} that is uniformly bounded in $L^{2^*}(B_R)$. 
\begin{itemize}
\item If $N \not \in \{3, 6\}$, then $(u_k)_{k \ge0}$ is precompact in $C^2(\overline{B_R})$. 
\item If $N=3$, there exists an exceptional radius $R^* >0$ such that if $R\ne R^*$, then $(u_k)_{k \ge0}$ is precompact in $C^2(\overline{B_R})$. 
\item If $N=6$ there exists as increasing sequence of exceptional radii $(R_\ell)_{\ell\ge 1}$, with $R_\ell \to + \infty$ and $\liminf_{\ell \to + \infty} (R_\ell-R_{\ell-1}) >0$, such that if $R \not \in \{R_\ell\}_{\ell\ge 1}$ then $(u_k)_{k \ge0}$ is precompact in $C^2(\overline{B_R})$.
\end{itemize}
\end{thm}
The value of the exceptional radius $R^*$ in dimension $N=3$ is described in Section \ref{sec:classi}. We believe that precompactness also holds for $R=R^*$ but it would require an additional and tedious analysis to prove it, see Remark \ref{Rem:N=3crit}.
In dimension $N=6$, the values of $(R_\ell)_{\ell\ge 1}$ are also univocally determined. We refer again to the proof of Theorem \ref{thstabi2etoile} in Section \ref{sec:classi}. 
Under an additional assumption which seems numerically true, we are able to fully prove precompactness, see Lemma \ref{lemme6} below and Remark \ref{Rem:N=6crit} that follows. 
This almost completely settles the exactly critical case $p_k = 2^*$ in \eqref{LNTk} and leads to the following conjecture :
\begin{conj}[Exactly critical case]
 \label{conjstabi2etoile}\mbox{ } \\
 Any uniformly (energy) bounded sequence $(u_k)_{k \ge0}$ of radial solutions of \eqref{LNTk} with $p_k=2^*$ for all $k\ge 0$, is precompact in $C^2(\overline{B_R})$. 
\end{conj}
The missing cases have just been mentioned : $N=3$ and $R=R^*$ in which case the conjecture is supported by Remark \ref{Rem:N=3crit} and $N=6$ and $R=R_\ell$ in which cases the conjecture is supported by Lemma \ref{lemme6} below and Remark \ref{Rem:N=6crit}.

\medskip

Theorems \ref{thstabi-Nge7} and \ref{thstabi2etoile}  are the main precompactness results of this paper. They do not just follow from classical precompactness results for (possibly sign-changing) solutions of equations similar to \eqref{LNT}, such as \cite{DevillanovaSolimini}.
There are two main novelties in our approach here: first, we include \emph{sign-changing} solutions in our analysis, and second, we allow for asymptotically \emph{supercritical} perturbations in \eqref{LNTk}. To our knowledge this seems to be the first \emph{a priori} blow-up analysis for asymptotically supercritical solutions of equations like \eqref{LNT}. We prove Theorems \ref{thstabi-Nge7} and \ref{thstabi2etoile} by contradiction. Assuming that blow-up occurs we prove a sharp pointwise asymptotic description of blowing up sequences of solutions (see Proposition \ref{theorieC0} below). The \emph{radial symmetry} forces blow-up to occur in the form of finite sums of positive bubbles at the origin, with possibly alternating signs (Proposition \ref{theorieC0} deals also with the subcritical case and sign-changing solutions exist in that case). This pointwise description is then used to rule out blow-up by means of Pohozaev identities. The ideas that we use in the proofs have been recently developed to prove compactness results for finite-energy sign-changing solutions of critical elliptic equations: we refer for instance to \cite{GhoussoubMazumdarRobert, PremoselliVetois3, PremoselliVetois2}, and to \cite{PremoselliVetois} for compactness results without bound on the energy. 
 
\medskip 

The techniques that we use to prove Theorems \ref{thstabi-Nge7} and \ref{thstabi2etoile} yield more generally necessary conditions for blow-up for positive solutions. We start with the subcritical case:
\begin{thm}[Blow-up in the subcritical case] \label{thblowup-intro-critandsubcrit}\mbox{ }\\
Assume that $p_k \le2^*$ for all $k\ge0$, $\lim_{k \to +\infty} p_k = 2^*$  and let $(u_k)_{k \ge0}$ be a sequence of \emph{positive radial} solutions of \eqref{LNTk} satisfying \eqref{energy} and denote by $u_0$ its weak limit.
If there is blow-up, then
\begin{itemize}
\item $u_0 \equiv 0$ for $3 \le N \le 5$ and $R\ge R^*$ when $N =3$,
\item $u_0(0) \le \frac12$ when $N=6$. 
\end{itemize}
When $3 \le N \le 6$, the blow-up occurs with a single positive bubble at $0$.
\end{thm}
The value of the radius $R^*$ in dimension $N=3$ is as in Theorem \ref{thstabi2etoile}. We refer to Theorem \ref{thblowup} for the necessary mass condition explaining the role or $R^*$. As a consequence of this result, a tower of bubbles, if it exist in the subcritical case, can arise only in dimension $N\ge 7$. 

Theorem \ref{thstabi2etoile} and Conjecture \ref{conjstabi2etoile} state that a purely critical blow-up is not possible. The missing cases are in dimension $3$ and $6$. In these dimensions the direct counterpart of Theorem \ref{thblowup-intro-critandsubcrit} implies that if there is blow-up, then
$u_0 \equiv0$ if $N=3$ and $R=R^*$, while $u_0(0) = \frac12$ if $N=6$, forcing $R$ to be in the set of critical points $\{R_\ell\}_{\ell\ge 1}$ of $u_0$ (those are fixed by the initial condition $u_0(0) = \frac12$, remembering that $u_0'(0)=0$). 

\medskip

We finally cover the supercritical case $p_k > 2^*$ for all $k \ge 0$. Theorem \ref{thstabi-Nge7} states that blow-up cannot happen in this case in dimension $N\ge 7$. In dimensions $3 \le N \le6$, a blow-up might occur, but we show that, at least when $N=4,5,6$, the weak limit of a bubbling sequence is strictly positive: 

\begin{thm} [blow-up in the supercritical case in dimension $3 \le N\le 6$] \label{thblowup-intro-supercritical}\mbox{ }\\
Assume $4 \le N \le 6$. Let $p_k \ge 2^*$ for all $k\ge0$ with $\lim_{k \to + \infty} p_k = 2^*$, $(u_k)_{k \ge0}$ be a sequence of \emph{positive radial} $C^2(\overline{B_R})$ solutions of \eqref{LNTk} satisfying \eqref{energy} and denote by $u_0$ its weak limit.  
If there is blow-up then $u_0 >0$ and $u_0(0) \ge 1/2$ if $N=6$.
If $N=3$ and $u_0=0$, then $R\le R^*$. 
\end{thm}

When $N=3$ the weak limit need not be positive, as the examples in \cite{ReyWei} show, see Remark \ref{remN3} below.

\medskip

It is natural to investigate the sharpness of the conclusion of Theorems \ref{thstabi-Nge7} to \ref{thblowup-intro-supercritical}: that is, to understand if blowing up sequences of solutions do exist when the necessary conditions of Theorems \ref{thblowup-intro-critandsubcrit} and \ref{thblowup-intro-supercritical} are fulfilled. We pursue this aims for positive solutions. We say that a sequence of positive radial solutions $(u_{k})_{k\ge0}$ to \eqref{LNTk} is of the form $u_0+B$ if there exists $u_0 \in C^{2,\theta}(B_R)$, $\theta \in (0,1)$,  $u_0 \ge 0$, and a sequence of bubbles
$$B_{\lambda_k} = [N(N-2)]^{\frac{N-2}{4}} \left( \dfrac{\lambda_k}{\lambda_k^2 +|x|^2} \right)^{\frac{N-2}{2}} $$
with $\lambda_k >0$, $\lambda_k \rightarrow 0$ as $k \to + \infty$, such that
\begin{equation} \label{u0B}
u_{k}=u_0+ B_{\lambda_k}+o(1) \quad \text{ in } H^1(B_R),
\end{equation}
as $k\rightarrow +\infty$. If $u_0 \equiv 0$ in the previous definition we just say that $(u_k)_{k\ge0}$ is of the form $B$. 

\medbreak

We prove in Section \ref{LS} that positive radial blowing up solutions satisfying \eqref{energy} exist for $N \ge 4$ when $p_k\to 2^*$ from below and in  all the cases not covered by Theorem\ref{thstabi-Nge7} when $p_k\to 2^*$ from above. We start with a subcritical bubbling sequence: 
\begin{thm}[Type $B$ in the subcritical case in dimensions $N\ge 4$] \label{thinstabi-typeB}\mbox{ }\\
If $N \ge 4$, there exists a sequence $p_k \to 2^*$ with $p_k < 2^*$ for all $k \ge0$ and a sequence $(u_k)_{k\ge0}$ of positive radial solutions of \eqref{LNTk} of type $B$. 
\end{thm}
Note that in the (sub)critical case a sequence of solutions of type $B$ or $u_0 +B$ is uniformly bounded in $L^{2^*}(B_R)$. That a blowing up sequence of solutions $(u_k)_{k \ge0}$ in the subcritical case has to be of type $B$ in dimensions $N=4,5$ follows from Theorem \ref{thblowup-intro-critandsubcrit}. Theorem \ref{thinstabi-typeB} does not cover the case $N=3$. In this case blowing up solutions of type $B$ have been constructed in  \cite{ReyWei}. We discuss the $N=3$ case in more details in Remark \ref{remN3} below.

Still in the subcritical case, we also construct solutions of type $u_0+B$, in dimension $N\ge 6$.
Our constructions of solutions of type $u_0+B$ require that $u_0$ is a non-degenerate solution of \eqref{LNT} with the critical exponent $p=2^*$. We recall that, given $h_0 \in C^0(\overline{B_R})$, a solution $u_0$ of 
\begin{equation} \label{eqexcrit}
\left \{ \begin{aligned}
- \triangle u_0 + h_0 u_0 & = |u_0|^{2^*-2}u_0 & \textrm{ in } B_R \\
\partial_\nu u_0 & = 0 & \textrm{ in } \partial B_R.
\end{aligned} \right.
\end{equation} 
is \emph{nondegenerate} if, for any $\varphi \in H^1(B_R)$,
\begin{equation} \label{nondege}
\left \{ \begin{aligned}
- \triangle \varphi + \Big( h_0 - (2^*-1)|u_0|^{2^*-2} \Big) \varphi &= 0  &\textrm{ in } B_R \\
\partial_\nu \varphi & = 0 & \textrm{ in } \partial B_R.
\end{aligned} \right. 
\end{equation}
implies $\varphi \equiv 0$. 
\begin{thm}[Type $u_0+B$ in the subcritical case in dimensions $N\ge 6$] \label{thinstabi_u0+Bsub}\mbox{ }\\
Let $N\ge 6$ and $u_0\in C^{2,\theta}$, $\theta \in (0,1)$, be a positive nondegenerate radial solution of \eqref{LNT} with the critical exponent $p=2^\ast$. If $N=6$, assume in addition that $u_0(0)< \frac12$. There exists a sequence $p_k$ with $\lim_{k \to + \infty} p_k = 2^\ast$, $p_k < 2^*$ for all $k \ge 0$, and a sequence  $(u_k)_{k \ge 0}$ of positive radial solutions of \eqref{LNTk} of type $u_0 + B$. Moreover, we have that
\begin{itemize}
\item $u_{k}-1$ has $i+1$ zeros, if $u_0(0)<1$, 
\item $u_k-1$ has $i$ zeros if $u_0(0) \ge 1$, 
\end{itemize}
where $i$ is the number of zeros of $u_0-1$. 
\end{thm}
The assumption that $u_0$ is non-degenerate is not restrictive and is generic in the choice of $R$ as Proposition \ref{nondegu0} below shows. We mention that in some cases non-degeneracy is not required to construct blowing up solutions, see \cite{RobertVetois4}.  If $N=6$, solutions $u_0$ with $u_0(0)<\frac12$ exist at least for $R>R_1$, see for instance \cite{Ni83} and Remark \ref{Rem:N=6}.
The conditions $N \ge 6$ and $u_0(0) \le \frac12$ when $N=6$ are necessary in view of Theorem \ref{thblowup-intro-critandsubcrit}. 

Our last constructive result of blow-up regards the supercritical case. As shown by Theorems  \ref{thstabi-Nge7} and \ref{thblowup-intro-supercritical}, only dimensions $3\le N\le 6$ allow for supercritical bubbling and solutions cannot be of type $B$ in dimension $4\le N\le 6$.

\begin{thm}[Type $u_0+B$ in the supercritical case in dimensions $3\le N\le 6$] \label{thinstabi}\mbox{ }\\ 
Let $3\leq N \le 6$ and  $u_0\in C^{2,\theta}$, $\theta \in (0,1)$, be a positive nondegenerate radial solution of \eqref{LNT}  with the critical exponent $p=2^\ast$. 
If $N=6$, assume in addition that $u_0(0)> \frac12$. Then there exists a sequence $p_k$ with $\lim_{k \to + \infty} p_k = 2^\ast$, $p_k > 2^*$ for all $k \ge 0$, and a sequence  $(u_k)_{k \ge 0}$ of positive radial solutions of \eqref{LNTk} of type $u_0 + B$. Moreover, we have that
\begin{itemize}
\item $u_{k}-1$ has $i+1$ zeros, if $u_0(0)<1$, 
\item $u_k-1$ has $i$ zeros if $u_0(0) \ge 1$, 
\end{itemize}
where $i$ is the number of zeros of $u_0-1$. 
\end{thm}
If $N=6$, non constant solutions $u_0$ with $u_0(0)>\frac12$ exist for any radius $R>0$, see \cite{Ni83,MR1106295,MR3564729} and Remark \ref{Rem:N=6}. 
Observe also that $u_0 \equiv 1$ can be chosen in Theorems \ref{thinstabi_u0+Bsub} and \ref{thinstabi}. In that case, the nondegeneracy condition is even explicit : $R$ should be such that $\frac{4}{N-2}$ is not a radial eigenvalue of $-\Delta$ on the ball $B_R$ with Neumann boundary condition, i.e. only a discrete set of radii is forbidden.  

When $N=3$ and $p_k > 2^*=6$ blowing up solutions may be of type $B$ and have for instance been constructed in \cite{ReyWei}. 

\medbreak

The Figure \ref{table} summarizes the classification. It shows the classification is almost complete. 
\begin{figure}[h!t]\label{table}
\begin{center}
\begin{tabularx}{16.475cm}{|c|p{6.2cm}|p{2cm}|p{6.2cm}|}
    \hline
    N & $p_k<2^*$ & $p_k=2^*$ & $p_k>2^*$ \\
    \hline
    $3$ & $\bullet$ {no blow-up if $R<R^*$} & {no blow-up}$^1$ &  \\
&  $\bullet$ {single bubble only} &  & $\bullet$ \underline{towers might exist} \\
&  $\bullet$ {only type $B$ is possible} &  & $\bullet$ {no type $B$ if $R>R^*$}  \\
 &  \emph{$\bullet$ type $B$ occurs for large $R$}  &  & \emph{$\bullet$ type $B$ occurs for small $R$}\\
 & & & $\bullet$ {type $u_0+B$ with $u_0>0$ occurs}  \\ 
\hline
$4,5$ &  $\bullet$ single bubble only& no blow-up  & $\bullet$ \underline{towers might exist}    \\
 &  $\bullet$ only type $B$ is possible &   & $\bullet$ no type $B$  \\
 & $\bullet$ type $B$ occurs & & $\bullet$ type $u_0+B$ with $u_0>0$ occurs  \\
 \hline
 6 & $\bullet$ single bubble only & no blow-up$^2$ & $\bullet$ \underline{towers might exist}\\
  & $\bullet$ only type $u_0+B$ with $u_0(0)\le1/2$  is possible & & $\bullet$ only type $u_0+B$ $u_0(0)\ge1/2$ is possible \\ 
   & $\bullet$ type $u_0+B$ with $u_0(0)<1/2$ occurs & & $\bullet$ type $u_0+B$ with $u_0(0)>1/2$ occurs\\ 
\hline
$\geq 7$ & $\bullet$ \underline{towers might exist} & no blow-up  &  no blow-up \\
& $\bullet$ type $B$ occurs & & \\
& $\bullet$ type $u_0+B$ with $u_0>0$ occurs & & \\
\hline 
\end{tabularx} 

$^1$ see Remark \ref{Rem:N=3crit} - $^2$ see Remark \ref{Rem:N=6crit}
\caption{All the assertions are proved in the paper except those two in italic which are proved in \cite{ReyWei}. We also underlined the cases where we expect towers of bubbles.}
\end{center}
\end{figure}
In dimension $N=3$, one expects blow-up occurs as soon as $R>R^*$ in the subcritical case and $R<R^*$ in the supercritical case. This would require a quantitative analysis of the conditions in \cite{ReyWei}. As already mentioned, the case $R=R^*$ remains open (also in the sub- and super-critical cases). The classification is complete in dimensions $N=4,5$ and $N\ge 7$. The missing piece in dimension $N=6$ concerns the critical case. Finally, we mention that we did not address in this paper the construction of towers of bubbles when they can exist. This part of the classification will be the subject of further works. 
We also emphasize that our focus in this paper has been put on the classification of \emph{radial} blow-up solutions of \eqref{LNTk} at the critical Sobolev exponent. The techniques used in this paper have however been successfully used to prove precompactness results even in the non-radial case: \cite{Premoselli13, PremoselliVetois3,PremoselliVetois2}. Let us also mention that equations like \eqref{LNT} are sensitive to perturbations: when the linear operator in the left-hand side is replaced by $- \triangle$ the situation can be quite different, see for instance  \cite{MR4400170} or \cite{PST}.

\medskip

The proof of Theorems \ref{thinstabi-typeB} to \ref{thinstabi} goes through a Lyapunov-Schmidt reduction and follows closely \cite{MichelettiPistoiaVetois} and \cite{RobertVetois3} (see also \cite{MR3215478}). 
We only construct positive solutions in Theorems \ref{thinstabi-typeB} to \ref{thinstabi} since this is enough to prove the sharpness of Theorems \ref{thstabi-Nge7} and \ref{thstabi2etoile}. It is however very likely that the approach we develop in Section \ref{LS} will yield sign-changing solutions of \eqref{LNTk} having the form $u_0-B$ under suitable conditions in the subcritical case. Towers of bubbles for \eqref{LNTk} are also likely to occur, with or without alternating signs (see Lemma \ref{bubbletowers} below where we prove necessary conditions for towers of bubbles). For critical elliptic equations such as \eqref{LNTk} towers of bubbles have been constructed in \cite{MichelettiPistoiaVetois, PistoiaVetois, MorabitoPistoiaVaira, Premoselli12}. The references \cite{MichelettiPistoiaVetois, RobertVetois3} concern a similar problem on a manifold without boundary (and are not restricted to the radial case) but the techniques straightforwardly adapt to open sets with Neumann boundary conditions and therefore to \eqref{LNTk}. 

\medskip

\subsection{Radial bifurcation analysis of \eqref{LNT}} \label{motivation}

We conclude this introduction by describing how Theorems \ref{thstabi-Nge7} to \ref{thinstabi} complement the radial bifurcation analysis of \cite{MR3564729} around the critical exponent $p=2^*$. In \cite{MR3564729} the authors have studied radial positive solutions of problem \eqref{LNT} bifurcating from the trivial solution $1$ as the parameter $p$ varies. Let $\lambda_i^{rad}$ be the $i$-th eigenvalue of the operator $-\Delta +1$ in $B_R$ with Neumann boundary condition on radial solutions. Among other things, the authors have shown in \cite{MR3564729} that for every $i\geq 2$, $(2+\lambda_i^{rad},1)$ is a bifurcation point for problem \eqref{LNT}. If $\Bcal_i$ denotes the continuum of solutions that branches out of $(2+\lambda_i^{rad},1)$ they also proved that
\begin{itemize}
\item[(i)] the branches $\Bcal_i$ are unbounded and do not intersect; close to $(2+\lambda_i^{rad},1)$, $\Bcal_i$ is a $C^1$-curve;
\item[(ii)] if $u \in \Bcal_i$ then $u >0$;
\item[(iii)] each branch consists of two connected components: the lower branch $\Bcal_i^-$, along which $u (0)<1$, and the upper branch $\Bcal_i^+$, along which $u (0)>1$;
\item[(iv)] if $u \in \Bcal_i$ then $u-1$ has exactly $i-1$ zeros, $u'$ has exactly $i-2$ zeros and each zero of $u'$ lies between two zeros of $u-1$;
\item[(v)] the functions satisfying $u (0)<1$ are uniformly bounded in the $C^1$-norm. 
\end{itemize}
\begin{figure}[h!t]\label{fig}
\begin{center}
{\includegraphics[height=6cm]{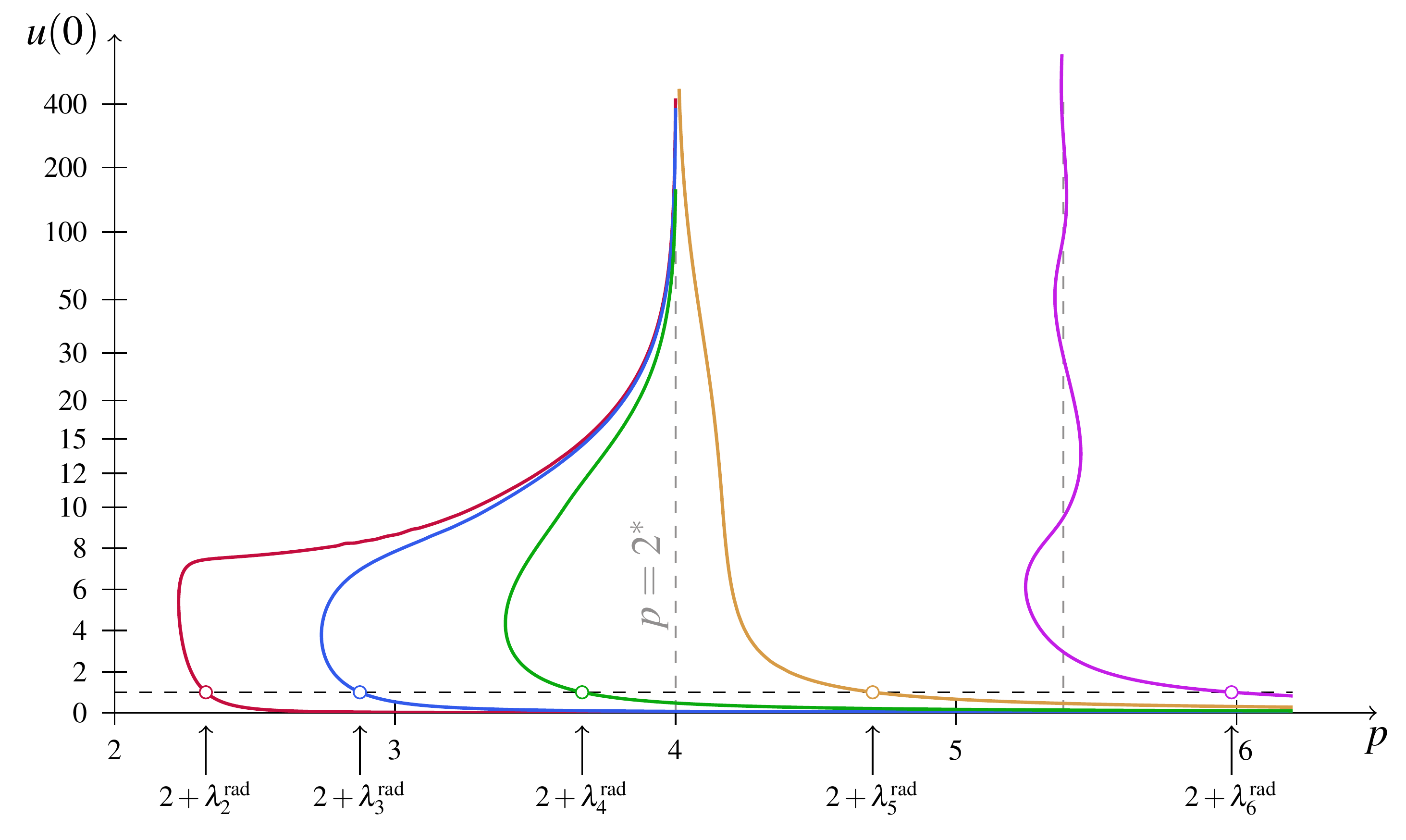}}
\caption{Bifurcation diagram in dimension $N=4$. This figure is taken from \cite{MR3564729} with the courtesy of the authors. }
\end{center}
\end{figure}

The numerical simulations of \cite{MR3564729} seem to indicate that the upper branches $\Bcal_i^+$ can be split into two categories: those having a vertical asymptote at $p =2^*$, that we shall call ``critical'', and those that have vertical asymptotes at supercritical exponents $p_i > 2^*$, that we shall call ``supercritical'' : see e.g. Figure $1$ of this paper and Figure $13$ and Figure $16$ of \cite{MR3564729}. 

Theorems \ref{thinstabi-typeB} to \ref{thinstabi} provide explicit examples of solutions that, we believe, describe the critical upper branches constructed in \cite{MR3564729}. For instance, as Theorems \ref{thinstabi-typeB} shows, positive radial solutions of type $B$ exist for all dimensions $N\ge 4$ in the subcritical case. These solutions take the value $1$ only once and satisfy $u(0)>1$ when $p$ is close to $2^*$. They are good candidates to belong to the second upper branch $\Bcal^+_2$ but a proof of this fact seems quite difficult. Similarly, the branches $\Bcal^+_i$, for $i\geq 3$, are made of functions that take $i+1$ times the value $1$  so it is natural to believe that they correspond to the solutions of the form $u_0^i +B$ constructed in Theorems \ref{thinstabi_u0+Bsub} and \ref{thinstabi}, where $u_0^i$ is a solution of \eqref{eqexcrit} which takes $i$ (if $u^i(0)<1$) or $i+1$ (if $u^i(0)\geq 1$) times the value $1$. A weak limit $u_0^i$ that satisfies $u_0^i(0) <1$ and that takes $i$ times the value $1$ is for instance obtained by taking the intersection of the lower branch $\Bcal^-_{i+1}$ with the vertical line $p=2^\ast$. This always occurs for the $i \ge 1$  satisfying $2+\lambda^i_{rad}<2^\ast$. Note that, for fixed $i\ge1$, $\lambda^i_{rad}$ goes to $0$ as $R \to +\infty$. Up to assuming that  $R$ is big enough we can thus generically find weak solutions $u_0^i$ that oscillate around $1$ a fixed arbitrary number of times. 

Thorem \ref{thstabi-Nge7} and Theorem \ref{thstabi2etoile} also provide informations on the behaviour of the critical branches $\Bcal^+_i$. We now explain how this sheds new light on the numerical simulations of \cite{MR3564729}. Consider Figure $1$ of this paper, which covers dimension $N=4$. We consider the two branches $\Bcal_3^+$ and $\Bcal_4^+$ (those starting from $2+\lambda_{rad}^3$ and $2+\lambda_{rad}^4$). These branches seem to blow up from the left of the vertical line $p=2^*$. \emph{We assume that this is the case.} If we moreover assume that they have bounded $H^1(B_R)$ energy contradicts Theorem \ref{thblowup-intro-critandsubcrit}: this Theorem indeed claims, since $N=4$, that solutions along $\Bcal_3^+$ and $\Bcal_4^+$ blow up with a single bubble at $0$. But then Proposition \ref{theorieC0} below would show that these solutions only take the value $1$ once, which is not the case. There are two possible explanations for this phenomenon. The first is that the solutions on the branches $\Bcal_3^+$ and $\Bcal_4^+$ in the figure all have finite energy and actually cross the line $p=2^\ast$ for very large values of $u(0)$ that are not in the figure; and they then blow up from the right of the vertical line $p=2^*$ according to Theorem \ref{thinstabi}, or may even blow up at a larger supercritical exponent $p_*>2^*$. The second is that the solutions along $\Bcal_3^+$ and $\Bcal_4^+$ do blow up from the left of $p=2^*$ but have infinite energy as $u(0) \to + \infty$ along the branch. The latter situation could correspond, for instance, to an infinite tower of bubbles centered at $0$ glued to a non-zero weak limit $u_0$ as $p \to 2^*$, $p \le 2^*$. Note however that infinite tower of bubbles as $p \to 2^*$ are not known to exist in the context of \eqref{LNT} yet.  In the same vein, Theorem \ref{thstabi2etoile} implies that the finite-energy (in the sense of \eqref{energy}) branches that blow up along the vertical line $p=2^*$ can only oscillate around it a finite number of times (dimension $3$ and $6$ are not fully covered by Theorem \ref{thstabi2etoile}). If Conjecture \ref{conjstabi2etoile} holds, it completely rules out a snaking behaviour of  branches of finite-energy solutions around $p=2^*$. 

We also expect the supercritical branches to correspond to the solutions that  have been constructed in \cite{MR4149344} when $p>2^\ast$. These are positive radial solutions blowing up at the origin and intersecting an arbitrary number of time the value $1$. These solutions were proved  to have infinite Morse index if their vertical asymptote $p$ satisfies $2^\ast -1 <p <p_{JL}$, where $p_{JL}$ is the Joseph and Lundgren exponent  (see \cite{MR340701}) and have finite Morse index if $p>p_{JL}$. This suggests that the branches with vertical asymptote between $2^\ast -1$ and $p_{JL}$ oscillates infinitely many times around the vertical line $p$, a phenomenon which is also observed numerically see Figure $1$. We refer to \cite{MR4053034} for similar results related to the Keller-Segel equation. In this case, the Joseph and Lundgren exponent is replaced by a condition on the dimension of $B_R$. 

The previous considerations describe the behaviour of the branches $\Bcal_i$ in the direction where $u(0)$ blows up. It is also natural to investigate what happens on the other end of $\Bcal_i$, i.e. when $p\rightarrow \infty$. We believe that the solutions corresponding to the lower branches $\Bcal_i^-$ are the solutions that have been constructed in \cite{MR3487266}. There the authors prove, when $p\rightarrow \infty$, the existence of radial positive solutions concentrating on an arbitrary number of spheres and thus intersecting an arbitrary number of times the value $1$  and whose value at the origin is strictly less than $1$ as $p\rightarrow \infty$. Let us finally mention that analogous results have also been obtained for the Keller-Segel equation in \cite{MR3641921}, where the polynomial nonlinearity is replaced by $e^{\mu (u-1)}$ and the bifurcation analysis is done with respect to the parameter $\mu>0$. 

\subsection{Organization of the paper}
In Section \ref{apriori} we prove a sharp pointwise asymptotic description of finite-energy solutions of \eqref{LNTk}, namely Proposition \ref{theorieC0}. This is, to our knowledge, the first sharp a priori asymptotic pointwise description of smooth solutions of \eqref{LNTk} when $p_k >2^*$, $p_k \to 2^*$. 
We apply it to obtain necessary condition for the existence of blowing up solutions of \eqref{LNTk} in Section \ref{pohosec} in a general setting (Theorem \ref{thblowup}). We prove Theorems \ref{thstabi-Nge7} to \ref{thblowup-intro-supercritical} as a consequence. We finally prove Theorems \ref{thinstabi-typeB} to \ref{thinstabi} in Section \ref{LS}.

\section{A priori analysis of finite-energy solutions of \eqref{LNTk} } \label{apriori}

Let $N \ge 3$ and let  $(p_k)_{k\ge0}$ be a sequence of positive numbers, $p_k \ge 2$, converging to $2^* = \frac{2N}{N-2}$ as $k \to + \infty$. For $R>0$ we will denote by $B_R \subset \R^N$ the euclidean ball centered at $0$ and of radius $R$. Throughout this section we let $(u_k)_{ k\ge 0}$ be a sequence of possibly sign-changing radial solutions of \eqref{LNTk} in $C^2(\overline{B_R})$ satisfying \eqref{energy}. A consequence of \eqref{LNTk} and \eqref{energy} is that $u_k$ weakly converges in $H^1(B_R)$ and $L^{2^*}(B_R)$ towards some weak limit $u_0$, that solves \eqref{LNT} with $p=2^*$ and is hence at least in $C^2(\overline{B_R})$.
\smallskip 

We assume throughout this section that the sequence $(u_k)_{k \ge0}$ blows up, that is
\begin{equation} \label{blowup}
\lim_{k \to + \infty} \Vert u_k \Vert_{L^\infty(B_R)} = + \infty 
\end{equation}
up to a subsequence. For any $k \ge 0$, since $u_k$ is radial, it satisfies the ODE
\begin{equation} \label{LNTkr} 
\left \{ \begin{aligned}
- u_k''(r) - \frac{N-1}{r} u_k'(r)  + u_k(r) & = |u_k(r)|^{p_k-2}u_k(r) & \textrm{ in } [0,R] \\
u_k'(0) =u_k'(R) & = 0. & 
\end{aligned} \right.
\end{equation}
Since the energy
\begin{equation} \label{dec}
 r \mapsto E_k(r) := \frac{|u_k'(r)|^2}{2} +  \frac{|u_k(r)|^{p_k}}{p_k} -  \frac{|u_k(r)|^2}{2}
\end{equation} 
is nonincreasing in $[0,R]$, 
$\max_{B_R} |u_k |= |u_k(0)|$, hence $ |u_k(0)|\to + \infty$ as $k \to + \infty$ by \eqref{blowup}. Let $1\le \ell \in \mathbb N$ and $(\mu_{1,k})_{k\ge0}$, $\dots$, $(\mu_{\ell,k})_{k \ge0}$ be $\ell$ sequences of positive numbers converging to $0$ as $k \to + \infty$. For $1 \le i \le \ell$, we adopt the notations
\begin{equation} \label{bulles}
 B_{i,k}(r) = \big(N(N-2) \big)^{\frac{N-2}{4}} \left( \frac{\mu_{i,k}}{\mu_{i,k}^2 + r^2}\right)^{\frac{N-2}{2}} \quad \textrm{ and } \quad \tilde{B}_{i,k}(r) =\mu_{i,k}^{\frac{N-2}{2} - \frac{2}{p_k-2}} B_{i,k}(r),
 \end{equation}
 \begin{equation} \label{B0}
 B_0(r) = \big( N(N-2) \big)^{\frac{N-2}{4}}\Big( 1+ r^2 \Big)^{1 - \frac{N}{2}},
 \end{equation}
 where $ r \ge 0$. 
It is well-known that
\[ \int_{\R^N} B_0^{2^*} dx = K_N^{-N}, \]
where $K_N$ is the sharp constant for the embedding of $\dot{H}^1(\R^N)$ (the homogeneous Sobolev space) into $L^{2^\ast}(\R^N)$ given by 
\begin{equation} \label{KNN}
K_N=\sqrt{\frac{4}{N (N-2)\omega_N^{\frac{2}{N}}}}.
\end{equation}
Since $p_k \to 2^*$, observe  that $\frac{N-2}{2} - \frac{2}{p_k-2} \to 0$ as $k \to + \infty$, and $\tilde{B}_{i,k}$ has to be understood as a perturbation of $B_{i,k}$ that takes into account the asymptotic criticality of \eqref{LNTkr}. The main result of this section is the following, sharp pointwise asymptotic description of $u_k$: 
\begin{prop} \label{theorieC0}
Assume that $(u_k)_{k \ge0}$ satisfies \eqref{blowup} \eqref{LNTkr} and \eqref{energy} and let $u_0$ be its weak limit in $H^1(B_R)$. There exist a sequence $(\e_k)_{k\ge0}$ converging to $0$, $C>0$, an integer $L \ge 1$, $L$ sequences $(\mu_{1,k})_{k\ge0}, \dots, (\mu_{L,k})_{k \ge0}$ of positive numbers satisfying $\mu_{i,k} = o(\mu_{i+1,k})$ for any $1 \le i \le L-1$ and a family $(\kappa_i)_{1 \le i \le L} \in \{\pm1\}^L$ such that
\[ \left| u_k(r) - u_0(r) - \sum_{i=1}^L \kappa_i\tilde{B}_{i,k}(r)\right| \le \e_k \Big( \Vert u_0 \Vert_{L^\infty(B_R)} + \sum_{i=1}^L \tilde{B}_{i,k}(r) \Big) + C \mu_{L,k}^{N-2-\frac{2}{p_k-2}}\]
for any $k \ge 0$ and every $r \in [0,R]$.
\end{prop}

Proposition \ref{theorieC0} implies that for any sequence $(r_k)_{k \ge0}$ of points of $[0,R]$, we have 
\[ u_k(r_k) = u_0(r_k) + \sum_{i=1}^L \kappa_i \tilde{B}_{i,k}(r_k) + o \Big( \Vert u_0 \Vert_{L^\infty(B_R)}  + \sum_{i=1}^L \tilde{B}_{i,k}(r_k)  \Big) +  O(\mu_{L,k}^{N-2-\frac{2}{p_k-2}}) \]
as $k \to + \infty$. In other words, $u_k$ looks like, at a pointwise level and in strong spaces,  the sum of alternating positive bubbles centred at $0$ and concentrating at different rates.

We prove Proposition \ref{theorieC0} in the general context of possibly \emph{sign-changing} solutions, and not just positive ones. Pointwise descriptions of finite-energy blowing up \emph{positive} sequences of solutions of equations like \eqref{LNTk} have been known to hold for a while. A general result is available in \cite{DruetHebeyRobert}, and other possible references are \cite{LiZhu,DruetJDG, HebeyZLAM}. Specific reference for the Lin-Ni-Takagi problem are \cite{DruetRobertWei, ThizyLinNi}. Very recently, in \cite{Premoselli13}, this type of results has been extended to the general setting of possibly \emph{sign-changing} solutions, and applications to the precompactness of sign-changing solutions of Schr\"odinger-Yamabe type equation on manifolds have been given in \cite{PremoselliVetois2, PremoselliVetois3}. An analogous result for sign-changing solutions of Hardy-Schr\"odinger type equations is in \cite{GhoussoubMazumdarRobert}.

The proof of Proposition \ref{theorieC0} goes through an iterative construction of bubbling scales. The radiality of $u_k$ rules out non-trivial sign-changing bubbles and forces $u_k$ to blow up, at the origin, as a sum of positive bubbles with possibly alternating signs. This observation allows us to adapt in the context of \eqref{LNTk} the proof of \cite{GhoussoubMazumdarRobert}, which is our main inspiration. Note that our proof does not assume the existence of a Struwe-type decomposition  \cite{Struwe} for solutions of \eqref{LNTk}: for this reason it also works in the asymptotically supercritical case $p_k \ge 2^*$ (when \eqref{energy} is satisfied).

\medskip

The remaining part of the section is dedicated to the proof of Proposition \ref{theorieC0}. The proof goes through several steps. The arguments in what follows are written using the notations from \eqref{LNTk} since many of them would adapt to the non-radial case. We however explicitly mention when radiality is crucially used. 

\medskip

\textbf{Step $1$: A global scale-invariant estimate.} We claim that there exists $C >0$ such that, for all $k \ge0$ and $r \in [0,R]$,
\begin{equation} \label{C01}
r^{\frac{2}{p_k-2}} |u_k(r)| \le C .
\end{equation}
In the case where $p_k \le 2^*$ for all $k \ge 0$, the estimate \eqref{C01} follows directly from Strauss's lemma \cite{Strauss} in $H^1(B_R)$, i.e.
\begin{equation*} \label{C01-sub}
r^{\frac{2}{p_k-2}} |u_k(r)| \le  R^{\frac{2p_k-N(p_k-2)}{2(p_k-2)}}r^{\frac{N-2}{2}}|u_k(r)| \le C.
\end{equation*}
When $p_k < 2^*$, the power $\frac{2}{p_k-2}$ is clearly not sharp. 
In the supercritical case $p_k > 2^*$, we need a preliminary step. The assumption \eqref{energy} implies $(u_k)_k$ is bounded in $W^{2,s}(B_R)$ with $s=\frac{N(p_k-2)}{2p_k-2}$ (observe this exponent is strictly bigger than $1$). By Sobolev embeddings,  $(u_k)_k$ is therefore bounded in $W^{1,t}(B_R)$ with $t= \frac{Ns}{N-s} = \frac{N(p_k-2)}{p_k}$. Now, Strauss's radial lemma in $W^{1,t}(B_R)$, see for instance \cite[Corollaire II.1]{MR683027}, implies 
\begin{equation*} \label{C01-super}
r^{\frac{N-t}{t}} |u_k(r)| = r^{\frac{2}{p_k-2}} |u_k(r)| \le C .
\end{equation*}
In the case $p_k > 2^*$, the power $\frac{2}{p_k-2}$ is probably sharp.

\smallskip A first consequence of \eqref{LNTk}, \eqref{C01} and standard elliptic theory is that 
\begin{equation} \label{C02}
 u_k \to u_0 \quad \text{ in } C^2([\delta, R]), \text{ for any fixed } \delta >0 
 \end{equation}
as $k \to + \infty$. Hence blow-up only occurs at $0$. We now iteratively construct the blow-up scales of $(u_k)_{k \ge0}$. If $\ell \ge 1$ is an integer we consider the following property that we denote by $(H_{\ell})$:
\begin{equation} \label{Hl} \tag{$H_\ell$}
\begin{array}{l}
\textrm{there exist } \ell \textrm{ sequences } (\mu_{1, k})_{k \ge0}, \dots, (\mu_{\ell,k})_{k\ge0} 
 \textrm{ of positive numbers }
 \textrm{ satisfying: }\\[2mm]
 \quad \bullet\ \mu_{i,k} = o(\mu_{i+1,k}) \text{ as } k \to + \infty \text{ for any } 1 \le i \le \ell-1 \text{ and } \lim_{k \to + \infty} \mu_{\ell,k} = 0 \\[2mm]
\quad  \bullet\ \mu_{i,k}^{\frac{2}{p_k-2}} u_k \big( \mu_{i,k}\, \cdot) \to \kappa_i B_0 \text{ in } C^2_{loc}(\R^N \backslash \{0\}) \text{ for some } \kappa_i \in \{-1,1\} \\[2mm]
 \text{ as } k \to + \infty \text{ for any } 1 \le i \le \ell. \\ 
\end{array} 
\end{equation}

\medskip

\textbf{Step $2$: $(H_1)$ holds true.}  By \eqref{blowup} and \eqref{dec}, $\max_{B_R} |u_k| = |u_k(0)| \to + \infty$ as $k \to + \infty$. We can thus let 
\begin{equation} \label{m1k}
 \mu_{1,k}=\big( N(N-2)\big)^{\frac{(N-2)(p_k-2)}{8}} |u_k(0)|^{- \frac{p_k-2}{2}}, 
 \end{equation}
and $\mu_{1,k} \to 0$ as $k \to + \infty$. The sequence $v_{1,k}(x) : = \mu_{1,k}^{\frac{2}{p_k-2}} u_k \big(\mu_{1,k} x \big)$, defined for $x \in B_{\frac{R}{\mu_{1,k}}}$ then satisfies $|v_{1,k}(0)|=\big(N(N-2)\big)^{\frac{N-2}{4}}$, $\Vert v_{1,k} \Vert_{L^\infty(B_{\frac{R}{\mu_{1,k}}})} \le \big(N(N-2)\big)^{\frac{N-2}{4}}$ and, by \eqref{LNTk}, 
\[ - \triangle v_{1,k} + \mu_{1,k}^2 v_{1,k} = v_{1,k}^{p_k-1} \quad \text{ in } B_{\frac{R}{\mu_{1,k}}}.\]
By standard elliptic theory, and since $u_k$ is radial, $v_{1,k}$ converges in $C^2_{loc}(\R^N)$, up to a subsequence, towards a radial solution $v_1$ of $- \triangle v_1 = |v_1|^{2^*-2}v_1$ in $\R^N$ with $|v_1(0) |= \big(N(N-2)\big)^{\frac{N-2}{4}}$ and $\Vert v_1 \Vert_{^\infty(\R^N)}\le \big(N(N-2)\big)^{\frac{N-2}{4}}$. By the Pohozaev identity \cite{Pohozaev} there do not exist radial sign-changing solutions of this equation in $\R^N$. Hence $v_1$ is of constant sign, and by the classification result of \cite{CaGiSp} there exists $\kappa_1 \in \{\pm1\}$ such that $v_1(r) = \kappa_1 B_0(r)$, where $B_0$ is as in \eqref{B0}. Since $v_{1,k} \to B_0$ in $C^2_{loc}(\R^N)$ this proves $(H_1)$.

\medskip

\textbf{Step $3$: the induction property.}

\medskip

\textbf{Step $3.1$.} We  claim that the following holds true: assume that $(H_\ell)$ is true for some $\ell \ge 1$. Then either 
\begin{equation} \label{C03}
 \lim_{M \to + \infty} \limsup_{k \to + \infty} \max_{[M \mu_{\ell,k}, R]} r^{\frac{2}{p_k-2}} |u_k(r) - u_0(r)| = 0 
 \end{equation}
or $(H_{\ell+1})$ is true. 

\smallskip

\begin{proof}[Proof of \eqref{C03}]
We proceed by contradiction and, up to passing to a subsequence, assume that there exists $s_k \in (0,R]$ with $s_k >> \mu_{\ell,k}$ and 
\begin{equation} \label{Step31}
s_k^{\frac{2}{p_k-2}}|u_k(s_k) - u_0(s_k)| \to a >0 
\end{equation}  as $k \to + \infty$. Remark that the right-hand side of \eqref{Step31} is always finite by \eqref{C01}. By \eqref{C02} $s_k \to 0$, thus $s_k^{\frac{2}{p_k-2}}|u_k(s_k) |\to a$ and $|u_k(s_k)| \to + \infty$. For all $r \in [0, \frac{R}{s_{k}}]$, we let 
\[ w_k(r) = s_k^{\frac{2}{p_k-2}} u_k\big( s_k r\big). \]
By \eqref{LNTk} and  \eqref{C01} $w_k$ satisfies $- \triangle w_k + s_{k}^2 w_k = |w_k|^{p_k-2}w_k$ in $B_{\frac{R}{s_k}}$ and $|w_k(r)| \le Cr^{- \frac{2}{p_k-2}}$ for any $0 < r \le \frac{R}{s_k}$. By standard elliptic theory $w_k$ thus converges, in $C^2_{loc}(\R^N \backslash \{0\})$, towards a function $w_0$ satisfying $|w_0(1) |= a >0$ and 
\[ - \triangle w_0 = |w_0|^{2^*-2}w_0 \quad \text{ in } \R^N \backslash \{0\}. \]
Using \eqref{energy} we have, for $k \ge 0$ and for all $0<\delta <1$ fixed, 
\[ \begin{aligned}
 \int_{B_{\frac{1}{\delta}} \backslash B_\delta}  |w_0|^{2^*} dx & = \int_{B_{\frac{1}{\delta}} \backslash B_\delta} |w_k|^{\alpha_k} dx + o(1) \\
 & = s_k^{\frac{2\alpha_k}{p_k-2} - N} \int_{B_{\frac{\nu_{\ell+1,k}}{\delta}} \backslash B_{\delta \nu_{\ell+1,k}}} |u_k|^{\alpha_k} dx +o(1)\\
 & \le C s_k^{\frac{2\alpha_k}{p_k-2} - N} +o(1). 
\end{aligned} \]
As explained in the proof of Step $1$ we have $s_k^{\frac{2\alpha_k}{p_k-2} - N} \le 1$ so that, after passing to the limit as $k\to + \infty$ and then as $\delta \to 0$, $w_0 \in L^{2^*}(\R^N)$. Classical removable singularity results (see e.g. Lemma $2.1$ in \cite{CaGiSp}) show that $w_0$ satisfies $- \triangle w_0 = |w_0|^{2^*-2}w_0$ in $\R^N$ and is thus at least of class $C^2$. As in the proof of Step $2$, since $u_k$ is radial $w_0$ is also radial and in particular does not change sign. Since $|w_0(1)| = a >0$ the classification result of \cite{CaGiSp} again shows that 
\[ w_0(r) = \kappa_{\ell+1}\frac{\lambda^{\frac{N-2}{2}}}{\big( \lambda^2 + \frac{r^2}{N(N-2)} \big)^{\frac{N-2}{2}}} \]
for some $\lambda >0$ depending on $a$ and some $\kappa_{\ell+1} \in \{\pm1\}$. It remains to let 
\[\mu_{ \ell+1,k} =\big(N(N-2)\big)^{\frac12}  \lambda s_k.\]
Since $s_k >> \mu_{\ell,k}$ and $s_k = o(1)$ this proves that $(H_{\ell+1})$ is satisfied for the sequences $(\mu_{1,k})_{k \ge0}, \dots, (\mu_{\ell,k})_{k \ge0}$ and $ (s_k)_{k \ge 0}$. 
\end{proof} 

\medskip

\textbf{Step $3.2$.} We  claim that the following holds true: assume that $(H_\ell)$ is true for some $\ell \ge 1$ and let $\delta >0$. Then either 
\begin{equation} \label{C032}
 \lim_{M \to + \infty} \limsup_{k \to + \infty} \max_{[M \mu_{i,k}, \delta \mu_{i+1,k} ]} r^{\frac{2}{p_k-2}} |u_k(r) - \tilde{B}_{i+1,k}(r)| = 0 
 \end{equation}
for all $1 \le i \le \ell-1$ or $(H_{\ell+1})$ is true. 

\begin{proof}[Proof of \eqref{C032}.]
We proceed again by contradiction and assume that there exists $i \in \{1, \dots, \ell-1\}$ and a sequence $(s_k)_{k \ge0}$ with $\frac{s_k}{\mu_{i,k}} \to + \infty$ and $s_k \le \delta \mu_{i+1,k}$ such that, up to a subsequence,
\begin{equation} \label{step32}
\lim_{k \to + \infty}  s_k^{\frac{2}{p_k-2}} |u_k(s_k) - \tilde{B}_{i+1,k}(s_k)| = a >0.
  \end{equation}
Note that the limit is finite by \eqref{bulles} and \eqref{C01}. By the second bullet in $(H_\ell)$ we also have $s_k = o(\mu_{i+1,k})$, and hence \eqref{step32} becomes, with \eqref{bulles},
\[   \lim_{k \to + \infty}  s_k^{\frac{2}{p_k-2}} |u_k(s_k)|  = a >0.\]
For all $r \in [0, \frac{R}{s_k}]$, we let 
\[ w_k(r) = s_k^{\frac{2}{p_k-2}} u_k\big( s_k r\big). \]
Mimicking the proof of \eqref{C03} shows that $(H_{\ell+1})$ is then satisfied for the sequences $(\mu_{1,k})_{k \ge0}, \dots, (\mu_{i,k})_{k\ge0}$, $\big(\big(N(N-2)\big)^{\frac12} \lambda s_k\big)_{k \ge0}$, $(\mu_{i+1,k})_{k \ge0}, \dots, (\mu_{\ell,k})_{k \ge0}$, where $\lambda$ is some positive number that depends on $a$. 
\end{proof}
\medskip

\textbf{Step $4$: the number of concentration points is finite.} We let 
$$L = \max \{ \ell \ge 1 \text{ such that } (H_\ell) \text{ holds true} \}.$$ We claim that $L < + \infty$ and that, for any fixed $ \delta >0$,
\begin{equation} \label{C04}
\begin{aligned}
&  \lim_{M \to + \infty} \limsup_{k \to + \infty} \max_{[M \mu_{i,k}, \delta \mu_{i+1,k} ]} r^{\frac{2}{p_k-2}} |u_k(r) - \tilde{B}_{i+1,k}(r)| = 0,  \quad  1 \le i \le L-1  \\
& \lim_{M \to + \infty} \limsup_{k \to + \infty} \max_{[M \mu_{L,k}, R]} r^{\frac{2}{p_k-2}} |u_k(r) - u_0(r)| = 0 .
 \end{aligned}
 \end{equation}

\begin{proof}[Proof of the fact that $L < + \infty$ and of \eqref{C04}.]
Let $\ell \ge 1$ for which $(H_\ell)$ is true and let $0<\delta <1$. Then by $(H_\ell)$ we can write that 
\[ \begin{aligned}
\int_{B_R}| u_k|^{\alpha_k} dx &\ge \sum_{i=1}^\ell \int_{B_{\frac{\mu_{i,k}}{\delta}} \backslash B_{\delta \mu_{i,k}}} |u_k|^{\alpha_k} dx \\
& = \sum_{i=1}^\ell \mu_{i,k}^{N - \frac{2\alpha_k}{p_k-2} } \Big( \int_{B_{\frac{1}{\delta}} \backslash B_{\delta }} B_0^{2^*} dx +o(1) \Big) \\
& = \sum_{i=1}^\ell \mu_{i,k}^{N - \frac{2\alpha_k}{p_k-2} } \Big( K_N^{-N} + \e_\delta +o(1) \Big),
\end{aligned} \]
where $\e_\delta \to 0$ as $\delta \to 0$. As before, $N - \frac{2\alpha_k}{p_k-2} \le 0$, so that $\mu_{i,k}^{N-\frac{2\alpha_k}{p_k-2}} \ge 1$ for all $1 \le i \le \ell$. Passing to the limit as $k \to +\infty$ and as $\delta \to 0$ we thus obtain
\[ \ell \le C K_N^{N} \]
where $C$ is the bound on the energy in \eqref{energy}. This shows that $L< \infty$. Equation \eqref{C04} then follows from the maximality of $L$, \eqref{C03} and \eqref{C032}.
\end{proof}

\smallskip

In what follows, for any $0< \delta <R$ fixed, we let 
\[ \eta_k(\delta) = \max_{r \in [\delta, R]} |u_k(r)|. \]

\medskip

\textbf{Step $5$: Interpolation inequalities.}

\medskip

In this step we obtain a first set of pointwise estimates on $|u_k|$ that interpolate between \eqref{C03} and \eqref{C032}.

\medskip

\textbf{Step $5.1$: interpolation between $\tilde{B}_{L,k}$ and $u_0$.} Let $L$ be as in Step $4$. We claim that for any $0 < \e < \frac12$ fixed there exists $\delta_{\e}$ >0 and $C_{\e} >0$ such that, for any $r \in [\mu_{L,k}, R]$,
\begin{equation}\label{C06}
| u_k(r)| \le C_{\e} \Big( \mu_{L,k}^{\frac{2}{p_k-2}(1-2\e)}r^{-\frac{4(1-\e)}{p_k-2}} + \eta_k(\delta_{\e})  r^{- \frac{4\e}{p_k-2}} \Big). 
\end{equation}

\begin{proof}[Proof of \eqref{C06}]
Let $0 < \e < \frac12$ and $0 < \delta <R$ be fixed. For $r \in [\mu_{L,k}, R]$ we let 
\[ \Phi_k^{\e}(r) = \mu_{L,k}^{\frac{2}{p_k-2}(1-2\e)} G(r)^{\frac{4(1-\e)}{(N-2)(p_k-2)}} + \eta_k(\delta)G(r)^{\frac{4\e}{(N-2)(p_k-2)}},\]
where, for $r = |x|$  we have let $G(r) =  \frac{1}{(N-2) \omega_{N-1}} |x|^{2-N}$, where $\omega_{N-1}$ is the area of $\mathbb{S}^{N-1} \subset \R^N$.  
We let $s_k \in [\mu_{L,k}, R]$ be such that
\[ \frac{u_k(s_k)}{\Phi_k^{\e}(s_k)} = \max_{s \in [\mu_{L,k},R]} \frac{u_k(s)}{\Phi_k^{\e}(s)}.\]
We first assume that, up to a subsequence, 
\begin{equation} \label{step512}
\mu_{L,k} =o( s_k) \quad \text{ and } \quad s_k \le \delta
\end{equation}
 as $k \to + \infty$. Let $x_k \in B_R \backslash B_{\mu_{L,k}}$ be such that $s_k = |x_k|$. By definition of $s_k$ we have 
\begin{equation} \label{step52}
 \frac{- \triangle u_k(x_k)}{u_k(x_k)} \ge \frac{- \triangle \Phi_k^{\e}(x_k)}{ \Phi_k^{\e}(x_k)}.
 \end{equation}
On the one side, straightforward computations show that 
\[  s_k^2 \frac{- \triangle \Phi_k^{\e}(x_k)}{ \Phi_k^{\e}(x_k)} \to (n-2)^2\e(1-\e) >0 \]
as $k \to + \infty$. On the other side, \eqref{LNTk}, \eqref{C04} and \eqref{step512} show, since $s_k \to 0$, that 
\[ s_k^2  \frac{- \triangle u_k(x_k)}{u_k(x_k)} \le  \delta^{2} \Vert u_0 \Vert_{L^\infty(B_R)}^{\frac{4}{N-2}} + o(1) \]
as $k \to + \infty$. If $u_0 \equiv 0$ we have a contradiction. If $u_0 \not \equiv 0$ we can now choose $\delta = \delta_\e$ so that $\delta_\e^{2} \Vert u \Vert_{L^\infty(B_R)}^{\frac{4}{N-2}} \le \frac12 \e(1-\e)$. We then get a contradiction with \eqref{step52}, which shows that \eqref{step512} cannot happen.

Thus, we either have $\frac{s_k}{\mu_{L,k}} \to C \ge 1$ or $s_k \ge \delta_\e$, for some $\delta_\e >0$, up to a subsequence. In the first case \eqref{C06} follows from the second bullet in $(H_L)$, while in the second case \eqref{C06} follows from the definition of $\eta_k(\delta_\e)$. This shows that
\begin{equation*}
 u_k(r) \le C_{\e} \Big( \mu_{L,k}^{\frac{2}{p_k-2}(1-2\e)}r^{-\frac{4(1-\e)}{p_k-2}} + \eta_k(\delta_{\e})  r^{- \frac{4\e}{p_k-2}} \Big) 
\end{equation*}
for any $r \in [\mu_{L,k}, R]$. Since $-u_k$ also satisfies \eqref{energy}, \eqref{blowup} and \eqref{LNTkr} the same inequality remains true also for $-u_k$. Hence \eqref{C06} holds true.
\end{proof}

\medskip

\textbf{Step $5.2$: interpolation between $\tilde{B}_{i,k}$ and $\tilde{B}_{i+1,k}$.}  Let $L$ be as in Step $4$ and $i \in \{1, \dots, L-1\}$. We claim that for any $0 < \e < \frac12$ there exists a constant $C_{\e} >0$ such that, for any $r \in [\mu_{i,k}, \mu_{i+1,k}]$,
\begin{equation}\label{C062}
 u_k(r) \le C_{\e} \Big( \mu_{i,k}^{\frac{2}{p_k-2}(1-2\e)}r^{-(1-\e)\frac{4}{p_k-2}} + \mu_{i+1,k}^{- \frac{2}{p_k-2}}\left( \frac{\mu_{i+1,k}}{r} \right)^{ \frac{4\e}{p_k-2}} \Big). 
\end{equation}

Remark that, when $r= \mu_{i+1,k}$, $u_k$ is of order $\mu_{i+1,k}^{- \frac{2}{p_k-2}}$ by the second bullet in $(H_L)$. Estimate \eqref{C062} is therefore the direct analogue of \eqref{C06} where $\eta_k(\delta)$ has been formally replaced by $\mu_{i+1,k}^{- \frac{2}{p_k-2}}$ (and  $r^{- \frac{4\e}{p_k-2}} $ has been replaced by $ \mu_{i+1,k}^{ \frac{4\e}{p_k-2}} r^{- \frac{4\e}{p_k-2}}$ to attain the value $\mu_{i+1,k}^{- \frac{2}{p_k-2}}$ at the correct scale).

\begin{proof}[Proof of \eqref{C062}]
The proof is very similar to the proof of \eqref{C06}. Let $i \in \{1, \dots, L-1\}$ and $0 < \e < \frac12$ be fixed. For $r \in [\mu_{i,k}, \mu_{i+1,k}]$ we let 
\[ \Phi_{i,k}^{\e}(r) = \mu_{i,k}^{\frac{2}{p_k-2}(1-2\e)} G(r)^{\frac{4(1-\e)}{(N-2)(p_k-2)}} + \mu_{i+1,k}^{- \frac{2(1-2\e)}{p_k-2}} G(r)^{\frac{4\e}{(N-2)(p_k-2)}},\]
where $G(r)$ is as in the proof of \eqref{C06}. We let $s_k \in [\mu_{L,k}, R]$ be such that
\[ \frac{u_k(s_k)}{\Phi_{i,k}^{\e}(s_k)} = \max_{s \in [\mu_{L,k},R]} \frac{u_k(s)}{\Phi_{i,k}^{\e}(s)}.\]
We first assume that, up to a subsequence, $\mu_{i,k} =o( s_k)$ and $s_k = o(\mu_{i+1,k})$ as $k \to + \infty$. Let $x_k$ be such that $s_k = |x_k|$. By definition of $s_k$ we have 
\begin{equation*}
 \frac{- \triangle u_k(x_k)}{u_k(x_k)} \ge \frac{- \triangle \Phi_{i,k}^{\e}(x_k)}{ \Phi_{i,k}^{\e}(x_k)}.
 \end{equation*}
Straightforward computations show again that 
\[  s_k^2 \frac{- \triangle \Phi_{i,k}^{\e}(x_k)}{ \Phi_{i,k}^{\e}(x_k)} \to  (n-2)^2\e(1-\e) >0 \]
as $k \to + \infty$, while \eqref{LNTk} and \eqref{C04} again show that 
\[ s_k^2  \frac{- \triangle u_k(x_k)}{u_k(x_k)} = o(1) \] 
as $k \to + \infty$, a contradiction. Thus, we either have $\frac{s_k}{\mu_{i,k}} \to C \ge 1$ or $\frac{s_k}{\mu_{i+1,k}} \to C \le 1$ up to a subsequence. In both cases  we obtain that 
$$\frac{u_k(s_k)}{\Phi_{i,k}^{\e}(s_k)}  \le C_\e$$
for some positive constant $C_\e$. The inequality on $-u_k$ is obtained as in the proof of \eqref{C06} and \eqref{C062} follows.
\end{proof}

\textbf{Step $6$: improved pointwise estimates.} Let $0 < \delta < \frac{R}{2}$ be fixed. We prove that there exists $C_\delta >0$ such that, for any $k \ge 0$ and $r \in [0,R]$,
\begin{equation} \label{C07}
\big| u_k(r) \big|\le C_\delta \Big( \sum_{i=1}^L \tilde{B}_{i,k}(r) + \eta_k(\delta) \Big)
\end{equation}
holds, where $\tilde{B}_{i,k}$ is as in \eqref{bulles}.

\begin{proof}[Proof of \eqref{C07}]
Let $G_1$ be the Green's function of $- \triangle +1$ in $B_R$ with Neumann boundary condition. Recall that there exists $C>0$ such that $\big| G_1(x,y)  \big|\le C |x-y|^{2-n}$ for any $x,y \in B_R, x \neq y$ (see the arguments in \cite{RobertNeumann}). Let $(s_k)_{k \ge 0}, s_k = |x_k|$, be any sequence of points in $[0, R]$. First, if $s_k \in [R/2,R]$, \eqref{C07} follows from the definition of $\eta_k(\delta)$. We can thus assume that $s_k \in [0,R/2]$ for all $k \ge 0$. A representation formula for $u_k$ at $s_k$ (or $x_k$) then gives, thanks to \eqref{LNTk} and $(H_L)$,
\begin{equation} \label{step61}
\begin{aligned}
 u_k(r_k) & = \int_{B_{\mu_{1,k}}} G_1(x_k, y) |u_k(|y|)|^{p_k-2}u_k(|y|) dy  \\
&+ \sum_{i=2}^L  \int_{B_{\mu_{i,k}} \backslash B_{\mu_{i-1,k}}} G_1(x_k, y) |u_k(|y|)|^{p_k-2}u_k(|y|) dy \\
&+ \int_{B_R \backslash B_{\mu_{L,k}}} G_1(x_k, y) |u_k(|y|)|^{p_k-2}u_k(|y|) dy.
\end{aligned} \end{equation}
Straightforward computations with the second bullet of $(H_L)$ and \eqref{C062} show that 
\begin{equation} \label{step62}
\left| \int_{B_{\mu_{1,k}}} G_1(x_k, y) |u_k(|y|)|^{p_k-2}u_k(y) dy \right| \le C \tilde{B}_{1,k}(s_k)
\end{equation}
(see e.g. \cite{HebeyZLAM}, Proposition $6.1$) and that, for $2 \le i \le L-1$,
\begin{equation} \label{step63}
 \left| \int_{B_{\mu_{i,k}} \backslash B_{\mu_{i-1,k}}} G_1(x_k, y) |u_k(|y|)|^{p_k-2}u_k(y) dy \right| \le C \Big( \tilde{B}_{i-1,k}(s_k) + \tilde{B}_{i,k}(s_k) \Big)
\end{equation}
 for some $C>0$ independent of $k$. Similarly \eqref{C06} shows, with straightforward computations, that 
 \begin{equation} \label{step64}
\left| \int_{B_R \backslash B_{\mu_{L,k}}} G_1(x_k, y) |u_k(|y|)|^{p_k-2}u_k(y) dy \right| \le C \Big( \tilde{B}_{L,k}(s_k) + \eta_k(\delta)^{p_k-1} \Big)
\end{equation}
 for some $C>0$ independent of $k$. Since $\eta_k(\delta)^{p_k-1} \le C \eta_k(\delta)$ by \eqref{C02}, plugging \eqref{step62}, \eqref{step63} and \eqref{step64} in \eqref{step61} proves \eqref{C07}.  
\end{proof}

\medskip

We are now in position to conclude the proof of Proposition \ref{theorieC0}.

\begin{proof}[End of the proof of Proposition \ref{theorieC0}]
Let $0 < \delta < \frac{R}{2}$ be fixed.  We first prove that, if $u_0 \equiv 0$, we have 
\begin{equation} \label{prop1}
\eta_k(\delta) \le C_{\delta}' \mu_{L,k}^{N-2 -\frac{2}{p_k-2}}
\end{equation}
for some $C_{\delta}' >0$. We proceed by contradiction and assume that, up to a subsequence, 
\begin{equation*} 
\eta_k(\delta) >> \mu_{L,k}^{N-2-\frac{2}{p_k-2}}
\end{equation*} 
as $k \to + \infty$, and let $\tilde{u}_k = \frac{u_k}{\eta_k(\delta)}$. We then have $\eta_k(\delta) >> \mu_{i,k}^{N-2-\frac{2}{p_k-2}}$ for any $1 \le i \le L$ and, by \eqref{C07} and the expression of $\tilde{B}_{i,k}$ in \eqref{bulles}, we have
\begin{equation} \label{contratildeuk}
 \Vert \tilde{u}_k \Vert_{L^\infty(B_R \backslash B_\sigma)} \le C_\delta+ o(1) 
 \end{equation}
for any fixed $0 <\sigma <R$. Since $u_0 \equiv 0$ we have, by \eqref{C02}, $\eta_k(\delta) = o(1)$ as $k \to + \infty$. Therefore, by \eqref{LNTk} and standard elliptic theory, $\tilde{u}_k \to \tilde{u}_0$ in $C^2_{loc}(B_R \backslash \{0\})$ where $\tilde{u}_0$ satisfies
\[ \left \{ \begin{aligned}
- \triangle \tilde{u}_0  + \tilde{u}_0 & = 0 & \textrm{ in } B_R \backslash \{0\} \\
\partial_\nu \tilde{u}_0 & = 0 & \textrm{ in } \partial B_R
\end{aligned} \right. \]
Passing \eqref{contratildeuk} to the limit shows that $|\tilde{u}_0| \le C_\delta$ in $B_R \backslash \{0\}$, so the singularity at $0$ is removable. Hence $\tilde{u}_0$ is smooth in $B_R$ and satisfies 
\[ \left \{ \begin{aligned}
- \triangle \tilde{u}_0  + \tilde{u}_0 & = 0 & \textrm{ in } B_R \\
\partial_\nu \tilde{u}_0 & = 0 & \textrm{ in } \partial B_R
\end{aligned} \right. . \]
Integrating by parts yields $\tilde{u}_0 \equiv 0$, which is a contradiction since, by definition, $|\tilde{u}_0(s_0)| = 1$ for some $s_0 \in [\delta, R]$. This proves \eqref{prop1}

\medskip

If now $u_0 \not \equiv 0$ in $B_R$, we have $\eta_k(\delta) \le \Vert u_0 \Vert_{L^\infty(B_R)} + o(1)$ by \eqref{C02}. Coming back to \eqref{C07} with \eqref{prop1} we have thus shown that, whatever $u_0$ is, there exists $C>0$ such that, for any $k \ge 0$, 
\begin{equation} \label{prop2}
\big| u_k(r) \big| \le C \Big( \sum_{i=1}^L \tilde{B}_{i,k}(r) + \Vert u_0 \Vert_{L^\infty(B_R)} + \mu_{L,k}^{N-2-\frac{2}{p_k-2}} \Big)
\end{equation} 
holds for any $r \in [0,R]$. We now improve this estimate into an optimal expansion of $u_k$. We let again $s_k = |x_k|$ be a sequence in $[0,R]$. If $s_k \not \to 0$, then 
\begin{equation} \label{prop21}
 u_k(s_k) = u_0(s_k) + o\big( \Vert u_0 \Vert_{L^\infty(B_R)} \big) + O(\mu_{L,k}^{N-2-\frac{2}{p_k-2}}) 
 \end{equation}
 by \eqref{C02} and \eqref{prop2}, which proves Proposition \ref{theorieC0} when $s_k \not \to 0$. We thus assume that  $s_k =o(1)$ as $k \to + \infty$. Let $0 < \sigma < 1$. Note first that 
\begin{equation} \label{prop3}
 \int_{B_{\frac{\mu_{L,k}}{\sigma}}} G_1(x_k, y) |u_0|^{2^*-2}(y) u_0(|y|) dy = o\big( \Vert u_0 \Vert_{L^\infty(B_R)} \big) 
 \end{equation}
as $k \to + \infty$. This is obvious if $u_0 \equiv 0$ and, if $u_0 \not \equiv 0$, it follows since $\mu_{L,k} = o(1)$. Taking the difference of two representation formulae, one for $u_k$ and one for $u_0$, we then obtain with \eqref{prop3} that
\begin{equation} \label{prop4}
\begin{aligned}
 u_k(s_k) - u_0(s_k) 
 & = \int_{B_{\frac{\mu_{1,k}}{\sigma}}} G_1(x_k, y) |u_k(|y|)|^{p_k-2}u_k(|y|) dy + o\big( \Vert u_0 \Vert_{L^\infty(B_R)} \big)   \\
&+ \sum_{i=2}^L  \int_{B_{\frac{\mu_{i,k}}{\sigma}} \backslash B_{\frac{\mu_{i-1,k}}{\sigma}}} G_1(x_k, y) |u_k(|y|)|^{p_k-2}u_k(|y|) dy \\
&+ \int_{B_R \backslash B_{\frac{\mu_{L,k}}{\sigma}}} G_1(x_k, y) \Big( |u_k(|y|)|^{p_k-2}u_k(|y|)  - |u_0(|y|)|^{2^*-2}u_0(|y|)  \Big) dy.
\end{aligned} \end{equation}
Using \eqref{prop2} and the dominated convergence theorem, straightforward computations show that
\begin{equation} \label{prop5} 
\begin{aligned}
 \int_{B_{\frac{\mu_{1,k}}{\sigma}}} G_1(x_k, y) |u_k(|y|)|^{p_k-2}u_k(|y|) dy & = \big(1+O(\e_\sigma) \big) \tilde{B}_{1,k}(s_k) \\
\end{aligned}
 \end{equation}
and that, for any $1 \le i \le L$,
\begin{equation} \label{prop6}
\begin{aligned}
   &\int_{B_{\frac{\mu_{i,k}}{\sigma}} \backslash B_{\frac{\mu_{i-1,k}}{\sigma}}} G_1(x_k, y) |u_k(|y|)|^{p_k-2}u_k(|y|) dy \\
   & = \int_{B_{\sigma \mu_{i,k}} \backslash B_{\frac{\mu_{i-1,k}}{\sigma}}} G_1(x_k, y) |u_k(|y|)|^{p_k-2}u_k(|y|) dy \\
   & +  \int_{B_{\frac{\mu_{i,k}}{\sigma}} \backslash B_{\sigma \mu_{i,k}}} G_1(x_k, y) |u_k(|y|)|^{p_k-2}u_k(|y|) dy  \\
   & =  \big(1 + O(\e_\sigma) \big) \tilde{B}_{i,k}(s_k) + O \big( \e_\sigma \tilde{B}_{i-1,k}(s_k) \big) + O \big( \mu_{L,k}^{N-2-\frac{2}{p_k-2}}\delta_{iL} \big)
\end{aligned}
\end{equation}
hold as $k \to + \infty$ (see e.g. Proposition $6.1$ and lemmatas $7.6$ and $7.7$  in \cite{HebeyZLAM}). In \eqref{prop5} and \eqref{prop6} $\e_\sigma$ denotes a quantity that satisfies $\lim_{\sigma \to 0} \e_\sigma = 0$ and the notation $\delta_{iL}$ in \eqref{prop6} means that the last term only appears if $i=L$. We now estimate the last integral in \eqref{prop4}. Assume first that $u_0 \not \equiv 0$. Then, using \eqref{C02} and \eqref{prop2}, we have 
\begin{equation} \label{prop7}
\begin{aligned}
 & \int_{B_R \backslash B_{\frac{\mu_{L,k}}{\sigma}}} G_1(x_k, \cdot)  \Big( |u_k|^{p_k-2}u_k  - |u_0|^{2^*-2}u_0  \Big) dy \\
 & =  \int_{B_{\sigma} \backslash B_{\frac{\mu_{L,k}}{\sigma}}} G_1(x_k, \cdot)  \Big( |u_k|^{p_k-2}u_k  - |u_0|^{2^*-2}u_0 \Big) dy + o(\Vert u_0 \Vert_{L^\infty(B_R)}) \\
 & = O \big(\e_\sigma \tilde{B}_{L,k}(s_k) \big) + o(\Vert u_0 \Vert_{L^\infty(B_R)})
\end{aligned}
\end{equation}
as $k\to + \infty$. Assume finally that $u_0 \equiv 0$. Then, with \eqref{prop2}, straightforward computations show that 
 \begin{equation} \label{prop8}
 \int_{B_R \backslash B_{\frac{\mu_{L,k}}{\sigma}}} G_1(x_k, \cdot) \big( |u_k|^{p_k-2} u_k - |u_0|^{2^*-2}u_0  \big)dy = O \big(\e_\sigma \tilde{B}_{L,k}(s_k) \big)
 \end{equation}
as $k \to + \infty$. Plugging \eqref{prop5}, \eqref{prop6}, \eqref{prop7} and \eqref{prop8} in \eqref{prop4} shows that for any $0 < \sigma < 1$ fixed and any $k \ge 0$ large enough we have  
\[ u_k(s_k) = u_0(s_k) + \big( 1 + O(\e_\sigma) \big) \sum_{i=1}^L \tilde{B}_{i,k}(s_k) + O\big(\e_\sigma \Vert u_0 \Vert_{L^\infty(B_R)} \big)+  O(\mu_{L,k}^{N-2-\frac{2}{p_k-2}}) ,\]
where $\lim_{\sigma \to 0} \e_\sigma = 0$. In other words, and up to passing to a subsequence, there is a  constant $C>0$ and a sequence $(\e_k)_{k \ge0}$ of positive numbers  converging to $0$ as $k \to + \infty$ such that 
\[ \Big| u_k(s_k) - u_0(s_k) - \sum_{i=1}^L \tilde{B}_{i,k}(s_k)\Big| \le  \e_k  \Big(  \sum_{i=1}^L \tilde{B}_{i,k}(s_k) + \Vert u_0 \Vert_{L^\infty(B_R)} \Big)+  C\mu_{L,k}^{N-2-\frac{2}{p_k-2}} .\]
That the sequence $(\e_k)_{k \ge0}$ and the constant $C$ can be made independent of the choice of $(s_k)_{k \ge0}$, and thus only dependent on the choice of the subsequence in $(u_k)_{k \ge0}$, follows from standard arguments (see e.g. \cite{HebeyZLAM}, proof of Proposition $7.3$). Together with \eqref{prop21} this concludes the proof of Proposition \ref{theorieC0}.
 \end{proof}
 
Pointwise estimates on the derivatives of $u_k$ are also obtained thanks to Proposition \ref{theorieC0}. Differentiating the representation formula \eqref{prop4} with respect to $s$ and using Proposition \ref{theorieC0} we obtain that there exists
a constant $C>0$ such that, up to passing to a subsequence for $(u_k)_{k\ge0}$,
\begin{equation} \label{estder}
|u_k'(r)| \le C \Bigg( \sum_{i=1}^L \frac{\mu_{i,k}^{N-2- \frac{2}{p_k-2}}}{\big( \mu_{i,k} + r \big)^{N-1}} + \Vert u_0 \Vert_{L^\infty(B_R)} \Bigg) \quad \text{ holds for any } r \in [0,R].
\end{equation}

The result of Proposition \ref{theorieC0} is robust and adapts to more general settings. Let for instance $h_0$ be a smooth radial function in $B_R$ such that $- \triangle + h_0$ has no kernel on $H^1$ functions with Neumann boundary condition. By this we mean that there does not exists any non-zero $u \in H^1(B_R)$ satisfying 
\begin{equation} \label{nokernel}
 \left \{ \begin{aligned}
- \triangle u + h_0 u & = 0 & \textrm{ in } B_R \\
\partial_\nu u & = 0 & \textrm{ in } \partial B_R
\end{aligned} \right.\end{equation}
in a weak sense. Such an operator $- \triangle + h_0$ then admits a uniquely defined radial Green's function $G_{h_0}$ in $B_R$ (see e.g. the arguments in \cite{RobertNeumann}). In particular, it comes with a representation formula as in \eqref{prop4}. A careful inspection of the proof of Proposition \ref{theorieC0} then shows that the exact statement of Proposition \ref{theorieC0} remains true if $- \triangle +1$ in \eqref{LNTk} is replaced by any operator $- \triangle + h_0$ with no kernel, for a smooth radial function $h_0$ in $[0,R]$. 

If $- \triangle + h_0$ has a kernel in the sense of \eqref{nokernel} and $(u_k)_{k \ge0}$ is a sequence of possibly sign-changing radial and $C^2(\overline{B_R})$ solutions of 
\begin{equation*}
\left \{ \begin{aligned}
- \triangle u_k + h_0 u_k & = |u_k|^{p_k-2}u_k & \textrm{ in } B_R \\
\partial_\nu u_k & = 0 & \textrm{ in } \partial B_R
\end{aligned} \right.
\end{equation*} 
satisfying \eqref{energy} and \eqref{blowup}, the proof of Proposition \ref{theorieC0} again adapts and shows that $u_k$ satisfies
\[ \left| u_k(r) - u_0(r) - \sum_{i=1}^L \kappa_i\tilde{B}_{i,k}(r)\right| \le \e_k \Big( 1 + \sum_{i=1}^L \tilde{B}_{i,k}(r) \Big) \]
for some $C>0$, $L \ge 1$ and $L$ sequences $(\mu_{1,k})_{k\ge0}, \dots, (\mu_{L,k})_{k \ge0}$ satisfying $\mu_{i,k} = o(\mu_{i+1,k})$ for $1 \le i \le L-1$ and a family $(\kappa_i)_{1 \le i \le L} \in \{\pm1\}^L$. The kernel of $- \triangle + h_0$ prevents in general to improve the term $\e_k$ in the right-hand side, even when $u_0 \equiv 0$. We refer to \cite{DruetRobertWei} for a detailed analysis in this case and to \cite{Premoselli13} for a general proof of pointwise estimates even in the presence of a kernel.

\section{Necessary conditions for blow-up and proof of Thorem \ref{thstabi-Nge7} and Theorem \ref{thstabi2etoile}} \label{pohosec}

In this section we apply Proposition \ref{theorieC0} to precisely describe the blow-up configurations that radial finite-energy solutions of \eqref{LNTk} can develop. We first provide a complete description of blow-up configurations for positive solutions of \eqref{LNTk}, which allows us to prove Theorems \ref{thstabi2etoile} to \ref{thblowup-intro-supercritical}. We then apply the same techniques to prove Theorem \ref{thstabi-Nge7}. The main tool for this is the classical Pohozaev identity that we now recall. Let $(u_k)_{k \ge0}$ be a sequence of smooth possibly sign-changing solutions of \eqref{LNTk} and let $(\delta_k)_{k \ge0}$ be a sequence of real numbers with $0 < \delta_k \le R$. Multiplying \eqref{LNTkr} by $r u_k'(r) + \frac{N-2}{2} u_k(r)$ and integrating over $B_{\delta_k}$ yields, after a few integration by parts (see e.g. \cite{HebeyZLAM}, proposition $6.2$), the following identity:

\begin{equation} \label{poho}
\begin{aligned}
 \int_{B_{\delta_k}} u_k^2 dx &+ \frac{(N-2)^2}{4N} (p_k-2^*) \int_{B_{\delta_k}} |u_k|^{p_k} dx \\
& = \omega_{N-1} \delta_k^{N-1} \Bigg( -\frac{\delta_k}{2}(u_k'(\delta_k))^2 - \frac{N-2}{2}u_k(\delta_k) u_k'(\delta_k) 
& + \frac{\delta_k}{2} (u_k(\delta_k))^2 - \frac{\delta_k}{p_k} |u_k(\delta_k)|^{p_k}\Bigg).
\end{aligned}
\end{equation}

This identity will be used repeatedly in the following to obtain necessary conditions on the blow-up behaviour of sequences of solutions of \eqref{LNTk}.

\subsection{Classification of positive blow-up configurations}\label{sec:classi}

We first recall the definition of the \emph{mass} of a Green's function. Let $h_0$ be a smooth radial function in $B_R$ such that $- \triangle + h_0$ has no kernel (see \eqref{nokernel}). Let $G_{h_0}$ be the Green's function of $- \triangle + h_0$ with Neumann boundary condition on $B_R$, which is then uniquely defined. The iterative construction of $G_{h_0}$ (see \cite{ReyWei} when $N=3$ but also, in the Dirichlet case and when $N \ge 4$, \cite{LiZhu, Aub, RobDirichlet, RobertNeumann}) show that $x \mapsto G_{h_0}(0,x)$ is radial and expands as 
\begin{equation} \label{greenfunc}
 G_{h_0}(0,r) = \frac{1}{(N-2) \omega_{N-1}} r^{2-N} + \left \{ 
\begin{aligned} &H_{h_0} + O(r^{1- \alpha}) & \textrm{ if } N=3 \\
&  \alpha_4 \ln \frac{ 1}{r} + O(1) & \textrm{ if } N=4 \\
&  \alpha_N r^{4-N} + O(r^{4+\alpha-N}) & \textrm{ if } N\ge 5 \\
 \end{aligned} \right. \end{equation}
for some $0<\alpha<1$, where $H_{h_0} \in \mathbb{R}$ and $\alpha_N \not = 0$ for $N \ge 4$. If $h_0$ vanishes at a certain order at $0$ these expansions can be improved (see \cite{RobertVetois5}). Assume for instance that $h_0(0) = h_0'(0) = 0$. Then 
\begin{equation} \label{greenfunc2}
 G_{h_0}(0,r) = \frac{1}{(N-2) \omega_{N-1}} r^{2-N} + \left \{ 
\begin{aligned} &H_{h_0} + O(r^{1- \alpha}) & \textrm{ if } N=3,4,5 \\
& - c_0 \ln \frac{1}{r} + O(1) & \textrm{ if } N=6 \\
&  \alpha_N r^{6-N} + O(r^{6+\alpha-N}) & \textrm{ if } N\ge 7, \\
 \end{aligned} \right. \end{equation}
where $c_0 = \frac{h_0''(0)}{32 \omega_5}$ and for some constant $\alpha_N \neq 0$. If we have in addition $h_0''(0) = 0$, the expansion in the six-dimensional case also becomes 
\begin{equation*}
 G_{h_0}(0,r) = \frac{1}{4 \omega_{5}} r^{-4} + H_{h_0} + O(r^{1- \alpha})  \quad \textrm{ as } r \to 0 .
 \end{equation*}
We say that $H_{h_0}$ in expansions \eqref{greenfunc} -- \eqref{greenfunc2} is the \emph{mass} of $G$ at $0$. Remark that our sign convention is different than the one in \cite{ReyWei}: using the notations of \cite{ReyWei}, our mass $H_1$ of $G_1$ defined in \eqref{greenfunc} reads as 
\begin{equation} \label{masseRW}
H_1 = - (1+ H_1(0,0))
\end{equation} where $H_1$ is defined in \cite[Equation (1.2)]{ReyWei}. 
 
 \medskip

The following Theorem completely describes the blow-up behaviour for positive radial solutions of \eqref{LNTk}: 

\begin{thm} \label{thblowup}
Let $(u_k)_{k \ge0}$ be a sequence of \emph{positive radial} solutions of \eqref{LNTkr} such that 
\begin{itemize}
\item[(i)] $\Vert u_k \Vert_{L^{\alpha_k}(B_R)} \le C$, for all $k \ge 0$, where $\alpha_k = \max \big( 2^*, \frac{N}{2}p_k - N \big)$,
\item[(ii)] $\lim_{k \to + \infty} \Vert u_k \Vert_{L^\infty(B_R)} = + \infty.$
\end{itemize}
Assume moreover that $u_k \in C^2(\overline{B_R})$ for all $k \ge 0$ if $p_k \ge 2^*$ and denote by $u_0$ its weak limit.

$1)$ Assume that $p_k \le 2^*$ for all $k\ge0$ and $\lim_{k\to +\infty} p_k = 2^*$. Then
\begin{itemize}
\item $u_0 \equiv 0$ for $3 \le N \le 5$ and $u_0(0) \le \frac12$ when $N=6$. 
\item If $N=3$ we have in addition $H_1 \le 0$ where $H_1$ is the mass of $G_1$ at $0$ as in \eqref{greenfunc}. 
\item When $3 \le N \le 6$ blow-up occurs with a single positive bubble at $0$.
\end{itemize}

\smallskip

$2)$ Assume that $p_k = 2^*$ for all $k\ge0$. Then
\begin{itemize}
\item  either  $N=3$ and then $u_0 \equiv0$, $H_1=0$, where $H_1$ is the mass of $G_1$ at $0$ as in \eqref{greenfunc}, and the blow-up occurs with a single positive bubble at $0$.
\item or $N=6$, $u_0(0) = \frac12$ and the blow-up occurs with a single positive bubble at $0$.
\end{itemize}

\smallskip

$3)$ Assume that $p_k \ge 2^*$ for all $k\ge0$ and $\lim_{k\to +\infty} p_k = 2^*$. Then $3 \le N\le 6$.

\begin{itemize}

\item If $4 \le N \le 6$ we have $u_0 >0$ and $u_0(0) \ge \frac12$ when $N=6$. 
\item If $N=3$ we either have $u_0 >0$ or $(u_0\equiv 0$ and $H_1 \ge0)$, where $H_1$ is the mass of $G_1$ at $0$ as in \eqref{greenfunc}. 
\end{itemize}
\end{thm}

The conditions in Theorem \ref{thblowup}, in particular the mass conditions when $N=3$, are well-known and have been highlighted for instance in \cite{ LiZhu,Druetdim3, DruetJDG}. Simplicity of the blow-up (that is the nonexistence of towers of bubbles) is here a consequence of the radiality of the solutions. As is easily seen, Theorem \ref{thblowup} immediately implies Theorems \ref{thblowup-intro-critandsubcrit} and \ref{thblowup-intro-supercritical}. The consequence of the sign of the mass on the admissible values of the radius $R$ in dimension $N=3$ is discussed in Remark \ref{Rem:mass-suite}. 

\medbreak

Before proving Theorem \ref{thblowup}, we show that Theorem \ref{thstabi2etoile} is a consequence of it :

\begin{proof}[Proof of Theorem \ref{thstabi2etoile} assuming Theorem \ref{thblowup}.]
We first assume that $N=6$. We let $(u_k)_{k \ge0}$ be a sequence of positive radial solutions to \eqref{LNTkr} with $p_k = 3$ satisfying the boundedness assumption (i) and assume by contradiction that $(u_k)_{k\ge0}$ blows up. We let $u_0$ be the weak limit of $(u_k)_{k \ge0}$. Theorem \ref{thblowup} shows that $u_k$ blows-up with a single bubble at the origin and that $u_0(0) = \frac12$. Since $u_0$ is radial this completely determines it: in particular by local uniqueness $u_0$ coincides with the restriction to $[0,R]$ of the unique radial positive solution of 
\begin{equation} \label{contrainteu0}
 - u'' - \frac{5}{r}u' + u-u^2 = 0 \text{ in } [0, + \infty) , \quad u'(0) = 0, u(0) = \frac12 
 \end{equation}
 constructed in Theorem $1.6$ of \cite{Ni83}. It is proven in \cite{Ni83} that this solution tends to $1$ as $r \to + \infty$ and oscillates an infinite number of times around the constant solution $1$. We denote by $(R_\ell)_{\ell\ge 1}$ the sequence of the local extrema  of $u_0$ in $(0, + \infty)$, which is a discrete set by local uniqueness. Since $u_0$ solves \eqref{LNTkr} with $p_k=3$ we necessarily have $R = R_\ell$ for some $k \ge 0$, otherwise the Neumann boundary condition is not satisfied. That $\lim_{\ell \to +\infty} R_\ell = +\infty$ and $\liminf_{\ell \to +\infty} (R_\ell - R_{\ell-1}) >0$ also follows from \cite{Ni83}. This concludes the proof of Theorem \ref{thstabi2etoile} when $N=6$.
 
We now consider the case $N=3$. Let $(u_k)_{k \ge0}$ be a sequence of positive radial solutions to \eqref{LNTkr} with $p_k = 6$ satisfying the boundedness assumption (i) and assume by contradiction that $(u_k)_{k\ge0}$ blows up. Theorem \ref{thblowup} then shows that $u_k$ blows up with a zero weak limit and one bubble at the origin and that the mass of $G$, the Green's function of $- \triangle +1$, defined as in \eqref{greenfunc}, vanishes at $0$. For the ball $B(0,R)$ the mass of $G$ can be explicitly computed. A simple scaling argument similar to that of Remark \ref{remscaling} shows that it vanishes if and only if the mass of the Green's function of $- \triangle + R^2$ in the ball $B(0,1) \subset \R^3$ vanishes. The mass of $- \triangle + R^2$ in $B(0,1)$ has been explicitly computed in \cite{MR3938021}, Lemma $2.3$, where it was shown that there exists a unique value of $R$, that we call $R^*$, for which it vanishes. Thus $R=R^*$ which proves Theorem \ref{thstabi2etoile}.
\end{proof}

\begin{rmq}[Critical case : $N=3$ and $R=R^*$]\label{Rem:N=3crit}
In the exactly critical case $p_k = 6$ for all $k \ge 0$ it seems very likely to us that blow-up for finite-energy positive radial solutions of \eqref{LNTkr} will not occur even when $R=R^*$. We believe a rigorous proof of this fact should follow from an adaptation of the techniques developed in \cite{KonigLaurain}. It would consist to go to the next order in the analysis. Since this would be quite lengthly, we leave that point open. 
Note also that our setting here is different from the one in \cite{MR3938021} where perturbations of $R$ were allowed. 
 \end{rmq}

\begin{rmq}[Critical case : mass condition in dimension $N=3$]\label{Rem:mass-suite}
In dimension $N=3$, as previously mentioned, the mass of $G_1$ at $0$ vanishes for one and only one radius $R^*$. In the assertion 1) of Theorem \ref{thblowup}, the conclusion $H_1\le 0$ implies $R\ge R^*$ whereas $H_1\ge 0$ implies $R\le R^*$ in 3). This easily follows from \eqref{masseRW} and the analysis in \cite{ReyWei} (see also \cite{MR3938021}). 
\end{rmq}

\begin{proof}[Proof of Theorem \ref{thblowup}]
Throughout this proof we let $(u_k)_{k\ge0}$ be a sequence of \emph{positive} functions satisfying \eqref{LNTkr}, and the assumptions (i) and (ii). We let $L \ge 1$ be as in Proposition \ref{theorieC0}. 

\medskip

\noindent\textbf{Claim 1: } if $p_k \le 2^*$ (up to a subsequence), then $u_0 \equiv0$ when $3 \le N \le 5$ and $u_0(0) \le \frac12$ when $N=6$ with equality if $p_k = 2^*$ for all $k\ge 0$.

We choose $\delta_k = \sqrt{\mu_{L,k}}$. Straightforward computations using Proposition \ref{theorieC0} and \eqref{estder} show that 
\[ \begin{aligned}
 \omega_{N-1} \delta_k^{N-1} \Bigg( -\frac{\delta_k}{2}(u_k'(\delta_k))^2 - \frac{N-2}{2}u_k(\delta_k) u_k'(\delta_k) + \frac{\delta_k}{2} (u_k(\delta_k))^2 - \frac{\delta_k}{p_k} |u_k(\delta_k)|^{p_k}\Bigg) \\
 = O \Big(\mu_{L,k}^{\frac{N-2}{2} + N-2 - \frac{4}{p_k-2}} \Big),
 \end{aligned} \]
 and that 
 \[ \begin{aligned}
& \int_{B_{\delta_k}} u_k^2 dx  = \left \{ \begin{aligned}& O \big( \mu_{L,k}^{2 - \frac{4}{p_k-2}} \big) & \textrm{ if } N=3 \\ &O \big( \mu_{L,k}^{\frac{7}{2} - \frac{4}{p_k-2}} \big) & \textrm{ if } N=4 \\ &\big(C_1(N) + o(1) \big) \mu_{L,k}^{N - \frac{4}{p_k-2}} & \textrm{ if } N \ge 5, \\ \end{aligned} \right.  
 \end{aligned}\]
 where $C_1(N) = \int_{\R^N} B_0^2dx$ when $N \ge 5$. Independently, Proposition \ref{theorieC0} with Fatou's lemma shows that 
  \[ \int_{B_{\delta_k}} u_k^{p_k}dx \ge  \mu_{L,k}^{N - 2 - \frac{4}{p_k-2}} \big( K_N^{-N} + o(1) \big) \]
as $k \to + \infty$. Since $p_k \le 2^*$ for all $k \ge0$ plugging the latter computations in \eqref{poho} yields
\[ 2^*-p_k =  O \big( \mu_{L,k}^{\frac{N-2}{2}} \big) +  \left \{ \begin{aligned}& O \big( \mu_{L,k} \big) & \textrm{ if } N=3 \\ &O \big( \mu_{L,k}^{\frac{3}{2}} \big) & \textrm{ if } N=4 \\ &O\big(\mu_{L,k}^{2}\big) & \textrm{ if } N \ge 5 \\ \end{aligned} \right. . \]
As a consequence, and since $N-2- \frac{4}{p_k-2} \to 0$ as $k \to + \infty$, we have that
\begin{equation} \label{th411}
 \mu_{L,k}^{N - 2 - \frac{4}{p_k-2}} = 1 +o(1) 
 \end{equation}
as $k \to + \infty$. Let, for $0 \le r \le \frac{R}{\mu_{L,k}}$, $\hat{u}_k(r) =\mu_{L,k}^{\frac{2}{p_k-2} - \frac{N-2}{2}} u_k(\sqrt{\mu_{L,k}} r)$. With \eqref{LNTk}, Proposition \ref{theorieC0}, \eqref{th411} and standard elliptic theory it is easily seen that 
\[ \hat{u}_k(r) \to \hat{u}_\infty(r) = \frac{(N(N-2))^{\frac{N-2}{4}}}{r^{N-2}} + u_0(0) \]
in $C^2_{loc}(]0,1])$ as $k \to + \infty$. Using the latter convergence we can compute again the boundary integral in \eqref{poho}. Scaling by a factor $\sqrt{\mu_{L,k}}$ and using the exact expression of $\hat{u}_\infty$ and \eqref{th411} we have 
\[ \begin{aligned}
 \omega_{N-1} \delta_k^{N-1} \Bigg( -\frac{\delta_k}{2}(u_k'(\delta_k))^2 - \frac{N-2}{2}u_k(\delta_k) u_k'(\delta_k) + \frac{\delta_k}{2} (u_k(\delta_k))^2 - \frac{\delta_k}{p_k} |u_k(\delta_k)|^{p_k}\Bigg) \\
 = \Big( \frac12 (N-2)^{2} (N(N-2))^{\frac{N-2}{4}}\omega_{N-1} u_0(0) + o(1) \Big) \mu_{L,k}^{\frac{N-2}{2}}.
 \end{aligned} \]
It remains to plug this new expression in \eqref{poho}. Using \eqref{th411} and combining it with the previous integral estimates, and since $p_k \le 2^*$, we obtain that 
 \begin{equation*}
  u_0(0) + o(1) = O\big(\mu_{L,k}^{\frac12} \big) \quad \text{ if } 3 \le N \le 5. 
  \end{equation*}
   Hence $u_0(0) =0$ (and thus $u_0 \equiv 0$) if $3 \le N \le 5$. If $N =6$ it is known that $\omega_6 = \frac{16}{15} \pi^3$, so that straightforward computations show that $ C_1(6) = 106 \pi^3$. Since $\omega_5 = \pi^3$, the identity \eqref{poho} now gives, with the two previous estimates and after passing to the limit:
 \begin{equation} \label{valeuru0}
  \begin{aligned}
 & u_0(0) \le \frac12 \quad & \textrm{ if } p_k \le 2^* \text{ for all  } k \ge0 \\
 & u_0(0) = \frac12 \quad & \textrm{ if } p_k =2^* \text{ for all  } k \ge0. \\
 \end{aligned} 
 \end{equation}
 
 \medbreak
 
 \noindent\textbf{Claim 2:}  if $p_k \le 2^*$ (up to a subsequence), then, in dimensions $3 \le N \le 6$,  a blow-up can only occur with a single positive bubble at the origin (regardless of the nature of $u_0$). 
 
 Assume by contradiction that $L \ge 2$ and let 
 \begin{equation} \label{blowup0}
 \delta_k = \sqrt{\mu_{1,k}\mu_{2,k}}.
 \end{equation} Using Proposition \ref{theorieC0} and \eqref{estder} we have that
 \[ \begin{aligned}
 \omega_{N-1} \delta_k^{N-1} \Bigg( -\frac{\delta_k}{2}(u_k'(\delta_k))^2 - \frac{N-2}{2}u_k(\delta_k) u_k'(\delta_k) + \frac{\delta_k}{2} (u_k(\delta_k))^2 - \frac{\delta_k}{p_k} |u_k(\delta_k)|^{p_k}\Bigg) \\
 = O \Big( \mu_{1,k}^{N-2-\frac{4}{p_k-2}} \left( \frac{\mu_{1,k}}{\mu_{2,k}}\right)^{\frac{N-2}{2}} \Big), 
 \end{aligned} \]
 that 
 \begin{equation} \label{blowup1}
  \begin{aligned}
& \int_{B_{\delta_k}} u_k^2 dx  = \left \{ \begin{aligned}& O \big( \mu_{1,k}^{2 - \frac{4}{p_k-2}} \big) & \textrm{ if } N=3 \\ &O \big( \mu_{1,k}^{4- \frac{4}{p_k-2}} \ln \frac{1}{\mu_{1,k}}\big) & \textrm{ if } N=4 \\ &\big(C_1(N) + o(1) \big) \mu_{1,k}^{N - \frac{4}{p_k-2}} & \textrm{ if } N \ge 5, \\ \end{aligned} \right.  
 \end{aligned}
 \end{equation}
 and that 
 \begin{equation} \label{integralecritique}
   \int_{B_{\delta_k}} u_k^{p_k}dx = \mu_{1,k}^{N - 2 - \frac{4}{p_k-2}} \big( K_N^{-N} + o(1) \big) 
   \end{equation}
as $k \to + \infty$. Estimate \eqref{blowup1} implies in particular that 
\[ \int_{B_{\delta_k}} u_k^2 dx =  O \Big( \mu_{1,k}^{N-2-\frac{4}{p_k-2}} \left( \frac{\mu_{1,k}}{\mu_{2,k}}\right)^{\frac{N-2}{2}} \Big) \]
for any $3 \le N \le 6$. Plugging the latter estimates in \eqref{poho} we thus obtain that 
\begin{equation} \label{blowup2}
 p_k -2^* = O \Bigg(  \left( \frac{\mu_{1,k}}{\mu_{2,k}}\right)^{\frac{N-2}{2}} \Bigg)
 \end{equation}
(remark that \eqref{blowup2} does not require the assumption $p_k \le 2^*$). For $0 \le r \le \frac{R}{\delta_k}$ we now let 
 \[ \hat{u}_k(r) = \mu_{1,k}^{\frac{2}{p_k-2} - (N-2)} \delta_k^{N-2} u_k\big( \delta_k r \big) \]
 where $\delta_k$ is still given by \eqref{blowup0}.  With \eqref{blowup2} we now have 
 $$ \left( \frac{\mu_{1,k}}{\mu_{2,k}}\right)^{\frac{N-2}{2} - \frac{2}{p_k-2}} = 1 + o(1) $$
as $k \to + \infty$, so that by Proposition \ref{theorieC0} we have
 \[ \hat{u}_k(r) \to \frac{(N(N-2))^{\frac{N-2}{4}}}{r^{N-2}} + \big( N(N-2) \big)^{\frac{N-2}{4}} \quad \text{ in } C^2_{loc}(]0,1])\]
 as $k \to + \infty$. The boundary term can thus be computed again as 
 \begin{equation} \label{blowup3}
  \begin{aligned}
& \omega_{N-1} \delta_k^{N-1} \Bigg( -\frac{\delta_k}{2}(u_k'(\delta_k))^2 - \frac{N-2}{2}u_k(\delta_k) u_k'(\delta_k) + \frac{\delta_k}{2} (u_k(\delta_k))^2 - \frac{\delta_k}{p_k} |u_k(\delta_k)|^{p_k}\Bigg) \\
& = \Big( \frac12 (N-2)^2 \big( N(N-2) \big)^{\frac{N-2}{4}} \omega_{N-1} + o(1) \Big)\left( \frac{\mu_{1,k}}{\mu_{2,k}}\right)^{\frac{N-2}{2}} \mu_{1,k}^{N-2 - \frac{4}{p_k-2}}.
 \end{aligned} 
 \end{equation}
  The latter estimate in \eqref{poho}, together with \eqref{blowup1}, now shows that 
 \[ \left( \frac{\mu_{1,k}}{\mu_{2,k}}\right)^{\frac{N-2}{2}}  =  \left \{ \begin{aligned}& O \big( \mu_{1,k} \big) & \textrm{ if } N=3 \\ &O \big( \mu_{1,k}^{\frac{3}{2}} \big) & \textrm{ if } N=4 \\ &\big(C_1(N) + o(1) \big) \mu_{1,k}^2 & \textrm{ if } N \ge 5. \\ \end{aligned} \right.  \]
Since $\mu_{2,k}  \to 0$ as $k \to + \infty$ this is a contradiction when $N \le 6$ and the claim is proved.

\medbreak

\noindent\textbf{Claim 3:} if $p_k = 2^*$ for all $k \ge 0$, then a blow-up is not possible when  $N\not \in \{3, 6\}$.

\emph{Assume first that $N=4,5$.} We have already proven in Claim 1 that $u_0 \equiv 0$ and in Claim 2 that $L=1$. We choose $\delta_k = \delta$ for some $0 < \delta < R$ fixed. For this choice of $\delta_k$ it is easily seen, using again Proposition \ref{theorieC0} and \eqref{estder}, that we have 
 \[ \begin{aligned}
 \omega_{N-1} \delta^{N-1} \Bigg( -\frac{\delta}{2}(u_k'(\delta))^2 - \frac{N-2}{2}u_k(\delta) u_k'(\delta) + \frac{\delta}{2} (u_k(\delta))^2 - \frac{\delta}{p_k} |u_k(\delta)|^{p_k}\Bigg) 
= O \big( \mu_{1,k}^{N-2 } \big)
 \end{aligned} \]
 and
 \[ \begin{aligned}
& \int_{B_{\delta_k}} u_k^2 dx  = \left \{ \begin{aligned} 
& 64 \omega_3 \mu_{1,k}^{2}\ln \frac{1}{\mu_{1,k}} + o \big(  \mu_{1,k}^{2} \ln \frac{1}{\mu_{1,k}}\big) & \textrm{ if } N=4 \\ &\big(C_1(5) + o(1) \big) \mu_{1,k}^{2 } & \textrm{ if } N = 5. \\ \end{aligned} \right.  
 \end{aligned}\]
Plugging the latter estimates in \eqref{poho} shows that 
\[ \begin{aligned}
& 64 \omega_3 \mu_{1,k}^{2}\ln \frac{1}{\mu_{1,k}} + o \big(  \mu_{1,k}^{2 }\ln \frac{1}{\mu_{1,k}}\big) = O \big(  \mu_{1,k}^{2} \big) \quad \textrm{ if } N=4 \\
&\big(C_1(5) + o(1) \big) \mu_{1,k}^2 = O(\mu_{1,k}^3) \quad \textrm{ if } N=5, 
 \end{aligned} \]
a contradiction in both cases. 

\medbreak

\emph{Consider now the dimensions $N \ge 7$.} In this setting, it is possible that $u_0 \not \equiv 0$ and $L >1$. We choose $\delta_k = \sqrt{\mu_{L,k}}$. With Proposition \ref{theorieC0} and \eqref{estder} we now have 
 \[ \begin{aligned}
 \omega_{N-1} \delta_k^{N-1} \Bigg( -\frac{\delta_k}{2}(u_k'(\delta_k))^2 - \frac{N-2}{2}u_k(\delta_k) u_k'(\delta_k) + \frac{\delta_k}{2} (u_k(\delta_k))^2 - \frac{\delta_k}{p_k} |u_k(\delta_k)|^{p_k}\Bigg) 
= O \big( \mu_{L,k}^{\frac{N-2}{2} } \big).
\end{aligned} \]
Straightforward computations using Proposition \ref{theorieC0} again show that 
 \[  \int_{B_{\delta_k}} u_k^2 dx  = \big(C_1(N) + o(1) \big) \mu_{L,k}^{2 }  \textrm{ if } N \ge 7. \]
 Plugging the latter two identities in \eqref{poho} then shows that 
 \[  \big(C_1(N) + o(1) \big) \mu_{L,k}^{2 }  = O( \mu_{L,k}^{\frac{N-2}{2}}), \]
 a contradiction since $N \ge 7$.

\medskip

\noindent\textbf{Claim 4:} if $N=3$ and $u_0=0$, then $H_1 = 0$ if $p_k = 2^*$ for all $k\ge 0$ whereas $H_1 \stackrel{\le}{\ge} 0$  if $p_k  \stackrel{\le}{\ge}  2^*$ for all $k \ge 0$ .

If $p_k \le 2^*$ we have in fact proven that $u_0 \equiv 0$ and $L=1$. If $p_k \ge 2^*$ for all $k \ge 0$, one can have $u_0>0$ and it is an assumption here that $u_0 \equiv 0$.
Since $u_0 \equiv 0$, by \eqref{LNTk} and Proposition \ref{theorieC0} we have 
\[ \mu_{L,k}^{\frac{2}{p_k-2}-1} u_k(r) \to G_1(r) \quad \text{ in } C^1_{loc}(]0,R]), \]
where $G_1$ is as in \eqref{greenfunc} and where we already know that $L=1$ if $p_k \le 2^*$ for all $k\ge0$. Choose again $\delta_k = \delta$ for some $0 < \delta < R$ fixed. Straightforward computations with \eqref{greenfunc} show that
 \[ \begin{aligned}
 \omega_{N-1} \delta^{N-1} \Bigg( -\frac{\delta}{2}(u_k'(\delta))^2 - \frac{N-2}{2}u_k(\delta) u_k'(\delta) + \frac{\delta}{2} (u_k(\delta))^2 - \frac{\delta}{p_k} |u_k(\delta)|^{p_k}\Bigg) \\
=\Big( \frac{1}{2} 3^{\frac14}\omega_2 H_1 + O(\delta) + o(1) \Big)  \mu_{L,k}^{1 - \frac{4}{p_k-2}},
 \end{aligned} \]
 where $H_1$ is as in \eqref{greenfunc}.  Independently, we have by Proposition \ref{theorieC0} that
 \[ \int_{B_\delta} u_k^2 dx = O \big( \delta \mu_{L,k}^{2 - \frac{4}{p_k-2}}  \big) + o\big( \mu_{L,k}^{2 - \frac{4}{p_k-2}}\big) \]
and that, up to passing to a subsequence, that
\begin{equation} \label{integralecritique2}
 \int_{B_\delta} u_k^{p_k} dx = \big( C_2 + o(1) \big) \mu_{L,k}^{1- \frac{4}{p_k-2}} 
 \end{equation}
for some $C_2 >0$. The latter equality is straightforward when $p_k \le 2^*$ since $L=1$. If $p_k > 2^*$ it is obtained, up to passing to a subsequence, as a double inequality. On the one side a bound from below follows from Fatou's theorem, and on the other side a bound from above is obtained from Proposition \ref{theorieC0} and the estimate 
\[ \left( \frac{\mu_{i,k}}{\mu_{L,k}} \right)^{N - \frac{2p_k}{p_k-2}} = O(1)\]
for all $1 \le i \le L-1$, that follows since  $N - \frac{2p_k}{p_k-2} \ge0$ when $p_k \ge 2^*$. Note that estimate \eqref{integralecritique2} is different than \eqref{integralecritique}: in \eqref{integralecritique2} we are integrating in $B_\delta$ -- and not on $B_{\sqrt{\mu_{1,k}\mu_{2,k}}}$ as before -- hence possible contributions from all the bubbles have to be taken into account. Plugging the latter estimates in \eqref{poho} finally gives 
\[ \big(H_1 + O(\delta) + o(1) \big)  \mu_{L,k} = \big(C + o(1) \big) (p_k-2^*) \]
for some positive constant $C$. If $p_k = 2^*$ for all $k$ this shows, after passing to the limit $\delta \to 0$, that $H_1 = 0$. If $p_k \le 2^*$ for all $k \ge 0$  we similarly get that $H_1 \le 0$ and, if $p_k \ge 2^*$ for all $k \ge0$, we obtain that $H_1 \ge0$. 

\medbreak

With Claim 1 to Claim 4, we are done with the assertions concerning $p_k \le 2^*$ and $p_k = 2^*$ for all $k \ge 0$. We have also completed the proof of the case $p_k \ge 2^*$ when $N =3$.

\medskip

\noindent\textbf{Claim 5:} if $p_k \ge 2^*$ for all $k \ge 0$ and  $4 \le N \le 6$, then $u_0 >0$ and $u_0(0) \ge \frac12$ when $N=6$. 

We proceed by contradiction and assume that $u_0 \equiv 0$. Choose again $\delta_k = \delta$ for some $0 < \delta < R$. Proposition \ref{theorieC0} and \eqref{estder} show that 
 \[ \begin{aligned}
 \omega_{N-1} \delta^{N-1} \Bigg( -\frac{\delta}{2}(u_k'(\delta))^2 - \frac{N-2}{2}u_k(\delta) u_k'(\delta) + \frac{\delta}{2} (u_k(\delta))^2 - \frac{\delta}{p_k} |u_k(\delta)|^{p_k}\Bigg)
= O \big( \mu_{L,k}^{2(N-2) - \frac{4}{p_k-2}} \big)
 \end{aligned} \]
and that 
 \[ \begin{aligned}
& \int_{B_{\delta_k}} u_k^2 dx  \ge \left \{ \begin{aligned} 
& 64 \omega_3 \mu_{L,k}^{4 - \frac{4}{p_k-2}}\ln \frac{1}{\mu_{L,k}} + o \big(  \mu_{L,k}^{4 - \frac{4}{p_k-2}}\ln \frac{1}{\mu_{L,k}}\big) & \textrm{ if } N=4 \\ &\big(C_1(N) + o(1) \big) \mu_{L,k}^{N - \frac{4}{p_k-2}} & \textrm{ if } N = 5,6. \\ \end{aligned} \right.  
 \end{aligned}\]
Since $p_k \ge 2^*$ for all $k \ge 0$, we infer from the Pohozaev identity \eqref{poho} that  
\[ \begin{aligned}
& \ln \frac{1}{\mu_{1,k}} + o \big(\ln \frac{1}{\mu_{1,k}}\big) = O \big( 1 \big) \quad \textrm{ if } N=4 \\
&\big(C_1(N) + o(1) \big) \mu_{L,k}^2 = O(\mu_{1,k}^{N-2}) \quad \textrm{ if } N=5,6, 
\end{aligned} \]
and this yields a contradiction. Whence $u_0 >0$. Finally, the same arguments that led to \eqref{valeuru0} similarly show that $u_0(0) \ge \frac12$ when $p_k \ge 2^*$. 

\medbreak

This concludes the proof of Theorem \ref{thblowup}.
\end{proof}

In the next section, we construct  examples of blowing up sequences of positive solutions in all the cases mentioned in Theorem \ref{thblowup} when $N \ge 4$. Unless specified otherwise, towers of positive bubbles at the origin (corresponding to $L\ge2$ in Proposition \ref{theorieC0}) are likely to occur. In this case the identity \eqref{poho} also allows us to obtain necessary conditions on the blow-up. Assume that $(u_k)_{k \ge0}$ is a sequence of radial positive solutions of \eqref{LNTk} satisfying \eqref{energy} and \eqref{blowup}. We claim for instance that the following holds:

\begin{lem} \label{bubbletowers}
Assume that $p_k \ge 2^*$ for all $k \ge 0$, $3 \le N \le 6$ and that $(u_k)_{k \ge0}$ blows up with at least two bubbles, that is $L \ge2$. Then, for any $1 \le i \le L-1$, there exists a constant $C_i >0$ such that
\[ \left( \frac{\mu_{i,k}}{\mu_{i+1,k}} \right)^{\frac{N-2}{2}}= \big(C_i + o(1) \big)(p_k-2^*) .\]
\end{lem}

\begin{proof}
We first prove that there exists a constant $C_1 >0$ such that 
\begin{equation} \label{tower}
p_k-2^* = \big(C_1+o(1) \big) \left( \frac{\mu_{1,k}}{\mu_{2,k}} \right)^{\frac{N-2}{2}}. 
\end{equation}
Assume by contradiction that $p_k-2^* = o \big( (\mu_{1,k}\mu_{2,k}^{-1})^{\frac{N-2}{2}} \big)$. This implies in particular that 
$$ \left( \frac{\mu_{1,k}}{\mu_{2,k}} \right)^{\frac{N-2}{2}- \frac{2}{p_k-2}} = 1+ o(1) $$
as $k \to + \infty$. With the latter estimates \eqref{blowup1}, \eqref{blowup2} and \eqref{blowup3} obtained in the proof of Theorem \ref{thblowup} are still true (since they did not use the assumption $p_k \le 2^*$). Then plugging \eqref{blowup3} and \eqref{blowup1} in \eqref{poho} gives
\[ \big(1+o(1) \big) \left( \frac{\mu_{1,k}}{\mu_{2,k}} \right)^{\frac{N-2}{2}} = \left \{ \begin{aligned}& O \big( \mu_{1,k} \big) & \textrm{ if } N=3 \\ &O \big( \mu_{1,k}^{2}\ln \frac{1}{\mu_{1,k}} \big) & \textrm{ if } N=4 \\ &\big(C_1(N) + o(1) \big) \mu_{1,k}^{2} & \textrm{ if } N = 5,6.  \\ \end{aligned} \right.  \]
Since $\mu_{2,k} \to 0$ as $k \to + \infty$ this is a contradiction. This proves that $p_k-2^* \gtrsim \big( (\mu_{1,k}\mu_{2,k}^{-1})^{\frac{N-2}{2}} \big)$. The other inequality is given by \eqref{blowup2}, and this proves \eqref{tower} up to passing to a subsequence. If now $i \in \{1, \dots, L-1\}$, we apply a Pohozaev identity on $B_{r_{i,k}}$, where $r_{i,k} = \sqrt{\mu_{i,k}\mu_{i+1,k}}$. Estimates \eqref{blowup1}, \eqref{blowup2} and \eqref{blowup3} remain true with $\mu_{1,k}$ and $\mu_{2,k}$ replaced by $\mu_{i,k}$ and $\mu_{i+1,k}$ respectively. The proof then follows the same lines as the proof of \eqref{tower}.
\end{proof}

\subsection{Proof of Theorem \ref{thstabi-Nge7}} 

\begin{proof}[Proof of Theorem \ref{thstabi-Nge7}]
We assume that $N \ge 7$, $p_k \ge 2^*$ for all $k \ge0$ and $\lim_{k \to + \infty} p_k = 2^*$, and we let $(u_k)_{k \ge0}$ be a sequence of solutions of \eqref{LNTkr} satisfying \eqref{energy}. As mentioned before, by the result of \cite{Ni83}, radial solutions of \eqref{LNTkr} when $p_k \ge 2^*$ have fixed sign. We will however not use this observation in this proof but instead rely on Proposition \ref{theorieC0} and   \eqref{poho}, that are independent of the sign of the solutions. We proceed again by contradiction and assume that \eqref{blowup} holds. We will again conclude by using the Pohozaev identity \eqref{poho}. Let $L \ge 1$ be the integer given by Proposition \ref{theorieC0} and let $\delta_k = \sqrt{\mu_{L,k}}$. Straightforward computations using Proposition \ref{theorieC0} and \eqref{estder} show that 
\[ \begin{aligned}
& \int_{B_{\delta_k}} u_k^2 dx  \ge \big(  C_1(N) + o(1) \big) \mu_{L,k}^{N - \frac{4}{p_k-2}} 
 \end{aligned}\]
for any $N \ge 7$, for some positive constant $C_1(N)$ only depending on $N$ and 
\[ \begin{aligned}
 \omega_{N-1} \delta_k^{N-1} \Bigg( -\frac{\delta_k}{2}(u_k'(\delta_k))^2 - \frac{N-2}{2}u_k(\delta_k) u_k'(\delta_k) + \frac{\delta_k}{2} (u_k(\delta_k))^2 - \frac{\delta_k}{p_k} |u_k(\delta_k)|^{p_k}\Bigg) \\
 = O\big( \mu_{L,k}^{\frac{N-2}{2} + N-2 - \frac{4}{p_k-2}} \big).
 \end{aligned} \]
Since $p_k \ge 2^*$, plugging these estimates in \eqref{poho} gives,
\[ C_1(N) \mu_{L,k}^2 \le O \big( \mu_{L,k}^{\frac{N-2}{2}} \big). \]
Since $N \ge 7$, this is a contradiction so that \eqref{blowup} cannot occur. The sequence $(u_k)_{k \ge0}$ is thus uniformly bounded in $L^\infty(B_R)$ and converges, up to a subsequence, in $C^2(\overline{B_R})$ by standard elliptic theory. This proves Theorem \ref{thstabi-Nge7}.
\end{proof}

\medskip

We conclude this subsection by showing that if an additional assumption is added we can improve the conclusion of Theorem \ref{thstabi2etoile} when $N=6$.

\begin{lem} \label{lemme6}
Assume that $N=6$ and $p_k=3$ for all $k \ge 0$. Let $(u_k)_{k\ge0}$ be a sequence of radial solutions of \eqref{LNTkr} that satisfies \eqref{energy} and let $u_0$ be its weak limit in $H^1(B_R)$. If $u_0$ is non-degenerate in the sense of \eqref{nondege} then $(u_k)_{k \ge0}$ is bounded in $L^\infty(B_R)$ (and hence precompact in $C^2(B_R)$).
\end{lem}

\begin{proof}
By Theorem \ref{thblowup} we know that  $u_0$ is positive, satisfies $u_0(0) = \frac12$ and the blow-up arises with a single bubble at the origin. For any $k \ge 0$, we let $v_k = u_k - u_0$. By \eqref{LNTk} $v_k$ satisfies
\begin{equation} \label{eqvk}
 - \triangle v_k + h_0 v_k = v_k^2 \quad \text{ in } B_R, 
 \end{equation}
where we have let $h_0 = 1 - 2u_0$. As a consequence of the non-degeneracy of $u_0$ (see \eqref{nondege}) it is easily seen that $- \triangle + h_0$ has no kernel in the sense of \eqref{nokernel}. Hence $G_{h_0}$, the Green's function of $- \triangle + h_0$ with Neumann boundary condition, is uniquely defined. Since $h_0(0) = 0$ and $h_0'(0) = 0$, $G_{h_0}$ satisfies \eqref{greenfunc2}.

We will need a refinement of the Pohozaev identity for \eqref{eqvk}. Let $0 < \delta < R$ be fixed. Integrating \eqref{eqvk} against $rv_k'(r)$ gives, after a few integration by parts (see again \cite{HebeyZLAM}, Proposition $6.2$), that for any $k \ge 0$
\begin{equation} \label{poho6}
\begin{aligned}
& \int_{B_{\delta}} \big( h_0 + \frac12 r h_0'(r) \big) v_k^2 dx 
& = \omega_{5} \delta^{5} \Bigg( -\frac{\delta}{2}(v_k'(\delta))^2 - 2 v_k(\delta) v_k'(\delta) + \delta \frac{h_0(\delta)}{2} (v_k(\delta))^2 - \frac{\delta}{3} v_k(\delta)^{3}\Bigg)
\end{aligned}
\end{equation}
holds. By Proposition \ref{theorieC0} applied to $u_k$ and by definition of $v_k$ we have 
\begin{equation} \label{estvk1}
 v_k(r) = \big( 1 + o(1) \big) B_{1,k}(r) + O(\mu_{1,k}^2) + o(\Vert u_0 \Vert_{L^\infty(B_R)}) 
 \end{equation}
as $k \to + \infty$, for any $r \in [0,R]$, where $B_{1,k}$ is as in \eqref{bulles}. We claim that we actually have 
\begin{equation} \label{estvk2}
 v_k(r) = \big( 1 + o(1) \big) B_{1,k}(r) + O(\mu_{1,k}^2) 
 \end{equation}
for any $r \in [0,R]$. This follows from an adaptation of the arguments in the proof of Proposition \ref{theorieC0}. More precisely, it is easily seen with \eqref{estvk1} that \eqref{C06} still holds for $v_k$ with $L=1$ and $\eta_k(\delta_\e)$ now defined as $\eta_k(\delta_\e) =\max_{[\delta_\e, R]} |v_k|$. As a consequence, \eqref{C07} remains true and is proven in the same way. Since $- \triangle + h_0$ has no kernel the same arguments as in the proof of \eqref{prop1} then show that $\eta_k(\delta_\e) = O(\mu_{1,k}^2)$, from which \eqref{estvk2} follows.

With \eqref{eqvk} and \eqref{estvk2} we then get that $\mu_{1,k}^{-2} v_k(r) \to G_{h_0}(r)$ in $C^2_{loc}(]0,\delta])$ as $k \to + \infty$. Using \eqref{greenfunc2} we then have 
\[ \begin{aligned}
\mu_{1,k}^{-4} \omega_{5} \delta^{5} \Bigg( -\frac{\delta}{2}(v_k'(\delta))^2 - 2 v_k(\delta) v_k'(\delta) + \delta \frac{h_0}{2} (v_k(\delta))^2 - \frac{\delta}{3} v_k(\delta)^{3}\Bigg) 
\to 2 C_0 \ln \delta + O(1) \quad \text{ as } k \to + \infty,
\end{aligned} \]
where $O(1)$ denotes a quantity that remains bounded as $\delta \to 0$. (Note however that $\delta$ is fixed in this proof). Independently, and by Proposition \ref{theorieC0}, we have
\[  \int_{B_{\delta}} \big( h_0 + \frac12 r h_0'(r) \big) v_k^2 dx  =(24)^5 \omega_5 h_0''(0) \mu_{1,k}^4 \ln \frac{1}{\mu_{1,k}} + o \big(\mu_{1,k}^4 \ln \frac{1}{\mu_{1,k}}  \big)   \]
as $k \to + \infty$. Plugging the latter two estimates in \eqref{poho6} shows that 
\[ h_0''(0) = O \Big( \big(\ln \frac{1}{\mu_{1,k}} \big)^{-1} \Big) \]
as $k \to + \infty$, hence $h_0''(0) = 0$. Since $h_0 = 1 - 2 u_0$ we thus have $u_0'(0) = u_0''(0) = 0$. However, applying \eqref{LNTkr} at $r=0$ with $u_0(0) = \frac12$ and $u_0'(0) = 0$ yields $u_0''(0) = \frac{1}{24}$, a contradiction.
\end{proof}

Lemma \ref{lemme6} provides an additional precompactness results in case the non-degeneracy assumption is satisfied. The peculiarity of the six-dimensional case for critical equations like \eqref{LNT} has been known since \cite{DruetJDG} (see also \cite{HebeyZLAM}). Theorem \ref{thstabi2etoile} and Lemma \ref{lemme6} are among the few precompactness results available in dimension $6$ for equations like \eqref{LNTk}, and they entirely rely on the assumption that the solutions considered are radial. Note that the non-degeneracy of $u_0$ is a generic assumption in the choice of $R$, see Proposition \ref{nondegu0} below. This genericity is however not enough to close Theorem \ref{thstabi2etoile}.

\begin{rmq}[Critical case, $N=6$]\label{Rem:N=6crit}
Theorem \ref{thstabi2etoile} states that precompactness for finite-energy solutions of \eqref{LNTk} when $N=6$ and $p_k = 3$ for all $k \ge 0$ holds at least if $R \not \in \{R_\ell\}_{\ell \ge 1}$ for a given discrete sequence $(R_\ell)_{\ell\ge 1}$. If $R = R_\ell$ for some $k \ge 0$, as shown in the proof of Theorem \ref{thstabi2etoile}, every finite-energy sequence $(u_k)_{k \ge0}$ of \eqref{LNTk} has the same weak limit which is given by \eqref{contrainteu0}. Lemma \ref{lemme6} then allows to completely prove Conjecture \ref{conjstabi2etoile} in dimension $N=6$ if one can show $u_0$ is non-degenerate in the sense of \eqref{nondege} when $R = R_\ell$. 
The solution $u_0$ is non-degenerate (for $R = R_\ell$) if the unique solution $v_0$ of 
\begin{equation}\label{eq:lin}
 - v'' - \frac{5}{r}v' + v-2 u_0 v = 0 \text{ in } [0, + \infty) , \quad v'(0) = 0, v(0) = 1 
 \end{equation}
is such that $v'(R_\ell)\ne 0$. Indeed this linear equation has a two-dimensional basis of solutions and $v_0$ is the only one which can span an element of the kernel of the linearization at $u_0$. 
Numerics provide an evidence that $u_0$ is non-degenerate, see Figure \ref{fig-nondege}.
Proving this fact theoretically seems quite tricky and very sensitive on the parameters involved. For instance, degenerate radial solutions have been highlighted in \cite{MR3564729} for other values of the parameters (dimension and power).
\begin{figure}[h!t]\label{fig-nondege}
\begin{center}
{\includegraphics[height=6cm]{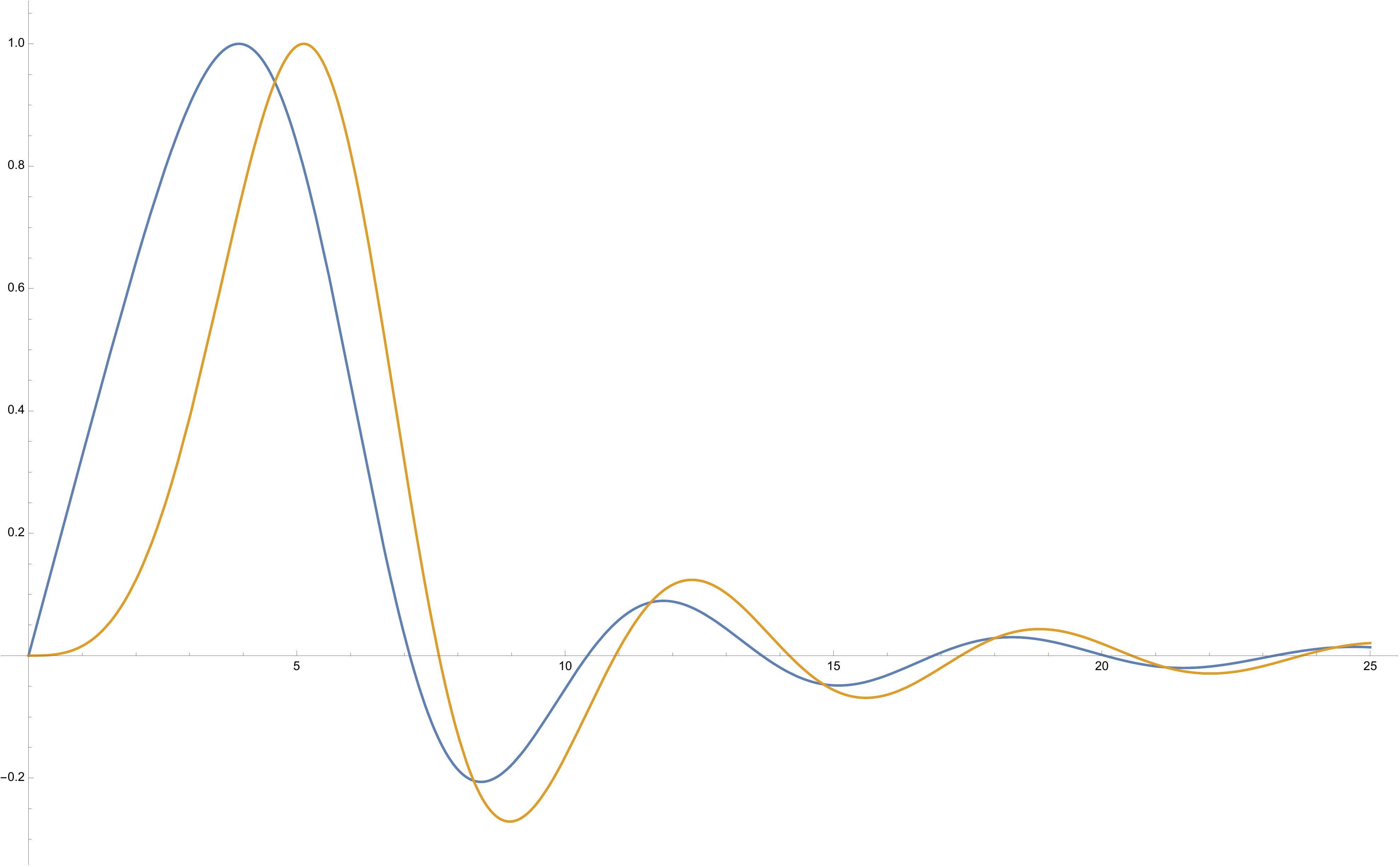}}
\caption{The derivative of (the numerically computed) $u_0$ is plotted in blue while the derivative of the (numerical) solution $v_0$ of \eqref{eq:lin} is in orange. Both curves have been normalized to fit in the same plot (since we are only interested by their roots). 
It seems clear, at least for the first zeros, that the derivative of $u_0$ and $v_0$ do not vanish at the same points.}
\end{center}
\end{figure}
\end{rmq}

\section{Lyapunov-Schmidt construction} \label{LS}

The goal of this section is to prove Theorems \ref{thinstabi-typeB} to \ref{thinstabi}. For clarity, we will change the notations with respect to the previous sections. For $\e \neq 0, |\e| < 1$, we consider solutions of the Neumann problem
\begin{equation}
\label{eq}
\left \{ \begin{aligned} -\Delta u +u & = u^{\frac{N+2}{N-2}-\varepsilon} \ \text{in}\ B_R \subset \R^N,\\ u&>0 ,\\ \partial_\nu u & =0\ \text{on}\ \partial B_R. \end{aligned} \right.
\end{equation}
If $(\e_k)_{k \ge0}$ is a sequence of real numbers going to $0$ as $k \to + \infty$, a solution of \eqref{eq} with $\e = \e_k$ is a positive solution of \eqref{LNTk} with $p_k = 2^*-\e_k$. We let $u_0\in C^{2,\theta}(\overline{B_R})$, $\theta \in (0,1)$, be a nondegenerate radial solution of \eqref{eq} with $\e=0$. Solutions of type $u_0+B$ or $B$ are defined as in \eqref{u0B}.

We will construct blowing up famillies $(u_\e)_{\e >0}$ of radial solutions of \eqref{eq} in the following cases: 
\begin{enumerate}
\item (subcritical blow-up) for $\e\to 0$ with $\e >0$, 
\begin{itemize}
\item solutions of type $B$ for any $N \ge 4$ as announced in Theorem \ref{thinstabi-typeB}.
\item solutions of type $u_0 + B$ as announced in Theorem \ref{thinstabi_u0+Bsub} under the following assumptions: 

$\diamond$ $N \ge 7$ or 

$\diamond$ $N=6$ and $u_0(0)< \frac12$. 
\end{itemize}
\item (supercritical blow-up) for  $\e\to 0$ with $\e < 0$, solutions of type $u_0 + B$ as announced in Theorem \ref{thinstabi} under the following assumptions: 

$\diamond$ $3\leq N \le 5$ or 

$\diamond$ $N=6$ and $u_0(0)> \frac12$.
\end{enumerate}
In view of Theorem \ref{thstabi2etoile} to \ref{thblowup-intro-supercritical}, this covers all the cases where radial blow-up can occur, except the case $N=3$ for which we refer to Remark \ref{remN3} below. Also, we do not address the existence of towers of bubbles. 
As proven in Theorem \ref{thblowup}, the assumption $u_0 >0$ is necessary when $\e <0$.

\subsection{Preliminaries and notation}

\medskip

In all the following, we will denote by $\left\langle .,.\right\rangle$ the scalar product given, for $u,v\in H^1 (B_R)$, by
$$\left\langle u,v\right\rangle= \int_{B_R} (\nabla u  \cdot\nabla v+uv) dx,$$
where $\cdot$ denotes the scalar product in $\R^n$, and by $\Vert \cdot\Vert$ the associated norm in $H^1(B_R)$.
We denote by $i^\ast : L^{\frac{2N}{N+2}}(B_R)\rightarrow H^1(B_R)$ the adjoint operator of the embedding $i: H^1(B_R)\rightarrow L^{\frac{2N}{N-2}}(B_R)$, 
i.e. for all $\varphi \in L^{\frac{2N}{N+2}}(B_R)$, the function $u=i^\ast (\varphi)\in H^1(B_R)$ is the unique solution 
of 
\begin{equation*}
\left \{ \begin{aligned} -\Delta u +u & = \varphi \ \text{in}\ B_R \subset \R^N,\\ \partial_\nu u & =0,\ \text{on}\ \partial B_R. \end{aligned} \right.
\end{equation*}
In particular, $\Vert u \Vert \le C \Vert \varphi \Vert_{L^{\frac{2N}{N+2}}(B_R)}$ for some $C = C(N,R) >0$. In view of equation \eqref{eq}, we consider the following problem:
find a positive $u\in H^1(B_R)$ such that
$$u=i^\ast(f_\varepsilon (u)),$$
where we have let, for $u \in \R$, 
$$f_\varepsilon (u)=(u)_+^{2^\ast -1-\varepsilon}$$
and $u_+ = \max(u,0)$. For $s> \frac{2N}{N-2}$ and by definition of $i$ we have, by standard elliptic regularity results,
\begin{equation} \label{esterreur}
\|i^\ast (w)\| + \| i^{\ast}(w) \|_{L^s(B_R)} \leq C \|w\|_{L^{\frac{Ns}{N+2s}}(B_R)}
\end{equation}
for some $C = C(N,R,s) >0$.  For $\varepsilon$ small enough, we define
$$\alpha_\varepsilon = \begin{cases} 2^\ast - \frac{N}{ 2}\varepsilon,& \text{if}\ \varepsilon <0,\\ 2^\ast , & \text{if}\ \varepsilon >0. \end{cases}$$
This exponent coincides with the exponent $\alpha_k$ defined by \eqref{energy} since $p_k = 2^*-\e$ in this section. We let $\mathcal{H}_\varepsilon = H^1 (B_R) \cap L^{\alpha_\varepsilon} (B_R)$ endowed with the norm
$$\|w\|_{\alpha_\e} =\|w\| +\|w\|_{L^{\alpha_\varepsilon}(B_R)},$$
which is a Banach space. It is known (see \cite{CaGiSp}, \cite{Obata}) that all the positive solutions $u$  of the equation
$$-\Delta u=u^{2^\ast-1} \text{ in } \R^N$$
are given by
$$B_{\delta, y}(x)=\delta^{\frac{2-N}{2}}B_0 \Big(\frac{x-y}{\delta} \Big),\ \delta >0,\ y\in \R^n$$
where $B_0$ is as in \eqref{B0}. It is also well known (see \cite{BianchiEgnell}) that all $v\in \dot{H}^1(\R^N)$ 
solving
$$-\Delta v= (2^\ast -1)B_0^{2^\ast-2}v$$
are linear combinations of
$$V_0(x)= \dfrac{|x|^2-1}{(1+|x|^2)^{\frac{N}{2}}}$$
and
$$V_i(x)= \dfrac{x_i}{(1+|x|^2)^{\frac{N}{2}}},\ i=1,\ldots ,N. $$
If $v$ is in addition assumed to be radial it is then a multiple of $V_0$.

\medskip

From now on we will only consider \emph{radial} solutions of \eqref{eq}. Let $\chi : \R\rightarrow \R$ be a smooth cutoff function such that $0\leq \chi\leq 1$, $\chi(x)=1$ if $|x|\leq  \frac{R}{4}$ and $\chi(x)=0$ if $|x|\geq  \frac{R}{2}$. We define, for any $\lambda >0$, 
$$B_{\lambda}(x)=\chi (|x|)\lambda^{\frac{2-N}{2}}B_0(\lambda^{-1}|x|)$$
and
$$Z_{\lambda ,0}(x)=\chi (|x|) \lambda^{\frac{2-N}{2}} V_0 (\lambda^{-1}|x|).$$
These are radial functions in $B_R$ and are compactly supported in $B_R$, thus have zero Neumann derivative. We denote by $\Pi_{\lambda}$, respectively $\Pi_{\lambda}^{\bot}$, the orthogonal projection of $H^1(B_R)$ onto
$$K_{\lambda}=\mathrm{span}\left\{Z_{\lambda ,0}\right\},$$
respectively onto
\begin{equation}\label{nodegen}K_{\lambda}^{\bot}=\left\{\phi \in H^1(B_R)\slash \left\langle \phi , Z_{\lambda ,0}\right\rangle =0 \right\}.\end{equation}
We recall that a solution $u_0$ of \eqref{eq} with $\e=0$ is said to be nondegenerate if \eqref{nondege} is satisfied with $h_0 \equiv 1$. 

\begin{rmq}[Encoding the radius of the ball with a parameter]  \label{remscaling}
If $\mu >0$ is fixed it is easily seen that $u$ solves $-\Delta u + \mu u =u^{2^*-1}$ in $B_R$, with $\partial_{\nu} u = 0$ on $\partial B_R$, if and only if $v(x):= \mu^{-\frac{N-2}{4}} u (\frac{x}{\sqrt{\mu}})$ solves $-\Delta v+v=v^{2^*-1}$ in $B_{R \sqrt{\mu}}$ with $\partial_{\nu} v = 0$ on $\partial B_{R \sqrt{\mu}}$ (no radiality is assumed here). Clearly, $v$ is non-degenerate if and only if $u$ is non-degenerate.
\end{rmq} 

With Remark \ref{remscaling} in mind, the following result shows that the existence of non-degenerate solutions of \eqref{eq} when $\e = 0$ is a generic property in the choice of the radius $R$:
\begin{prop}
\label{nondegu0}
Let $u_0 \in H^1(B_R)$ be a positive solution to 
$$\left \{ \begin{aligned}-\Delta u_0 + \mu u_0 & = u_0^{2^\ast -1},& \text{in}\ B_R,\\ \partial_\nu u_0 & =0,& \text{on}\ \partial B_R\end{aligned} \right.$$ for some $\mu>0$. Fix $\nu>0$. Then, there exist $\mu_\nu >0$ and $u_{0,\nu} >0$ such that $|\mu_\nu -\mu|+\|u_0 - u_{0,\nu} \|_{C^{2,\theta}(\overline{B_R})}<\nu$, $\theta \in (0,1)$, and $u_{0,\nu}>0$ is a non-degenerate solution of
$$\left \{ \begin{aligned}-\Delta u_{0,\nu}+\mu_\nu u_{0,\nu}& =u_{0,\nu}^{2^\ast -1}, \text{ in }\ B_R,\\ \partial_\nu u_{0,\nu} &=0 \text{ on }\ \partial B_R.\end{aligned} \right.$$
\end{prop}

\begin{proof}
We proceed as in \cite{RobertVetois3}. We set, for $\eta \geq 0$,
$$A_\eta =\inf_{u\in H^1 (B_R) \backslash \{0\}} \dfrac{\int_{B_R} \Big(|\nabla u|^2 + (\mu - (2^\ast -1)u_0^{2^\ast -2} -\eta )u^2 \Big)dx }{(\int_{B_R} |u|^{2^\ast} dx)^{\frac{2}{2^\ast}}}.$$
It is easy to see that $A_\eta <0$ and $\lim_{\eta \rightarrow 0} A_\eta =A_0<0$. Direct minimisation yields, for all $\eta \ge 0$, a minimizer $w_\eta \in C^{2,\theta}(\overline{B_R})$, $w_\eta >0$, such that 
\begin{equation} \label{Eeta}
\left \{ \begin{aligned} -\Delta w_\eta + (\mu - (2^\ast -1)u_0^{2^\ast -2} -\eta)w_\eta & = A_\eta w_\eta^{2^\ast -1} \text{  in } B_R \\
 \partial_{\nu} w_\eta  & =0 \text{ on } \partial B_R \end{aligned} \right. .
 \end{equation}
With \eqref{Eeta}, and since $A_\eta < 0$ for all $\eta \ge 0$, it is easily seen that the family $(w_\eta)_{\eta \geq 0}$ is relatively compact in $C^2 (\overline{B_R})$. 

We now claim that, when $\eta = 0$, $u_0$ is the only solution of \eqref{Eeta} that is positive in $\overline{B_R}$. We assume that $u$ and $v$ are two such solutions and we let $x_0 \in \overline{B_R}$ be the maximum point of $\frac{u}{v}$ in $\overline{B_R}$. Straightforward computations show that $w := \frac{u}{v}$ satisfies
$$ - \triangle w - A_0 w^{2^*-1}\Big(1 - \frac{1}{w^{2^*-2}} \Big) = 2 \langle \nabla w, \nabla \ln v \rangle \quad \text{ in } \overline{B_R}. $$
By definition of $x_0$, and since $\partial_\nu w = 0$ in $\partial B_R$, we always have $\nabla w(x_0)= 0$, regardless of whether $x_0$ is an interior or boundary point. If we write $w(x) = w(r\omega)$ in spherical coordinates for $x = r\omega \in \overline{B_R}$, the condition that $x_0$ is the global maximum of $w$ in $\overline{B_R}$ implies both $\triangle_{\mathbb{S}^{n-1}}w(R \cdot) \le 0$ and, since $\partial_{\nu} w(x_0) = 0$, $\partial_{\nu}^2 w(x_0) \le 0$. Therefore $- \triangle w(x_0) \ge 0$, so that applying the latter equation at $x_0$ and since $A_0 <0$ yields $w(x_0) \le 1$. This implies that $u \le v$ in $\overline{B_R}$. The same argument with the roles of $u$ and $v$ reversed also shows that $v \le u $ in $\overline{B_R}$, and thus $u \equiv v$ in $\overline{B_R}$.

A first consequence of this uniqueness result is that $A_0 = - (2^*-2)$, which is easily seen integrating \eqref{Eeta} by parts against $u_0$ when $\eta =0$. As a second consequence, we get that $\|w_\eta - u_0\|_{C^{2,\theta}(\overline{B_R})} \rightarrow 0$ as $\eta \rightarrow 0$. Let now, for $\eta >0$, $\mu_\eta = \mu - (2^\ast -1) (u_0^{2^\ast -2} -w_\eta^{2^\ast -2}) - \eta$. Clearly $\lim_{\eta \to 0} \mu_\eta = \mu  >0$. By \eqref{Eeta} $w_\eta$ satisfies
$$ - \triangle w_\eta + \mu_\eta w_\eta = \big( 2^*-1 + A_\eta \big) w_\eta^{2^*-1}  \quad \text{ in } \overline{B_R}.  $$
We let $u_{0,\eta} =  \big( 2^*-1 + A_\eta \big)^{\frac{N-2}{4}}w_\eta$. Since $2^*-1 + A_\eta \to 1$ as $\eta \to 0$ we have $u_{0,\eta} \to u_0$ in $C^2(\overline{B_R})$ as $\eta \to 0$. Let $\nu >0$ be fixed. It remains to choose $\eta >0$ small enough so that $- \triangle + \big( 1 - (2^*-1) u_0^{2^*-2} - \eta \big)$ has no kernel in the sense of \eqref{nokernel} and $|\mu_\eta -\mu|+\|u_0 - u_{0,\eta} \|_{C^{2,\theta}(\overline{B_R})}<\nu$. This choice of $\eta$ implies that $u_{0,\eta}$ is still non-degenerate and concludes the proof.
\end{proof}

\subsection{Finite dimensional reduction}

To prove Theorem \ref{thinstabi} we will look, for $\e \neq 0$, for a radial solution $u_\e$ to \eqref{eq} of the form $$u_\e=u_0+ B_{\lambda_\e (t)}+\phi_\e(t),$$
where $u_0$ is a nondegenerate nonnegative radial solution of \eqref{eq}, $\phi_\e (t)\in K_{\lambda_\e (t)}^{\bot}$ is a radial function and, for any $t >0$ and $\e \neq 0$,
\begin{equation}\label{delta}
\lambda_\e (t)=\left\{\begin{aligned}
& (|\e|t)^2 & \text{ if } N=3, \\
&|\e|t  &\text{ if } N=4 \text{ and }   u_0 >0, \\
&\Big( \frac{|\e|}{\ln \frac{1}{|\e|}}\Big)^{\frac12} t 
& \text{ if }\  N=4 \text{ and } u_0=0,  \\
&(|\varepsilon|t)^{\frac{2}{3}} & \text{ if }\ N=5  \text{ and }  u_0 >0, \\ 
&\sqrt{|\varepsilon|t} & \text{ if } N=5 \text{ and } u_0=0,\\
&\sqrt{|\varepsilon|t} & \text{ if } N\ge 6. \
\end{aligned}\right. \end{equation}

In the following, for the sake of simplicity, we will often write $\lambda(t)$ and $\phi(t)$ instead of $\lambda_\e(t)$ and $\phi_\e(t)$. It is well-known that equation \eqref{eq} is equivalent to the following system
\begin{equation}
\label{eq1}
\Pi_{\lambda (t) }(u_0 + B_{\lambda (t) }+\phi(t) -i^\ast (f_\varepsilon(u_0+ B_{\lambda (t) }+\phi(t) )))=0,
\end{equation}
and
\begin{equation}
\label{eq2}
\Pi_{\lambda (t) }^\bot(u_0 + B_{\lambda (t) }+\phi(t)  -i^\ast (f_\varepsilon(u_0 +B_{\lambda (t) }+\phi(t))))=0.
\end{equation}
Since $u_0$ and $B_{\lambda(t)}$ satisfy the Neumann boundary condition, and by definition of $i^\ast$, $\phi (t)$ also satisfies the Neumann boundary condition in the two previous equations. We begin by solving \eqref{eq2}. Classically this requires a quantitative estimate on the error of the approximate solution  $u_0 + B_{\lambda(t)}$

From now on, if $a_\varepsilon$ and $b_\varepsilon$ denote two positive families of real numbers, we will write $a_\varepsilon \lesssim b_\varepsilon$ if there exists a constant $C$ independent of $\varepsilon$ and $t$ such that
$$a_\varepsilon \leq C b_\varepsilon \text{ for all } \varepsilon.$$

\begin{lem}
\label{reste}
Let $0 <a <b$ be real numbers. There exists a positive constant $C_{a,b}$ such that for $|\varepsilon|$ small  we have
 $$
 \begin{aligned}
 \left\|i^\ast (f_\varepsilon (u_0+B_{\lambda (t)}))-u_0-B_{\lambda (t) }\right\|_{\alpha_\varepsilon} 
 & \leq C_{a,b}
\left \{ \begin{aligned}
& \Big(\frac{|\e|}{\ln \frac{1}{|\e|}}\Big)^{\frac{1}{2}} \text{ if } N=4\ \text{ and } u_0=0,\\ 
& |\varepsilon|^{\frac{3}{4}}   \text{ if } N=5 \text{ and } u_0=0,\\ 
&|\varepsilon| \ln \frac{1}{|\varepsilon |}   \text{ otherwise }. \end{aligned} \right.
\end{aligned}$$
uniformly  for $t \in [a, b]$. 
\end{lem}

\begin{proof} 
Let $0<a <b$ be real numbers. By \eqref{delta}, unless $N=4$ and $u_0 \equiv 0$, $\lambda_\e(t)$ converges uniformly to zero for $t \in [a,b]$ as $\e \to 0$. If $N=4$ and $u_0 \equiv 0$, again by \eqref{delta}, $\lambda_\e(t)$ converges uniformly to zero for $t \in [a,+ \infty)$ as $\e \to 0$. 
For simplicity we set $\alpha_\varepsilon=\alpha$ in this proof. We have, by \eqref{esterreur}
\begin{align*}
\left\|i^\ast (f_\varepsilon (u_0+B_{\lambda (t)}))-u_0-B_{\lambda (t) }  \right\|_{\alpha}
&\lesssim \left\| (f_\varepsilon (u_0+B_{\lambda (t)}))-(-\Delta +Id)(u_0+B_{\lambda (t) })\right\|_{L^\frac{N \alpha }{N+2\alpha}(B_R) .}
\end{align*}
The triangle inequality yields
\begin{align*} 
\left\|i^\ast (f_\varepsilon (u_0+B_{\lambda (t)}))-u_0-B_{\lambda (t) }  \right\|_{\alpha}
&\lesssim  \left\|f_\varepsilon (u_0+B_{\lambda (t)})-f_\varepsilon (u_0)- f_\varepsilon (B_{\lambda (t)})\right\|_{L^\frac{N\alpha}{N+2\alpha}(B_R)}\nonumber\\
&\quad +   \left\|f_\varepsilon (u_0)+\Delta u_0- u_0\right\|_{L^\frac{N\alpha}{N+2\alpha}(B_R)}\nonumber\\
&\quad +    \left\|f_\varepsilon (B_{\lambda (t)})+ (\Delta - Id)(B_{\lambda (t)})\right\|_{L^\frac{N\alpha}{N+2\alpha}(B_R)}\nonumber\\
&\lesssim  (I_1 +I_2+I_3).
\end{align*}
We begin by estimating $I_1$. Using once more the triangle inequality, we get
\begin{align}
I_1& \lesssim  \left\|(f_\varepsilon (u_0+B_{\lambda (t)})-f_\varepsilon (B_{\lambda (t)}) ) 1_{B_ {\sqrt{\lambda (t)}}} \right\|_{L^\frac{N\alpha}{N+2\alpha}(B_R)}\nonumber\\
&\quad +\left\|(f_\varepsilon (u_0+B_{\lambda (t)})-f_\varepsilon (u_0))1_{B_R\backslash B_{\sqrt{\lambda (t)}}}\right\|_{L^\frac{N\alpha}{N+2\alpha}(B_R)}\nonumber\\
&\quad +\left\|f_\varepsilon (B_{\lambda (t)})1_{B_R\backslash B_{\sqrt{\lambda (t)}}} \right\|_{L^\frac{N\alpha}{N+2\alpha}} +   \left\|f_\varepsilon (u_0)1_{ B_{\sqrt{\lambda (t)}}}  \right\|_{L^\frac{N\alpha}{N+2\alpha}(B_R)}\nonumber\\
&\lesssim \left\|u_0 B_{\lambda (t)}^{2^\ast -2-\varepsilon} 1_{B_{\sqrt{\lambda (t)}}} \right\|_{L^\frac{N\alpha}{N+2\alpha}(B_R)}
+\left\|u_0^{2^\ast-2-\varepsilon}B_{\lambda (t)}1_{B_R \backslash B_{\sqrt{\lambda (t)}}}\right\|_{L^\frac{N\alpha}{N+2\alpha}(B_R)} 
 \nonumber\\
&\quad +\left\|f_\varepsilon (B_{\lambda (t)})1_{B_R\backslash B_{\sqrt{\lambda (t)}}} \right\|_{L^\frac{N\alpha}{N+2\alpha}(B_R)} +   \left\|f_\varepsilon (u_0)1_{ B_{\sqrt{\lambda (t)}}}  \right\|_{L^\frac{N\alpha}{N+2\alpha}(B_R)}.
\end{align}
A first remark is that, by definition of $\lambda(t) = \lambda_\e(t)$ in \eqref{delta}, we have
\begin{equation} \label{controledelta}
\lambda_\e(t)^{\e} = 1 + o(1)\quad  \text{ uniformly for } t \in [a,b]  
\end{equation}
and 
\begin{equation} \label{controledelta2}
\e |\ln \lambda(t)| \lesssim |\e| \ln \frac{1}{|\e|} \quad  \text{ uniformly for } t \in [a,b] . 
\end{equation}
By definition of $\alpha = \alpha_\e$, we have $\frac{N\alpha}{N+2 \alpha} = \frac{2N}{N+2} + O(|\e|)$ as $\e \to 0$. Since $B_{\lambda_\e(t)} \lesssim \lambda(t)^{- \frac{N-2}{2}}$ in $\overline{B_R}$ we can compute, using \eqref{controledelta} and \eqref{controledelta2}:
$$ \begin{aligned}
 \left\|u_0 B_{\lambda (t)}^{2^\ast -2-\varepsilon} 1_{B_{\sqrt{\lambda (t)}}} \right\|_{L^\frac{N\alpha}{N+2\alpha}(B_R)} &\lesssim  \left\|u_0 B_{\lambda (t)}^{2^\ast -2} 1_{B_{\sqrt{\lambda (t)}}} \right\|_{L^\frac{2N}{N+2}(B_R)} \\
 &\lesssim \left \{ \begin{aligned} & \lambda(t)^{\frac{N-2}{2}} & \text{ if } N=3,4,5, \\ & \lambda(t)^{2} |\ln \lambda(t)|^{\frac{2}{3}} & \text{ if } N=6, \\ & \lambda(t)^{2} & \text{ if } N\ge 7. \end{aligned} \right. 
 \end{aligned}
$$
Similar computations using \eqref{controledelta} and \eqref{controledelta2} show that
$$ \begin{aligned}
I_1 \lesssim \lambda(t)^{\frac{N+2}{4}} +  \left \{ 
\begin{aligned} 
& \lambda(t)^{\frac{N-2}{2}} & \text{ if } N=3,4,5, \\ 
& \lambda(t)^{2} |\ln \lambda(t)|^{\frac{2}{3}} & \text{ if } N=6, \\ 
& \lambda(t)^{2} & \text{ if } N\ge 7. 
\end{aligned} \right. 
\end{aligned} 
$$
Let us turn to $I_2$. From Taylor's expansion, we see that
$$I_2=\left\|f_\varepsilon (u_0)-f_0(u_0) \right\|_{L^\frac{N\alpha}{N+2\alpha}(B_R)}\lesssim |\varepsilon|.$$
Next we estimate $I_3$. We have that
\begin{align*}
I_3 &\lesssim \left\|\chi^{2^\ast-1-\varepsilon}(B_{\lambda (t)}^{2^\ast-1-\varepsilon}-B_{\lambda (t)}^{2^\ast-1} )\right\|_{L^\frac{N\alpha}{N+2\alpha}(B_{R} )}\\
&\quad +\left\|(\chi^{2^\ast-1-\varepsilon}-\chi^{2^\ast-1})B^{2^\ast -1}_{\lambda (t) }  \right\|_{L^\frac{N\alpha}{N+2\alpha}(B_R)}\\
&\quad + \left\|f_0 (B_{\lambda (t)})-(-\Delta +Id)(B_{\lambda (t)}) \right\|_{L^\frac{N\alpha}{N+2\alpha}(B_R)}.
\end{align*}
Using Taylor's expansions, \eqref{controledelta} and \eqref{controledelta2} we find
\[\left\|\chi^{2^\ast-1-\varepsilon}(B_{\lambda (t)}^{2^\ast-1-\varepsilon}-B_{\lambda(t)}^{2^\ast-1} )\right\|_{L^\frac{N\alpha}{N+2\alpha}(B_R)}\lesssim |\varepsilon| \ln \frac{1}{|\e|} ,\]
and
\[\left\|(\chi^{2^\ast-1-\varepsilon}-\chi^{2^\ast-1})B_{\lambda (t)}^{2^\ast -1}  \right\|_{L^\frac{N\alpha}{N+2\alpha}(B_R)}\lesssim \lambda(t)^{\frac{N+2}{2}} ,\]
where it is intended that all these expansions are uniform for $t \in [a,b]$. Finally, for the last term in the estimation of $I_3$ we have, using again \eqref{controledelta} and \eqref{controledelta2}:
\begin{align*}
& \left\|f_0 (B_{\lambda (t)})-(-\Delta +Id)(B_{\lambda (t)}) \right\|_{L^\frac{N\alpha}{N+2\alpha} (B_R)}\\
&\lesssim \|B_{\lambda (t)} \Delta \chi\|_{L^\frac{N\alpha}{N+2\alpha}(B_R)} +\|\nabla \chi \nabla B_{\lambda(t)} \|_{L^\frac{N\alpha}{N+2\alpha}(B_R)} +\|\chi B_{\lambda (t)}\|_{L^\frac{N\alpha}{N+2\alpha}(B_R)}\\
& \lesssim  \lambda(t)^{\frac{N-2}{2}} + \left \{ \begin{aligned} & \lambda(t)^{\frac{N-2}{2}} & \text{ if } N=3,4,5, \\ & \lambda(t)^{2} |\ln \lambda(t)|^{\frac{2}{3}} & \text{ if } N=6, \\ & \lambda(t)^{2} & \text{ if } N\ge 7. \end{aligned} \right.
\end{align*}
Combining the previous estimates shows that 
\begin{equation} \label{errfinale}\begin{aligned}
 \Vert i^\ast (f_\varepsilon (u_0+B_{\lambda (t)}))-u_0-B_{\lambda (t) }  \Vert_{\alpha} & \lesssim  \lambda(t)^{\frac{N+2}{4}} + |\e|  \ln \frac{1}{|\e|}  \\
& \quad +  \lambda(t)^{\frac{N-2}{2}} + \left \{ \begin{aligned} & \lambda(t)^{\frac{N-2}{2}} & \text{ if } N=3,4,5 \\ & \lambda(t)^{2} |\ln \lambda(t)|^{\frac{2}{3}} & \text{ if } N=6 \\ & \lambda(t)^{2} & \text{ if } N\ge 7 \end{aligned} \right.
\end{aligned} 
\end{equation}
It is now easily seen with \eqref{delta} that \eqref{errfinale} concludes the proof of Lemma \ref{reste}. 
\end{proof} 

The general resolution of \eqref{eq2} up to kernel elements is given by the following Proposition: 
\begin{prop}
\label{prop4.1}
Let $u_0\in C^{2,\theta}(B_R)$ be a nondegenerate radial solution of \eqref{eq} with $\e=0$. Let $a<b$ be two real numbers. There exists a positive constant $C_{a,b}$ depending only on $a$ and $b$ such that for $|\varepsilon|$ small enough and for all $t \in [a,b]$ there exists a unique radial function $\phi(t) = \phi_\varepsilon(t) \in K_{\lambda(t)}^\bot$ solving equation \eqref{eq2} and satisfying
\begin{equation}
\label{eqprop4.1}
\begin{aligned}
\left\|\phi (t)\right\|_{\alpha_\varepsilon} 
\leq C_{a,b} \left\|i^\ast (f_\varepsilon (u_0+B_{\lambda (t)}))-u_0-B_{\lambda (t) }\right\|_{\alpha_\varepsilon} 
\end{aligned}
\end{equation}
Moreover, $t \in [a,b]  \mapsto \phi(t) \in H^1(B_R)$ is continuously differentiable. 
\end{prop}
The proof is very classical and we omit it. We refer for instance to \cite{RobertVetois3, RobertVetois} for the same kind of result (see also \cite{ReyWei}).

\subsection{The reduced problem.}
For $\varepsilon$ small enough, we define the energy in $H^1(B_R)$ associated to \eqref{eq} by
$$J_\varepsilon (u)=\dfrac{1}{2}\int_{B_R} (|\nabla u|^2 + u^2 )dx - \int_{B_R} F_\varepsilon(u)dx,$$
where $F_\varepsilon(s)= \frac{1}{2^*-\e} s_+^{2^*-\e}$. We set 
\begin{equation} \label{energiereduite}
I_\varepsilon(t)=J_\varepsilon \big(u_0+B_{\lambda(t)}+\phi(t) \big), \quad t\in [a,b],
\end{equation} 
where $\phi (t)\in K_{\lambda (t)}^\bot$ is the function defined in Proposition \ref{prop4.1}. In the next proposition, we give an expansion of $I_\varepsilon$ as $\varepsilon\to 0$.

\begin{prop}
\label{energiered}
Let $u_0\in C^{2,\theta}(B_R)$, $\theta\in (0,1)$ be a nondegenerate radial solution of \eqref{eq} with $\e=0$.  For $N \ge 3$ we let 
\begin{equation} \label{defc5}
c_5(N,u_0)=\dfrac{1}{N}K_N^{-N} \left( \dfrac{2(N-1)}{(N-2)(N-4)}1_{N\geq 6}- \dfrac{2^{N} u_0(0)\omega_{N-1}}{(N (N-2))^{\frac{N-2}{4}} \omega_N} 1_{N\leq 6}\right),
\end{equation}
where $\omega_N$ stands for the volume of $\mathbb{S}^N$ and $K_N$ is as in \eqref{KNN}.We consider $I_\e$ defined in \eqref{energiereduite}.

\begin{itemize}
\item  Assume that $N=4$ and $u_0 \equiv 0$. There exist numerical  constants $\hat{c}_i, 1 \le i \le 5$, with $\hat{c}_5 >0$, such that, as $\e \to 0$,
\begin{equation} \label{fin1}
I_\varepsilon(t)=\hat{c}_1+\hat{c}_2\varepsilon +\hat{c}_3\varepsilon \ln \frac{1}{ |\varepsilon|} +\hat{c}_4 \e \ln \ln \frac{1}{|\e|} + \hat{c}_5\big( \e\ln \frac{1}{t} + |\e|\frac32 t^2 \big)  +o(\e). 
\end{equation}

\item In all the other cases: there exist constants $c_i(N,u_0)$, $i=1,2,5$ depending on $N$ and $u_0$ and $c_i(N)$, $i=3,4$, depending only $N$ such that, as $\e \to 0$:
\begin{equation}
\label{fin3}
\begin{aligned}
I_\varepsilon(t)& =c_1(N,u_0)+c_2(N,u_0)\varepsilon+c_3(N) \varepsilon \ln \frac{1}{|\varepsilon|} +c_4(N) \varepsilon \ln \frac{1}{t}+c_5(N,u_0) |\varepsilon| t +o(\varepsilon),
\end{aligned}
\end{equation}
and the constant $c_4(N)$ is positive in these cases. 
\end{itemize}
All these expansions are uniform with respect to $t\in [a,b]$. 
\end{prop}

\begin{proof}
As before we write $\lambda(t)$ instead of $\lambda_\e(t)$ in this proof, where $\lambda(t)$ is given by \eqref{delta}. Using Lemma \ref{reste} and Proposition \ref{prop4.1}, we have
\begin{align*}
I_\varepsilon (t)-J_\varepsilon (u_0+B_{\lambda (t)})
&=
\left\langle u_0+B_{\lambda (t)}-i^\ast (f_\varepsilon(u_0+B_{\lambda (t)})),\phi (t) \right\rangle+O(\left\|\phi (t)\right\|^2)\\
&= \begin{cases}
O \Big( \frac{|\varepsilon|}{\ln \frac{1}{|\e|}} \Big),& \text{if}\ N=4\ \text{and}\ u_0=0,\\ 
 O \big(|\varepsilon|^{\frac{3}{2}} \big),& \text{if}\ N=5\ \text{and}\ u_0=0,\\
O \big( \varepsilon^2 \ln \frac{1}{|\varepsilon|}^2 \big),& \text{otherwise}, \end{cases} \\
& = o(\e)
\end{align*}
when $\varepsilon\rightarrow 0$, uniformly for $t \in [a,b]$. Now, the proof of the proposition is reduced to estimate $J_\varepsilon (u_0+B_{\lambda (t)})$.
Since $u_0$ is a solution of \eqref{eq}, it follows that
\begin{align*}
J_\varepsilon(u_0+B_{\lambda (t)}) &= \dfrac{1}{N}\int_{B_R} u_0^{2^\ast}dx + \dfrac{\varepsilon}{2^\ast}\int_{B_R} u_0^{2^\ast} (\ln u_0 - \dfrac{1}{2^\ast})dx + I_{1,\varepsilon,t}+I_{2,\varepsilon,t}+I_{3,\varepsilon,t}+O(\varepsilon^2),
\end{align*} 
where 
$$
I_{1,\varepsilon,t}=\dfrac{1}{2}\int_{B_R} \big( |\nabla B_{\lambda (t)}|^2 + B_{\lambda (t)}^2 \big)dx -\int_{B_R} F_\varepsilon(B_{\lambda (t)})dx, \quad 
I_{2,\varepsilon,t}=-\int_{B_R} f_\varepsilon(B_{\lambda (t)}) u_0dx,$$
and
$$I_{3,\varepsilon,t}=-\int_{B_R} (F_\varepsilon(u_0+B_{\lambda (t)})-F_\varepsilon(u_0)-F_\varepsilon(B_{\lambda (t)})-f_\varepsilon(u_0)B_{\lambda (t)}- f_\varepsilon(B_{\lambda (t)})u_0)dx.$$

We begin by the estimate of $I_{3,\e,t}$. Remark that this term vanishes when $u_0 \equiv 0$. Using Taylor expansion and rough estimates, we get
\begin{align*}
|I_{3,\varepsilon,t}| \lesssim {} &  \left\|(F_\varepsilon(u_0+B_{\lambda (t)})-F_\varepsilon(B_{\lambda (t)})-f_\varepsilon(B_{\lambda (t)})u_0)\right\|_{L^1(B_{\sqrt{\lambda (t))}})}\\
&+ \left\|(F_\varepsilon(u_0+B_{\lambda (t)})-F_\varepsilon(u_0)-f_\varepsilon(u_0)B_{\lambda (t)})\right\|_{L^1 (B_R\backslash B_{\sqrt{\lambda(t)}})}\\
&+ \left\|F_\varepsilon(u_0)\right\|_{L^1 (B_{\sqrt{\lambda(t)}})}+\left\|f_\varepsilon(u_0)B_{\lambda (t)} \right\|_{L^1 (B_{\sqrt{\lambda (t)}})}\\
&+\left\|F_\varepsilon(B_{\lambda (t)})\right\|_{L^1 (B_R\backslash B_{\sqrt{\lambda(t)}})}+\left\|u_0 f_\varepsilon(B_{\lambda (t)})\right\|_{L^1 (B_R\backslash B_{\sqrt{\lambda(t)}})}\\
\lesssim {} &  \left\|u_0^2 B_{\lambda (t)}^{2^\ast-2-\varepsilon}\right\|_{L^1 (B_{\sqrt{\lambda(t)}})}+\left\|u_0^{2^\ast -2-\varepsilon}B_{\lambda (t)}^2 \right\|_{L^1 (B_R\backslash B_{\sqrt{\lambda(t)}})}\\
&+\left\|F_\varepsilon(B_{\lambda (t)})\right\|_{L^1 (B_R\backslash B_{\sqrt{\lambda(t)}})}+\left\|u_0 f_\varepsilon(B_{\lambda (t)})\right\|_{L^1 (B_R \backslash B_{\sqrt{\lambda(t)}})}\\
&+ \left\|F_\varepsilon(u_0)\right\|_{L^1(B_{\sqrt{\lambda(t)}})}+\left\|f_\varepsilon(u_0)B_{\lambda (t)}\right\|_{L^1 (B_{\sqrt{\lambda(t)}})}\\
\lesssim {} &  \left\|u_0^2 B_{\lambda (t)}^{2^\ast-2-\varepsilon}\right\|_{L^1 (B_{\sqrt{\lambda(t)}})}+\left\|u_0^{2^\ast -2-\varepsilon}B_{\lambda (t)}^2 \right\|_{L^1 (B_R\backslash B_{\sqrt{\lambda(t)}})}\\
&+ \lambda (t)^{\frac{N}{2}} .
\end{align*}
 Therefore estimating the last two terms and using the definition of $\lambda (t)$ given by \eqref{delta}, we obtain
\begin{equation}
\label{enei1}
\begin{aligned} 
|I_{3,\varepsilon,t}|& \lesssim \left\{\begin{array}{lll}
\lambda (t) \ \mathrm{if}\ N=3 \\
\lambda (t)^2 |\ln \lambda(t)| \ \mathrm{if}\ N=4 \ \mathrm{and} \ u_0 >0\\
\lambda (t)^{\frac{N}{2}}\ \mathrm{if}\ N\ge5 \\
\end{array} \right. \\
& = o(\e) 
\end{aligned}
\end{equation}
uniformly for $t \in [a,b]$. Now, let us estimate $I_{2,\varepsilon,t}$. For any $N \ge 3$ we have, with \eqref{delta} \eqref{controledelta} and \eqref{controledelta2},
\begin{align}
\label{enei2}
 I_{2,\varepsilon,t} = {} & - u_0 (0) \omega_{N-1}(N(N-2) )^{\frac{N-2}{4} (2^\ast -1-\varepsilon)}\lambda (t)^{\frac{N-2}{2}(1+\varepsilon)}\nonumber\\
&  \times\int_0^{\frac{R}{4\lambda(t)}} 
\dfrac{r^{N-1}}{(1+r^2)^{\frac{N+2}{2}-\varepsilon \frac{N-2}{2}}}dr\nonumber +  O(\lambda (t)^{\frac{N+2}{2} - \frac{N-2}{2}\e})\nonumber\\
= {}&  -\dfrac{2^{N}u_0(0) K_N^{-N}\omega_{N-1}\lambda (t)^{\frac{N-2}{2}}}{N^{\frac{N+2}{4}} (N-2)^{\frac{N-2}{4}} \omega_N}+O(\lambda (t)^{\frac{N}{2}}).
\end{align}
The last integral in \eqref{enei2} has been computed by standard considerations, see for instance \cite[equation $(4.6)$]{PremoselliVetois3}.

Finally, let us consider $I_{1,\varepsilon ,t}$. We get
$$\int_{B_R} |\nabla B_{\lambda (t)}|^2 dx = K_N^{-N} +O(\lambda (t)^{N-2}),$$

$$\int_{B_R} B_{\lambda (t)}^2 dx =\begin{cases}O(\lambda (t)) ,& N=3,\\ \frac{3}{2} K_4^{-4}\lambda (t)^2 |\ln \lambda (t)| +O(\lambda (t)^2 ),& N=4,\\  \frac{4(N-1)}{N(N-2)(N-4)} K_N^{-N}\lambda^2 (t)+O(\lambda (t)^{N-2} ),& N\geq 5,  \end{cases}$$
and, with \eqref{controledelta} and \eqref{controledelta2},
$$\int_{B_R} F_\varepsilon (B_{\lambda (t)}) dx = \dfrac{N-2}{2N} K_N^{-N} (\lambda (t)^{\frac{N-2}{2}\varepsilon}+\dfrac{2\beta_N}{N-2}\varepsilon )+O(\lambda (t)^{N - (N-2) \e}) + O(\e^2),$$
where $K_N$ is defined as in \eqref{KNN} and
$$\beta_N = 2^{N-3} (N-2)^2 \dfrac{\omega_{N-1}}{\omega_N} \int_0^\infty \dfrac{r^{\frac{N-2}{2}} \ln (1+r)}{(1+r)^N} dr+\dfrac{(N-2)^2}{4N} (1-2N \ln \sqrt{N(N-2)}) .$$
So in view of the previous estimates and \eqref{delta}, \eqref{controledelta} and \eqref{controledelta2}, we find
\begin{equation}
\label{enei3}
\begin{aligned}
I_{1,\varepsilon,t}& = \dfrac{K_N^{-N}}{N}\Bigg(1-\beta_N \varepsilon-\Big(\dfrac{N-2}{2}\Big)^2\varepsilon \ln \lambda (t) + O(\e^2)\\
&+ \left\{ \begin{aligned} & O(\lambda (t)) & \text{ if }  N=3\\ & 3 \lambda (t)^2 |\ln \lambda (t)| +O(\lambda (t)^2 ) & \text{ if } N=4\\  & \frac{2(N-1)}{(N-2)(N-4)} \lambda^2 (t)+O(\lambda (t)^{N-2} )& \text{ if }  N\geq 5  \end{aligned} \right \} \Bigg)\\
\end{aligned}
\end{equation}
Assume first that $N=3$. Combining \eqref{enei1}, \eqref{enei2} and \eqref{enei3}, we obtain, with \eqref{delta},
\begin{equation} \label{DLN3}
\begin{aligned}
J_\varepsilon(u_0+B_{\lambda (t)}) &=  \frac{1}{3}\int_{B_R} u_0^{6}dx + \dfrac{\varepsilon}{6}\int_{B_R} u_0^{6} (\ln u_0 - \dfrac{1}{6})dx +\dfrac{1}{3}K_3^{-3}\left(1-\beta_3 \varepsilon \right)\\
&  - \frac16 K_3^{-3} \e \ln |\e|
 -\frac{1}{12} K_3^{-3} \varepsilon \ln t+ c_5(3,u_0)|\e| t+o(\varepsilon).
  \end{aligned}
\end{equation}
Assume now that $N \ge 5$. Combining again \eqref{enei1}, \eqref{enei2} and \eqref{enei3}, we obtain, with \eqref{delta}
\begin{equation} \label{DLN5}
\begin{aligned}
 J_\varepsilon(u_0+B_{\lambda (t)}) 
&=  \frac{1}{N}\int_{B_R} u_0^{2^\ast}dx + \dfrac{\varepsilon}{2^\ast}\int_{B_R} u_0^{2^\ast} (\ln u_0 - \dfrac{1}{2^\ast})dx +\dfrac{1}{N}K_N^{-N}\left(1-\beta_N \varepsilon \right) \\
& +\dfrac{1}{N}K_N^{-N} \Bigg( -\Big( \dfrac{N-2}{2} \Big)^2\varepsilon \ln \lambda (t) +   \dfrac{2(N-1)}{(N-2)(N-4)}\lambda (t)^2\\
&- \dfrac{2^{N}u_0(0) \omega_{N-1}\lambda (t)^{\frac{N-2}{2}}}{(N (N-2))^{\frac{N-2}{4}} \omega_N}\Bigg)+o(\varepsilon).
\end{aligned}
\end{equation}
By \eqref{delta} $\lambda(t)$ is a power of $|\e|t$ when $N \ge 5$. As a consequence we have $\ln \lambda(t) = C( \ln |\e| + \ln t)$ for some dimensional constant $C$, $\lambda(t)^2 = o\big( \lambda(t)^{\frac{3}{2}}\big)$ if $N = 5$ and $\lambda(t)^{\frac{N-2}{2}} = o\big(\lambda(t)^2\big)$ if $N \ge 7$ as $\e \to 0$. Proposition \ref{energiered} then follows from \eqref{delta},\eqref{controledelta} and \eqref{controledelta2} when $N \ge 5$. 

We now assume that $N = 4$. Assume first that $u_0 >0$. Then $\lambda(t) = t|\e|$, so that $\lambda(t)^2 \ln \lambda(t) = o\big( \lambda(t) \big)$ as $\e \to 0$. Combining \eqref{enei1}, \eqref{enei2} and \eqref{enei3} with \eqref{delta} we thus get
\begin{equation} \label{DLN41}
\begin{aligned}
 J_\varepsilon(u_0+B_{\lambda (t)}) 
&=  \frac{1}{4}\int_{B_R} u_0^{4}dx + \dfrac{\varepsilon}{4}\int_{B_R} u_0^{4} (\ln u_0 - \dfrac{1}{4})dx +\dfrac{1}{4}K_4^{-4}\left(1-\beta_4 \varepsilon \right) \\
& -\dfrac{1}{4}K_4^{-4} \e \ln |\e| - \dfrac{1}{4}K_4^{-4} \Big(\varepsilon \ln t + c_1(4, u_0) |\e| t\Big)+o(\varepsilon).
\end{aligned}
\end{equation}
Assume finally that $N=4$ and $u_0 \equiv 0$. In this case $I_{2,\e,t}$ and $I_{3,\e,t}$ vanish and \eqref{enei3}, \eqref{delta}, \eqref{controledelta} and \eqref{controledelta2} show that
\begin{equation} \label{DLN42}
\begin{aligned}
I_\varepsilon(t)& = \dfrac{K_4^{-4}}{4}\big(1-\beta_4 \varepsilon\big)  + O(\e^2|\ln \e|^2) +O(\lambda (t)^2 )+\frac{1}{4} K_4^{-4} \Big(\varepsilon+ 3 \lambda (t)^2 \Big) \ln \frac{1}{\lambda (t)}.
\end{aligned}
\end{equation}
Since $\lambda(t) = \big( \frac{|\e|}{\ln \frac{1}{|\e|}}\big)^{\frac12} t$ by \eqref{delta}, we have
\begin{equation*}
\begin{aligned}
\Big( \varepsilon+ 3 \lambda (t)^2 \Big) \ln \frac{1}{\lambda (t)} 
& =  \big( \e + 3 |\e|\frac{t^2}{\ln \frac{1}{|\e|}} \big) \big( \frac12 \ln \frac{1}{|\e|} + \frac12\ln  \ln \frac{1}{|\e|} + \ln \frac{1}{t} \big) \\
& = \hat{c}_3 |\e| \ln \frac{1}{|\e|} + \hat{c}_4 |\e| \ln \ln \frac{1}{|\e|}  +\big( \e\ln \frac{1}{t} + |\e|\frac32 t^2 \big) + o(|\e|)
\end{aligned} 
\end{equation*}
uniformly for $t \in [a,b]$.
This concludes the proof of Proposition \ref{energiered}. 
\end{proof}

\subsection{Proof of Theorems \ref{thinstabi-typeB}, \ref{thinstabi_u0+Bsub} and \ref{thinstabi}.} 
We are now in position to conclude the proof of Theorems \ref{thinstabi-typeB} to  \ref{thinstabi}. We prove the three theorems together. It is a classical result that if, for some $\e \neq 0$, $t_\e >0$ is a critical point of $t \mapsto I_\e(t)$ given by \eqref{energiereduite} then the function $ u_0+B_{\lambda (t_\varepsilon)}+ \phi(t_\varepsilon)$ is a radial solution to \eqref{eq}. We refer for instance to \cite{RobertVetois}. Using Proposition \ref{energiered} we will construct, in each case, a critical point of $t \mapsto I_\e(t)$. In every case we end up with a critical point $u_\e$ of $J_\e$, which is a solution of 
\begin{equation*}
\left \{ \begin{aligned} -\Delta u_\e +u_\e & = (u_\e)_+^{2^\ast -2-\varepsilon} u_\e ,\ \text{in}\ B_R \subset \R^N,\\ \partial_\nu u_\e & =0,\ \text{on}\ \partial B_R. \end{aligned} \right.
\end{equation*}
Maximum principles and standard elliptic theory show that $u_\e >0$ in $\overline{B_R}$, $u_\e \in C^{2,\theta}(\overline{B_R})$ and that $u_\e$ blows up with a single bubble at $0$ as $\e \to 0$ (independently of whether $u_0 \equiv 0$ or $u_0 >0$). Proposition \ref{theorieC0} thus applies: if $i$ denotes the number of zeros of $u_0-1$, it shows that $u_\varepsilon-1$ has $i+1$ zeros if $u_0 (0) <1$, while $u_\e-1$ vanishes $i$-times if $u_0(0) \geq 1$. 

\medbreak

We now construct critical points of $I_\e$. 
We treat first the case $N \neq 4 $ or $u_0 >0$, i.e. type $B$ solutions in dimension $N\ne 4$ or solutions of type $u_0+B$ with $u_0>0$. Let $\mathcal{G}$ be the function defined on $\R^+$ by
$$\mathcal{G}(t)=c_4(N) \ln  \frac{1}{t} + \text{sign}(\varepsilon) c_5 (N,u_0) t , $$
where $c_4(N) >0$ is given by Proposition \ref{energiered}, $c_5(N,u_0)$ is as in \eqref{defc5}. Expansion \eqref{fin3} shows that 
$$\lim_{\varepsilon \rightarrow 0} \dfrac{I_{\varepsilon  }(t) - c_1 (N,u_0) - c_2 (N,u_0) \varepsilon -c_3 (N) \varepsilon \ln \frac{1}{|\varepsilon|} }{\varepsilon}= \mathcal{G} (t), $$
uniformly with respect to $t$ in compact sets of $]0, + \infty[$. It is easily checked that, under the assumptions of Theorem \ref{thinstabi},
$\mathcal{G}$ has a unique global maximum or minimum point $t_0$ in $]0, + \infty[$ that only depends on $N$ and $u_0(0)$. Choose $0 < a <b$ such that $a < t_0 < b$. It is easily checked that $t_0$ is a non-degenerate extremal point of $\mathcal{G}$, hence it is a $C^0$-stable critical point for $\mathcal{G}$ in $[a,b]$, according to the terminology of \cite{Li1997}. For $\e$ small enough there thus exists a family $(t_\varepsilon)_\e$ of critical points of $I_\e$ in $[a,b]$ and $u_\varepsilon = u_0 +B_\lambda (t_\varepsilon )+\phi (t_\varepsilon)$ is a critical point of $J_\e$. We have thus proved Theorem \ref{thinstabi-typeB} when $N\ne 4$ and Theorems \ref{thinstabi_u0+Bsub} and \ref{thinstabi}.

\medbreak

It remains to complete the proof of Theorem \ref{thinstabi-typeB} when $N =  4$, i.e. to show the existence of type $B$ solutions in dimension $N=4$ when $\e >0$. Define, for $t>0$, 
$$\mathcal{H}(t) =  \ln \frac{1}{t} + \frac32 t^2. $$
Expansion \eqref{fin1} shows that
$$\lim_{\varepsilon \to 0} \frac{I_{\varepsilon }(t) -\hat{c}_1-\hat{c}_2\varepsilon -\hat{c}_3\varepsilon \ln \frac{1}{ |\varepsilon|} -\hat{c}_4 \e \ln \ln \frac{1}{|\e|}}{\hat{c}_5 \e} = \mathcal{H} (t) $$
uniformly with respect to $t$ in compact subsets of $]0, + \infty[$. It is easily checked that $\mathcal{H}$ has unique critical point $t_0 = \frac{1}{\sqrt{3}}$ in $]0, + \infty[$, that also happens to be a strict global minimum. If we choose $0 < a <b$ so that $a < t_0 < b$, $t_0$ is now a $C^0$-stable critical point of $\mathcal{H}$. Hence, for $\e$ small enough, there thus exists a family $(t_\varepsilon)_\e$ of critical points of $I_\e$ in $[a,b]$ and $u_\varepsilon = u_0 +B_\lambda (t_\varepsilon )+\phi (t_\varepsilon)$ is a critical point of $J_\e$.

\qed

\begin{rmq}[The dimension $N=3$]  \label{remN3}
We construct solutions of type $B$ in Theorem \ref{thinstabi-typeB} only in the case $N \ge 4$. When $N=3$, one-bubble blowing up solutions with zero weak limit concentrating at an interior point have been constructed in \cite[Theorem 1.2]{ReyWei} when $\e \neq 0$. The examples of \cite{ReyWei} apply to our setting of \eqref{eq}: keeping in mind Remark \ref{remscaling} they yield solutions of type $B$ when $\e >0$ and $R$ is large and when $\e <0$ and $R$ is small (be careful that the sign of $\e$ is reversed in \cite{ReyWei}). More recently, the purely critical case has been considered in \cite{MR3938021} in dimension $N=3$. The examples of \cite{MR3938021} rely however on slight perturbations of $R$ and is therefore of a different nature. 
 \end{rmq}

\begin{rmq}[The dimension $N=6$]  \label{Rem:N=6}
In dimension $N=6$, an extra assumption arises to construct solutions of type $u_0+B$. 
It is therefore natural to ask if solutions with the required assumptions can indeed be used in the constructions we did. The references \cite{Ni83,MR1106295,MR3564729} give a first answer to this question. A careful analysis of the unique solution of 
\begin{equation} \label{contrainteu0-1/2}
 - u'' - \frac{5}{r}u' + u-u^2 = 0 \text{ in } [0, + \infty) , \quad u'(0) = 0, u(0) = a, 
 \end{equation}
 shows the following things.
\begin{enumerate}
\item We can take $u_0=1$ in Theorem \ref{thinstabi} if $1$ is not one of the radial eigenvalue $\mu_i(R)_{i\ge 0}$ of $-\Delta$ on the ball $B_R$ with Neumann boundary condition, i.e. only a discrete set of radii is forbidden.
 \item For every $R>0$ such that $\mu_1(R)>1$ (i.e. $R$ small), there exists $a>1$ such that the solution of \eqref{contrainteu0-1/2} is a decreasing solution of \eqref{eq}, with $N=6$ and $\e=0$, see \cite[Theorem - assertion (b)- page 276]{MR1106295}. These solutions can generically be used in Theorem \ref{thinstabi}. The existence of these decreasing solutions precisely disprove Lin-Ni conjecture in dimension $6$. Our solutions of type $u_0+B$ is an other example of non constant solutions in balls of any size (but the exponent is not fixed). 
 \item For every $R>0$ such that $\mu_1(R)<1$ (i.e. $R$ large), there exists $a<1$ such that the solution of \eqref{contrainteu0-1/2} is an increasing solution of \eqref{eq}, with $N=6$ and $\e=0$, see \cite[Theorem - assertion (a)- page 276]{MR1106295}. These solutions can generically be used in in Theorem \ref{thinstabi} when $u(0)>1/2$ and in Theorem \ref{thinstabi_u0+Bsub} when $u(0)<1/2$. By continuity and monotonicity, the precise range of admissible radii can be expressed in the following way. If $R_0>0$ is such that $\mu_1(R_0)=1$, we can generically use these solutions in Theorem \ref{thinstabi} for $R_0<R<R_1$ where $R_1$ is given in the proof of Theorem \ref{thstabi2etoile}. We can generically use them in Theorem \ref{thinstabi_u0+Bsub} when $R>R_1$. 
 
\item As shown by Ni in \cite{Ni83}, whatever $a\ne 1$, the solution of \eqref{contrainteu0-1/2} can generically be used in Theorem \ref{thinstabi} when $a>1/2$ and in Theorem \ref{thinstabi_u0+Bsub} when $a<1/2$. A precise analysis of the admissible radii would be lengthly. A qualitative analysis is somehow readable in the bifurcation analysis \cite{MR3564729} recalled in Section \ref{motivation}. In fact, with the Remark \ref{remscaling} at hand, a bifurcation analysis can be done exactly as in \cite[Section 4]{MR3564729} yielding solutions at the critical exponent $p=3$ when $R$ varies.This provides many examples of more complex patterns (those are discussed in Section \ref{motivation}).
 
\end{enumerate}
 \end{rmq}

\bibliographystyle{alpha}
\bibliography{biblio2}

\newcommand{\etalchar}[1]{$^{#1}$}
\def \cfac#1{\ifmmode\setbox7\hbox{$\accent"5E#1$}\else
  \setbox7\hbox{\accent"5E#1}\penalty 10000\relax\fi\raise 1\ht7
  \hbox{\lower1.15ex\hbox to 1\wd7{\hss\accent"13\hss}}\penalty 10000
  \hskip-1\wd7\penalty 10000\box7}
\begin{thebibliography}{dPMRW19}

\bibitem[Aub82]{Aub}
Thierry Aubin.
\newblock {\em Nonlinear analysis on manifolds. {M}onge-{A}mp\`ere equations},
  volume 252 of {\em Grundlehren der Mathematischen Wissenschaften [Fundamental
  Principles of Mathematical Sciences]}.
\newblock Springer-Verlag, New York, 1982.

\bibitem[AWZ13]{MR3073256}
Weiwei Ao, Juncheng Wei, and Jing Zeng.
\newblock An optimal bound on the number of interior spike solutions for the
  {L}in-{N}i-{T}akagi problem.
\newblock {\em J. Funct. Anal.}, 265(7):1324--1356, 2013.

\bibitem[AY91]{MR1106295}
Adimurthi and S.~L. Yadava.
\newblock Existence and nonexistence of positive radial solutions of {N}eumann
  problems with critical {S}obolev exponents.
\newblock {\em Arch. Rational Mech. Anal.}, 115(3):275--296, 1991.

\bibitem[AY97]{MR1480241}
Adimurthi and S.~L. Yadava.
\newblock Nonexistence of positive radial solutions of a quasilinear {N}eumann
  problem with a critical {S}obolev exponent.
\newblock {\em Arch. Rational Mech. Anal.}, 139(3):239--253, 1997.

\bibitem[BC14]{MR3215478}
Laurent Bakri and Jean-Baptiste Casteras.
\newblock Non-stability of {P}aneitz-{B}ranson type equations in arbitrary
  dimensions.
\newblock {\em Nonlinear Anal.}, 107:118--133, 2014.

\bibitem[BCF20]{MR4053034}
Denis Bonheure, Jean-Baptiste Casteras, and Juraj Foldes.
\newblock Singular radial solutions for the {K}eller-{S}egel equation in high
  dimension.
\newblock {\em J. Math. Pures Appl. (9)}, 134:204--254, 2020.

\bibitem[BCN17a]{MR3625083}
Denis Bonheure, Jean-Baptiste Casteras, and Benedetta Noris.
\newblock Layered solutions with unbounded mass for the {K}eller-{S}egel
  equation.
\newblock {\em J. Fixed Point Theory Appl.}, 19(1):529--558, 2017.

\bibitem[BCN17b]{MR3641921}
Denis Bonheure, Jean-Baptiste Casteras, and Benedetta Noris.
\newblock Multiple positive solutions of the stationary {K}eller-{S}egel
  system.
\newblock {\em Calc. Var. Partial Differential Equations}, 56(3):Paper No. 74,
  35, 2017.

\bibitem[BCR21]{MR4299895}
Denis Bonheure, Jean-Baptiste Casteras, and Carlos Rom\'{a}n.
\newblock Unbounded mass radial solutions for the {K}eller-{S}egel equation in
  the disk.
\newblock {\em Calc. Var. Partial Differential Equations}, 60(5):Paper No. 198,
  30, 2021.

\bibitem[BE91]{BianchiEgnell}
Gabriele Bianchi and Henrik Egnell.
\newblock A note on the {S}obolev inequality.
\newblock {\em J. Funct. Anal.}, 100(1):18--24, 1991.

\bibitem[BGNT16]{MR3487266}
Denis Bonheure, Massimo Grossi, Benedetta Noris, and Susanna Terracini.
\newblock Multi-layer radial solutions for a supercritical {N}eumann problem.
\newblock {\em J. Differential Equations}, 261(1):455--504, 2016.

\bibitem[BGT16]{MR3564729}
Denis Bonheure, Christopher Grumiau, and Christophe Troestler.
\newblock Multiple radial positive solutions of semilinear elliptic problems
  with {N}eumann boundary conditions.
\newblock {\em Nonlinear Anal.}, 147:236--273, 2016.

\bibitem[BK79]{BrezisKato}
Ha\"{\i}m Br\'{e}zis and Tosio Kato.
\newblock Remarks on the {S}chr\"{o}dinger operator with singular complex
  potentials.
\newblock {\em J. Math. Pures Appl. (9)}, 58(2):137--151, 1979.

\bibitem[BKP91]{MR1103293}
C.~Budd, M.~C. Knaap, and L.~A. Peletier.
\newblock Asymptotic behavior of solutions of elliptic equations with critical
  exponents and {N}eumann boundary conditions.
\newblock {\em Proc. Roy. Soc. Edinburgh Sect. A}, 117(3-4):225--250, 1991.

\bibitem[BNW12]{MR2948289}
Denis Bonheure, Benedetta Noris, and Tobias Weth.
\newblock Increasing radial solutions for {N}eumann problems without growth
  restrictions.
\newblock {\em Ann. Inst. H. Poincar\'{e} C Anal. Non Lin\'{e}aire},
  29(4):573--588, 2012.

\bibitem[BST13]{MR3056708}
Denis Bonheure, Enrico Serra, and Paolo Tilli.
\newblock Radial positive solutions of elliptic systems with {N}eumann boundary
  conditions.
\newblock {\em J. Funct. Anal.}, 265(3):375--398, 2013.

\bibitem[CdPM20]{MR4046015}
Carmen Cort\'{a}zar, Manuel del Pino, and Monica Musso.
\newblock Green's function and infinite-time bubbling in the critical nonlinear
  heat equation.
\newblock {\em J. Eur. Math. Soc. (JEMS)}, 22(1):283--344, 2020.

\bibitem[CF20]{MR4149344}
Jean-Baptiste Casteras and Juraj F\"{o}ldes.
\newblock Singular radial solutions for the {L}in-{N}i-{T}akagi equation.
\newblock {\em Calc. Var. Partial Differential Equations}, 59(5):Paper No. 168,
  20, 2020.

\bibitem[CGMN19]{arxiv.1912.00721}
Charles Collot, Tej-Eddine Ghoul, Nader Masmoudi, and Van~Tien Nguyen.
\newblock Refined description and stability for singular solutions of the 2d
  keller-segel system.
\newblock {\em ar{X}iv.1912.00721}, 2019.

\bibitem[CGS89]{CaGiSp}
Luis~A. Caffarelli, Basilis Gidas, and Joel Spruck.
\newblock Asymptotic symmetry and local behavior of semilinear elliptic
  equations with critical {S}obolev growth.
\newblock {\em Comm. Pure Appl. Math.}, 42(3):271--297, 1989.

\bibitem[DdPD{\etalchar{+}}19]{arxiv.1911.12417}
Juan Davila, Manuel del Pino, Jean Dolbeault, Monica Musso, and Juncheng Wei.
\newblock Infinite time blow-up in the patlak-keller-segel system: existence
  and stability.
\newblock {\em arXiv:1911.12417}, 2019.

\bibitem[DdPW20]{MR4054257}
Juan D\'{a}vila, Manuel del Pino, and Juncheng Wei.
\newblock Singularity formation for the two-dimensional harmonic map flow into
  {$S^2$}.
\newblock {\em Invent. Math.}, 219(2):345--466, 2020.

\bibitem[DHR04]{DruetHebeyRobert}
Olivier Druet, Emmanuel Hebey, and Fr{\'e}d{\'e}ric Robert.
\newblock {\em Blow-up theory for elliptic {PDE}s in {R}iemannian geometry},
  volume~45 of {\em Mathematical Notes}.
\newblock Princeton University Press, Princeton, NJ, 2004.

\bibitem[dPMM14]{MR3262455}
Manuel del Pino, Fethi Mahmoudi, and Monica Musso.
\newblock Bubbling on boundary submanifolds for the {L}in-{N}i-{T}akagi problem
  at higher critical exponents.
\newblock {\em J. Eur. Math. Soc. (JEMS)}, 16(8):1687--1748, 2014.

\bibitem[dPMP05]{MR2114411}
Manuel del Pino, Monica Musso, and Angela Pistoia.
\newblock Super-critical boundary bubbling in a semilinear {N}eumann problem.
\newblock {\em Ann. Inst. H. Poincar\'{e} C Anal. Non Lin\'{e}aire},
  22(1):45--82, 2005.

\bibitem[dPMRW19]{MR3938021}
Manuel del Pino, Monica Musso, Carlos Rom\'{a}n, and Juncheng Wei.
\newblock Interior bubbling solutions for the critical {L}in-{N}i-{T}akagi
  problem in dimension 3.
\newblock {\em J. Anal. Math.}, 137(2):813--843, 2019.

\bibitem[dPMW21a]{MR4307216}
Manuel del Pino, Monica Musso, and Juncheng Wei.
\newblock Existence and stability of infinite time bubble towers in the energy
  critical heat equation.
\newblock {\em Anal. PDE}, 14(5):1557--1598, 2021.

\bibitem[dPMW21b]{MR4157676}
Manuel del Pino, Monica Musso, and Juncheng Wei.
\newblock Geometry driven type {II} higher dimensional blow-up for the critical
  heat equation.
\newblock {\em J. Funct. Anal.}, 280(1):Paper No. 108788, 49, 2021.

\bibitem[dPPV16]{MR3527634}
Manuel del Pino, Angela Pistoia, and Giusi Vaira.
\newblock Large mass boundary condensation patterns in the stationary
  {K}eller-{S}egel system.
\newblock {\em J. Differential Equations}, 261(6):3414--3462, 2016.

\bibitem[Dru02]{Druetdim3}
Olivier Druet.
\newblock Elliptic equations with critical {S}obolev exponents in dimension 3.
\newblock {\em Ann. Inst. H. Poincar\'{e} Anal. Non Lin\'{e}aire},
  19(2):125--142, 2002.

\bibitem[Dru03]{DruetJDG}
Olivier Druet.
\newblock From one bubble to several bubbles: the low-dimensional case.
\newblock {\em J. Differential Geom.}, 63(3):399--473, 2003.

\bibitem[DRW12]{DruetRobertWei}
Olivier Druet, Fr\'{e}d\'{e}ric Robert, and Juncheng Wei.
\newblock The {L}in-{N}i's problem for mean convex domains.
\newblock {\em Mem. Amer. Math. Soc.}, 218(1027):vi+105, 2012.

\bibitem[DS02]{DevillanovaSolimini}
Giuseppe Devillanova and Sergio Solimini.
\newblock Concentration estimates and multiple solutions to elliptic problems
  at critical growth.
\newblock {\em Adv. Differential Equations}, 7(10):1257--1280, 2002.

\bibitem[GMR]{GhoussoubMazumdarRobert}
Nassif Ghoussoub, Saikat Mazumdar, and Fr{\'e}d{\'e}ric Robert.
\newblock Multiplicity and stability of the pohozaev obstruction for
  hardy-schr{\"o}dinger equations with boundary singularity.
\newblock {\em Mem. Amer. Math. Soc.}
\newblock To appear.

\bibitem[Heb14]{HebeyZLAM}
Emmanuel Hebey.
\newblock {\em Compactness and stability for nonlinear elliptic equations}.
\newblock Zurich Lectures in Advanced Mathematics. European Mathematical
  Society (EMS), Z\"urich, 2014.

\bibitem[JL73]{MR340701}
D.~D. Joseph and T.~S. Lundgren.
\newblock Quasilinear {D}irichlet problems driven by positive sources.
\newblock {\em Arch. Rational Mech. Anal.}, 49:241--269, 1972/73.

\bibitem[JM20]{MR4108414}
Jacek Jendrej and Yvan Martel.
\newblock Construction of multi-bubble solutions for the energy-critical wave
  equation in dimension 5.
\newblock {\em J. Math. Pures Appl. (9)}, 139:317--355, 2020.

\bibitem[KL22]{KonigLaurain}
Tobias K{\"o}nig and Paul Laurain.
\newblock Multibubble blow-up analysis for the {B}rezis-{N}irenberg problem in
  three dimensions.
\newblock {\em arXiv:2208.12337}, 2022.

\bibitem[KS70]{KellerSegel}
Evelyn~F. Keller and Lee~A. Segel.
\newblock Initiation of slime mold aggregation viewed as an instability.
\newblock {\em J. Theoret. Biol.}, 26(3):399--415, 1970.

\bibitem[KST09]{MR2494455}
Joachim Krieger, Wilhelm Schlag, and Daniel Tataru.
\newblock Slow blow-up solutions for the {$H^1(\Bbb R^3)$} critical focusing
  semilinear wave equation.
\newblock {\em Duke Math. J.}, 147(1):1--53, 2009.

\bibitem[Li97]{Li1997}
YanYan Li.
\newblock On a singularly perturbed elliptic equation.
\newblock {\em Adv. Differential Equations}, 2(6):955--980, 1997.

\bibitem[Lio82]{MR683027}
Pierre-Louis Lions.
\newblock Sym\'{e}trie et compacit\'{e} dans les espaces de {S}obolev.
\newblock {\em J. Functional Analysis}, 49(3):315--334, 1982.

\bibitem[LN88]{MR974610}
Chang~Shou Lin and Wei-Ming Ni.
\newblock On the diffusion coefficient of a semilinear {N}eumann problem.
\newblock In {\em Calculus of variations and partial differential equations
  ({T}rento, 1986)}, volume 1340 of {\em Lecture Notes in Math.}, pages
  160--174. Springer, Berlin, 1988.

\bibitem[LNT88]{MR929196}
C.-S. Lin, W.-M. Ni, and I.~Takagi.
\newblock Large amplitude stationary solutions to a chemotaxis system.
\newblock {\em J. Differential Equations}, 72(1):1--27, 1988.

\bibitem[LW08]{MR2431658}
Fanghua Lin and Changyou Wang.
\newblock {\em The analysis of harmonic maps and their heat flows}.
\newblock World Scientific Publishing Co. Pte. Ltd., Hackensack, NJ, 2008.

\bibitem[LZ99]{LiZhu}
Yanyan Li and Meijun Zhu.
\newblock Yamabe type equations on three-dimensional {R}iemannian manifolds.
\newblock {\em Commun. Contemp. Math.}, 1(1):1--50, 1999.

\bibitem[Miy15]{MR3408072}
Yasuhito Miyamoto.
\newblock Structure of the positive radial solutions for the supercritical
  {N}eumann problem {$\varepsilon^2\Delta u-u+u^p=0$} in a ball.
\newblock {\em J. Math. Sci. Univ. Tokyo}, 22(3):685--739, 2015.

\bibitem[MM02]{MR1923818}
Andrea Malchiodi and Marcelo Montenegro.
\newblock Boundary concentration phenomena for a singularly perturbed elliptic
  problem.
\newblock {\em Comm. Pure Appl. Math.}, 55(12):1507--1568, 2002.

\bibitem[MM07]{MR2296306}
Fethi Mahmoudi and Andrea Malchiodi.
\newblock Concentration on minimal submanifolds for a singularly perturbed
  {N}eumann problem.
\newblock {\em Adv. Math.}, 209(2):460--525, 2007.

\bibitem[MNW05]{MR2124160}
A.~Malchiodi, Wei-Ming Ni, and Juncheng Wei.
\newblock Multiple clustered layer solutions for semilinear {N}eumann problems
  on a ball.
\newblock {\em Ann. Inst. H. Poincar\'{e} C Anal. Non Lin\'{e}aire},
  22(2):143--163, 2005.

\bibitem[MPV09]{MichelettiPistoiaVetois}
Anna~Maria Micheletti, Angela Pistoia, and J\'er\^ome V\'etois.
\newblock Blow-up solutions for asymptotically critical elliptic equations on
  {R}iemannian manifolds.
\newblock {\em Indiana Univ. Math. J.}, 58(4):1719--1746, 2009.

\bibitem[MPV17]{MorabitoPistoiaVaira}
Filippo Morabito, Angela Pistoia, and Giusi Vaira.
\newblock Towering phenomena for the {Y}amabe equation on symmetric manifolds.
\newblock {\em Potential Anal.}, 47(1):53--102, 2017.

\bibitem[Ni83]{Ni83}
Wei~Ming Ni.
\newblock On the positive radial solutions of some semilinear elliptic
  equations on {${\bf R}^{n}$}.
\newblock {\em Appl. Math. Optim.}, 9(4):373--380, 1983.

\bibitem[Oba72]{Obata}
Morio Obata.
\newblock The conjectures on conformal transformations of {R}iemannian
  manifolds.
\newblock {\em J. Differential Geometry}, 6:247--258, 1971/72.

\bibitem[Poh65]{Pohozaev}
S.~I. Poho\v{z}aev.
\newblock On the eigenfunctions of the equation {$\Delta u+\lambda f(u)=0$}.
\newblock {\em Dokl. Akad. Nauk SSSR}, 165:36--39, 1965.

\bibitem[Pre21]{Premoselli13}
Bruno Premoselli.
\newblock A priori estimates for finite-energy sign-changing solutions
  blowing-up solutions of critical elliptic equations.
\newblock {\em arXiv:2111.02470}, 2021.

\bibitem[Pre22]{Premoselli12}
Bruno Premoselli.
\newblock Towers of bubbles for {Y}amabe-type equations and for the
  {B}r\'{e}zis-{N}irenberg problem in dimensions {$n\geq 7$}.
\newblock {\em J. Geom. Anal.}, 32(3):Paper No. 73, 65, 2022.

\bibitem[PSnT22]{PST}
Angela Pistoia, Alberto Salda\~{n}a, and Hugo Tavares.
\newblock Existence of solutions to a slightly supercritical pure neumann
  problem.
\newblock {\em arXiv:2209.02113}, 2022.

\bibitem[PV13]{PistoiaVetois}
Angela Pistoia and J\'er\^ome V\'etois.
\newblock Sign-changing bubble towers for asymptotically critical elliptic
  equations on {R}iemannian manifolds.
\newblock {\em J. Differential Equations}, 254(11):4245--4278, 2013.

\bibitem[PV15]{MR3304582}
Angela Pistoia and Giusi Vaira.
\newblock Steady states with unbounded mass of the {K}eller-{S}egel system.
\newblock {\em Proc. Roy. Soc. Edinburgh Sect. A}, 145(1):203--222, 2015.

\bibitem[PV19]{PremoselliVetois}
Bruno Premoselli and J\'{e}r\^{o}me V\'{e}tois.
\newblock Compactness of sign-changing solutions to scalar curvature-type
  equations with bounded negative part.
\newblock {\em J. Differential Equations}, 266(11):7416--7458, 2019.

\bibitem[PV22a]{PremoselliVetois3}
Bruno Premoselli and J\'{e}r\^{o}me V\'{e}tois.
\newblock Sign-changing blow-up for the yamabe equation at the lowest energy
  level.
\newblock {\em arXiv:2206.08770}, 2022.

\bibitem[PV22b]{PremoselliVetois2}
Bruno Premoselli and J{\'e}r{\^o}me V{\'e}tois.
\newblock Stability and instability results for sign-changing solutions to
  second-order critical elliptic equations.
\newblock {\em Journal de Math{\'e}matiques Pures et Appliqu{\'e}es},
  167:257--293, 2022.

\bibitem[Roba]{RobertNeumann}
Fr{\'e}d{\'e}ric Robert.
\newblock Construction and asymptotics for the green's function with neumann
  boundary condition.
\newblock
  https://iecl.univ-lorraine.fr/files/2021/04/NotesGreenNeumannRobert.pdf.

\bibitem[Robb]{RobDirichlet}
Fr{\'e}d{\'e}ric Robert.
\newblock Existence et asymptotiques optimales des fonctions de {G}reen des
  op{\'e}rateurs elliptiques d'ordre deux.
\newblock
  https://iecl.univ-lorraine.fr/files/2021/04/NotesGreenNeumannRobert.pdf.

\bibitem[RV]{RobertVetois5}
Fr{\'e}d{\'e}ric Robert and J{\'e}r{\^o}me V{\'e}tois.
\newblock Blowing-up solutions for second-order critical elliptic equations:
  the impact of the scalar curvature.
\newblock {\em Int. Math. Res. Not. IMRN}.
\newblock To appear.

\bibitem[RV13]{RobertVetois3}
Fr{\'e}d{\'e}ric Robert and J{\'e}r{\^o}me V{\'e}tois.
\newblock Sign-changing blow-up for scalar curvature type equations.
\newblock {\em Comm. Partial Differential Equations}, 38(8):1437--1465, 2013.

\bibitem[RV14]{RobertVetois}
Fr{\'e}d{\'e}ric Robert and J{\'e}r{\^o}me V{\'e}tois.
\newblock A general theorem for the construction of blowing-up solutions to
  some elliptic nonlinear equations with {L}yapunov-{S}chmidt's
  finite-dimensional reduction.
\newblock {\em Concentration Compactness and Profile Decomposition (Bangalore,
  2011), Trends in Mathematics, Springer, Basel}, pages 85--116, 2014.

\bibitem[RV15]{RobertVetois4}
Fr{\'e}d{\'e}ric Robert and J{\'e}r{\^o}me V{\'e}tois.
\newblock Sign-changing solutions to elliptic second order equations: glueing a
  peak to a degenerate critical manifold.
\newblock {\em Calc. Var. Partial Differential Equations}, 54(1):693--716,
  2015.

\bibitem[RW04]{ReyWei}
Olivier Rey and Juncheng Wei.
\newblock Blowing up solutions for an elliptic {N}eumann problem with sub- or
  supercritical nonlinearity. {I}. {$N=3$}.
\newblock {\em J. Funct. Anal.}, 212(2):472--499, 2004.

\bibitem[SnT22]{MR4400170}
Alberto Salda\~{n}a and Hugo Tavares.
\newblock On the least-energy solutions of the pure {N}eumann {L}ane-{E}mden
  equation.
\newblock {\em NoDEA Nonlinear Differential Equations Appl.}, 29(3):Paper No.
  30, 36, 2022.

\bibitem[ST11]{MR2765510}
Enrico Serra and Paolo Tilli.
\newblock Monotonicity constraints and supercritical {N}eumann problems.
\newblock {\em Ann. Inst. H. Poincar\'{e} C Anal. Non Lin\'{e}aire},
  28(1):63--74, 2011.

\bibitem[Str77]{Strauss}
Walter~A. Strauss.
\newblock Existence of solitary waves in higher dimensions.
\newblock {\em Comm. Math. Phys.}, 55(2):149--162, 1977.

\bibitem[Str84]{Struwe}
Michael Struwe.
\newblock A global compactness result for elliptic boundary value problems
  involving limiting nonlinearities.
\newblock {\em Math. Z.}, 187(4):511--517, 1984.

\bibitem[SW19]{MR3936129}
Philippe Souplet and Michael Winkler.
\newblock Blow-up profiles for the parabolic-elliptic {K}eller-{S}egel system
  in dimensions {$n\geq 3$}.
\newblock {\em Comm. Math. Phys.}, 367(2):665--681, 2019.

\bibitem[Thi16]{ThizyLinNi}
Pierre-Damien Thizy.
\newblock The {L}in-{N}i conjecture in negative geometries.
\newblock {\em J. Differential Equations}, 260(4):3658--3690, 2016.

\bibitem[Tru68]{Trudinger}
Neil~S. Trudinger.
\newblock Remarks concerning the conformal deformation of {R}iemannian
  structures on compact manifolds.
\newblock {\em Ann. Scuola Norm. Sup. Pisa (3)}, 22:265--274, 1968.

\bibitem[Win20]{MR4179771}
Michael Winkler.
\newblock Blow-up profiles and life beyond blow-up in the fully parabolic
  {K}eller-{S}egel system.
\newblock {\em J. Anal. Math.}, 141(2):585--624, 2020.

\bibitem[WWY10]{MR2645043}
Liping Wang, Juncheng Wei, and Shusen Yan.
\newblock A {N}eumann problem with critical exponent in nonconvex domains and
  {L}in-{N}i's conjecture.
\newblock {\em Trans. Amer. Math. Soc.}, 362(9):4581--4615, 2010.

\bibitem[WWY11]{MR2806101}
Liping Wang, Juncheng Wei, and Shusen Yan.
\newblock On {L}in-{N}i's conjecture in convex domains.
\newblock {\em Proc. Lond. Math. Soc. (3)}, 102(6):1099--1126, 2011.

\end{thebibliography}

\end{document}